\documentclass[11pt]{amsart}
\usepackage{graphicx}
\usepackage{amssymb, amsmath, amsthm, amscd}
\usepackage{epstopdf}
\usepackage{textcomp, yfonts,mathrsfs,ulem, wasysym, pifont, bbm}

\def\mapright#1.{\buildrel #1 \over \longrightarrow}
\def\sc#1{{\mathcal #1}}
\def\cal#1{{\mathcal #1}}
\def \C  {\sc C}
\def \D {\sc D}
\def \Cf {{\sc C}_f}
\def \s  {\sc S}

\def \P {\sc P}
\def \N {N_{\bullet}}
\def\Fun  {\operatorname{Fun}}
\def\WR  {\operatorname{WR}}

\def\tfiber{\operatorname{tfiber}}
\def\hfib{\operatorname{hofiber}}
\def\hofib{\operatorname{hofiber}}
\def\ifiber{\operatorname{ifiber}}
\def\Fun{\operatorname{Fun}}
\def\WR{\operatorname{WR}}
\def\Hom{\operatorname{Hom}}
\def\HomFun{\operatorname{Hom}_{\Fun}}

\def\Cfn#1{{\sc C}_f^{\times #1}}
\def\holim{\operatorname{holim}}
\def\hocolim{\operatorname{hocolim}}
\def\sk{\mathrm{sk}}
\def\longrightarrow{\relbar\joinrel\rightarrow}
\def\longleftarrow{\leftarrow\joinrel\relbar}
\def\mapright#1.{\buildrel #1 \over \longrightarrow}
\def\mapleft#1.{\buildrel #1 \over \longleftarrow}
\def\mapup#1.{\Big\uparrow\rlap{$\vcenter{\hbox{$\scriptstyle#1$}}$}}
\def\mapdown#1.{\Big\downarrow\rlap{$\vcenter{\hbox{$\scriptstyle#1$}}$}}
\def\ox{\otimes}

\def\M{\mathcal{M}}
\def\x{\times}

\newcommand{\pt}{\bullet}
\newcommand{\ra}{\rightarrow}

\input xy
\xyoption{all}
\xyoption{frame}

\newtheorem{lem}{Lemma}[section]
\newtheorem{rem}[lem]{Remark}
\newtheorem{prop}[lem]{Proposition}
\newtheorem{thm}[lem]{Theorem}
\newtheorem{cor}[lem]{Corollary}

\newtheorem{defn}[lem]{Definition}
\newtheorem{ex}[lem]{Example}

\title{Cross effects and calculus in an unbased setting \\(with an appendix by Rosona Eldred)}
\author{Kristine Bauer }
\address{Department of Mathematics \& Statistics, University of Calgary}
\email{kristine@math.ucalgary.ca}

\author{Brenda Johnson}
\address{Department of Mathematics, Union College}
\email{johnsonb@union.edu}

\author{Randy McCarthy}
\address{Department of Mathematics, University of Illinois at Urbana-Champaign}
\email{randy@math.uiuc.edu}

\address{Department of Mathematics, Universit\"at  Hamburg}
\email{rosona.eldred@math.uni-hamburg.de}

\begin{document}

\begin{abstract}
We study functors $F:\Cf\rightarrow \mathcal {D}$ where $\C$ and $\mathcal D$ are simplicial model categories and $\Cf$ is the category consisting of objects that factor a fixed morphism $f:A\rightarrow B$ in $\C$.   We define the analogs of Eilenberg and Mac Lane's cross effect functors in this context, and identify explicit adjoint pairs of functors whose associated cotriples are the  diagonals of the cross effects.   With this, we generalize the cotriple Taylor tower construction of \cite{RB} from the setting of functors from pointed categories to abelian categories to that of functors from $\Cf$ to $\s$, a suitable category of spectra, to produce a 
tower of functors $\dots \rightarrow \Gamma_{n+1}F\rightarrow \Gamma_nF\rightarrow \Gamma_{n-1}F\rightarrow \dots \rightarrow F(B)$ whose $n$th term is a degree $n$ functor.   We compare this tower to Goodwillie's tower, $\dots\rightarrow P_{n+1}F\rightarrow P_nF\rightarrow P_{n-1}F\rightarrow \dots \rightarrow F(B)$, of $n$-excisive approximations to $F$ found in \cite{G3}.   When $F$ is a functor that commutes with realizations, the towers agree.  More generally, for functors that do not commute with realizations, we show
that the terms of the towers agree when evaluated at the initial object of $\Cf$.  
\end{abstract}

\maketitle

\section{Introduction}
Tom Goodwillie's calculus of homotopy functors is a technique for studying homotopy functors of spaces and spectra (\cite{G2}, \cite{G3}).  It provides a means by which a homotopy functor can be approximated by an $n$-excisive functor in a manner analogous to the degree $n$ Taylor polynomial approximation of a real-valued function.  Because of this analogy,  the sequence of approximating functors, $P_1F, P_2F, \dots, P_nF, \dots$, associated to a functor $F$ by Goodwillie's method is referred to as  the {\it Taylor tower} of $F$.   In the decades since its initial development, Goodwillie's theory has been 
further developed and applied by many other mathematicians.   

\pagestyle{myheadings} \markboth{ {K. BAUER, B. JOHNSON, and R. MCCARTHY}}{CROSS EFFECTS AND CALCULUS IN AN UNBASED SETTING}

In an effort to apply the calculus of functors to a more algebraic setting and to better understand the combinatorics underlying Goodwillie's constructions, the second and third  authors of this paper developed a model for a Taylor tower for functors of abelian categories based on a particular collection of cotriples arising from Eilenberg and Mac Lane's cross effect functors (\cite{RB}).  For a functor $F$, the terms in the resulting sequence of approximations, $\{\Gamma_nF\},$ behave in a slightly different way than the $n$-excisive approximations provided by Goodwillie.  Goodwillie's functors $P_nF$ satisfy a  higher-order excision property, while the polynomial functors $\Gamma_nF$ satisfy a kind of higher additivity property.  The cotriple model for calculus has the advantage that the difference between a homotopy functor $F$ and its polynomial approximation $\Gamma_nF$ can be modeled by cotriple homology, which is well understood.  Furthermore,  in good situations the cotriple method recovers information about Goodwillie's functors.  In particular, if $\C$ is the category of based topological spaces, and $F:\C \to {\mathcal C}$ is a homotopy functor taking values in connected spaces that commutes with geometric realization, Andrew Mauer-Oats generalized the cotriple method and showed that $P_nF\simeq \Gamma_nF$ (\cite{Mauer}, \cite{AMO-thesis}).

The cotriple method as established in \cite{RB}  is limited;  it only applies to functors $F:\C \to {\mathcal D}$ where $\C$ is a pointed  category (a category with an object that is both initial and final) with finite coproducts and $\mathcal D$ is an abelian category.   The generalization of the cotriple method to the topological setting in \cite{Mauer} and \cite{AMO-thesis} is similarly limited as it applies to functors of {\it based} spaces.  On the other hand, Goodwillie's construction can  be used for functors whose  source categories  are not  pointed and whose target categories are not abelian,  in particular, functors from the category $Top$ of unbased topological spaces or $Top/Y$ of topological spaces over a fixed space $Y$ to categories of spaces or spectra.

  In \cite{RB}, the essential cotriples are obtained by identifying  adjoint pairs of functors for which the right adjoint is a cross effect functor.    Goodwillie (\cite{G3})  identifies a similar adjunction {\it up to homotopy} in the topological setting of (not necessarily basepointed) spaces and spectra.    Working with  basepointed spaces, Mauer-Oats (\cite{Mauer}) shows directly that diagonals of cross effect functors form the cotriples in which we are interested, but does not identify the adjoint pairs from which these cotriples arise.    This suggests that there should be some 
adjoint pairs of functors that generate the cotriples in the  topological setting, at least when the objects are basepointed.    A key result in the present paper is to show that this is true for fairly general model categories, even in the unpointed case.   As an application of the cotriples that one obtains from these strict adjoint pairs, we generalize the construction of the Taylor tower in \cite{RB} and obtain  analogous results, including  a variation of Mauer-Oats' result that relates the terms in the cotriple Taylor tower to those of Goodwillie's.    We summarize the main results of the paper below.

We work with functors $F:\Cf \rightarrow \D$ where $\C$ and $\D$ are simplicial model categories 
and $\Cf $ is the category that consists of objects $A\rightarrow X\rightarrow B$ factoring a fixed morphism $f:A\rightarrow B$ in $\C$.    In \cite{RB}, we used an adjoint pair involving 
the $n$th cross effect functor to define a cotriple $\perp _{n}$ on the category of functors from a pointed category with finite coproducts to an abelian category.  The cotriple $\perp_n$  yielded the $(n-1)$st term in our Taylor tower.   
The main difficulty in reconstructing the cotriple $\perp_{n}$ in the category of functors from $\Cf$ to $\D$ is that the pair of functors used in \cite{RB} is no longer an adjoint pair, but instead only gives us an  adjunction up to homotopy.  We resolve this issue by factoring through a category of coalgebras associated to a certain cotriple $t$ to obtain a pair of adjunctions whose composition produces the adjunction we need.   This gives us the following result.   The functor $\perp _n$ is the diagonal of the $n$th cross effect.  
\vskip .1 in 
\noindent{\bf Theorem 3.8, Theorem 3.14, Theorem 3.17.}  {\it For each $n\geq 1$,  there is a cotriple $t$ on the category of functors of $n$ variables from $\Cf$ to $\D$, and  an adjoint pair of functors  $(U^+, t^+)$ between this  category and the category of $t$-coalgebras, where the forgetful functor $U^+$ is the left adjoint.   There is a second adjoint pair of functors  $(\Delta ^*, \sqcup _n)$ between the category of functors of $n$ variables from $\Cf$ to $\D$ and the category of functors of a single variable from $\Cf$ to $\D$, with the diagonal functor $\Delta^*$ as the left adjoint.
The composition yields the  adjoint pair $(\Delta ^*\circ U^+, t^+\circ \sqcup _n)$ whose associated 
cotriple is $\perp_n$, defined on the category of functors from $\Cf$ to $\D$.}  

\vskip .1 in 

We can use the cotriples of Theorem 3.17 as the basis for constructing terms in a  Taylor tower for $F:\Cf \rightarrow  \s$, where $\s$ is a suitable category of spectra.   However, the $n$th term in this tower, $\Gamma_nF$,  is not an $n$-excisive 
functor as Goodwillie constructs, but instead a degree $n$ functor.    A functor is degree $n$ if its $(n+1)$st cross effect vanishes, whereas a functor is $n$-excisive if it takes strongly cocartesian $(n+1)$-cubical diagrams of objects (i.e., diagrams whose square faces are all homotopy pushouts) to homotopy pullback diagrams.  We compare the notions of $n$-excisive and degree $n$, proving that being degree $n$ is a weaker condition that can yield
$n$-excisive behavior in certain circumstances. When $F$ commutes with realizations, we prove that the notions of degree $n$ and $n$-excisive coincide, and that the functors $P_nF$ and $\Gamma_nF$ 
agree.   In particular, we have the following results.

\vskip .1 in 
\noindent {\bf Proposition \ref {p:realizations}.}  {\it If  $F:\Cf\rightarrow \s$ commutes with realizations, then $F$ is degree $n$ if and only if $F$ is $n$-excisive.}

\vskip .1 in

We use Proposition \ref{p:realizations} to obtain a  Mauer-Oats style result,  showing that there is a fibration sequence of functors involving 
$P_nF$   and $\perp_{n+1}^{*+1}F$, the simplicial object associated to the cotriple $\perp_{n+1}$ and functor $F$.
\vskip .1 in 

\noindent{\bf Theorem \ref{t:commute}.}\ {\it  Let $F:\Cf \rightarrow \s$ be a functor that commutes with realizations.   Then there is a (co)fibration sequence of functors
\[
|\perp _{n+1}^{*+1}F|\rightarrow F\rightarrow P_nF.
\]}
\vskip .1 in 

As a consequence of Theorem \ref{t:commute}, we obtain 
\vskip .1 in 

\noindent {\bf Corollary 6.8.}  {\it Let $F: \Cf\rightarrow \s$ be a functor that commutes with realizations.   Then $P_nF$ and 
$\Gamma_nF$ are weakly equivalent as functors from $\Cf $ to $\s$.  
}
\vskip .1 in 

When $F$ does not commute with realizations, the role of the initial object, $A$,   in $\Cf$ becomes more critical in comparing the notions of degree $n$ and $n$-excisive, and in comparing $\Gamma_nF$ and $P_nF$.   
We say that a functor is $n$-excisive relative to $A$ if it behaves like an $n$-excisive functor on strongly cocartesian $(n+1)$-cubical diagrams whose initial objects are $A$, and prove 
\vskip .1 in 
\noindent{\bf Proposition \ref{p:degreen}.}  {\it Let $F$ be a functor from $\Cf$ to $\s$.  Let $n\geq 1$ be an integer.  The functor $F$ is degree $n$  if and only if $F$ is $n$-excisive relative to $A$.} 
\vskip .1 in 

We also prove that the results of Theorem 6.5 and Corollary 6.8 apply  when the functors are evaluated at $A$.   In particular, we have
\vskip .1 in 
\noindent {\bf Theorem \ref{t:agreeatA}.}  {\it Let $F:\Cf \rightarrow \s$ where $\Cf$ is the category of objects factoring the morphism $f:A\rightarrow B$.   Then $\Gamma _nF(A)\simeq P_nF(A)$. }  

\vskip .1 in 

An important realization is that Theorem \ref{t:agreeatA} can be rephrased to show that for any $X$, the $n$th term of Goodwillie's tower can be recovered from the $n$th term in {\it some} cotriple Taylor tower, even though the towers do not agree as functors.  To do so, we change our focus to the category of objects over a fixed terminal object $B$.  This focus on the terminal object is exactly the same as the setting in \cite{G3}.  Let ${\C_{/B}}$ be the category of objects in $\C$ over $B$, and let $F$ be a functor from ${\C_{/B}}$ to spectra.  Given any $\beta:X\to B$ in ${\C_{/B}}$, we have a weak equivalence of spectra
\[ P_nF(X) \simeq \Gamma_n^{\beta}F(X)\]
where $\{\Gamma_n^{\beta}F\}$ is the cotriple Taylor tower obtained by restricting $F$ to the category $\C_\beta$.  

%
%
A key step in proving Theorem \ref{t:agreeatA} is the observation below.   The functor $T_nF$ is the first stage in the sequence of functors that Goodwillie uses to construct $P_nF$.  
\vskip .1 in 
\noindent {\bf Lemma \ref{l:Tn}.} {\it
Let $F:\Cf \rightarrow \s$.   Then
\[\perp _{n+1}F(A)\rightarrow F(A)\rightarrow T_nF(A)
\]
is a fibration sequence in $\s$.}
\vskip .1 in 
As an appendix, we include a generalization of this result due to Rosona Eldred:
\vskip .1 in 
\noindent {\bf Proposition \ref{prop:alldim}.}    {\it For a functor $F:\Cf \rightarrow \s$,  
$|\sk_k( \perp^{*+1}_{n+1} F)(A)| \rightarrow F(A) \rightarrow T_n^{k+1} F(A)$ is a homotopy fiber sequence where $\sk _k$ denotes the $k$-skeleton of the simplicial object $\perp^{*+1}_{n+1}F$.}  

\vskip .1 in

The paper is organized as follows.   
In section 2 we take care of preliminaries:  we define the types of categories in which we will be working, describe the models for and properties of homotopy limits and colimits that we use, and review some basic notions associated to $n$-cubical diagrams of objects in our categories.   
In section 3 we define cross effects for  functors from $\Cf$ to $\D$.   We also identify the composition of adjoint pairs that yields $\perp_n$ as a cotriple.  
In section 4, notions of degree $n$ and $n$-excisive are compared via the intermediate concept of $n$-excisive relative to $A$.  In section 5, the cotriple Taylor tower is defined and various properties are verified for it.  This leads 
to a comparison in section 6 of the cotriple Taylor tower in this context with Goodwillie's tower.
\vskip .1 in 
\noindent {\bf  Acknowledgments:}   The authors were able to meet and work together several 
times during the writing of this paper because of the generosity and hospitality of the following:  the Midwest Topology Network (funded by NSF grant DMS-0844249), the Union College Faculty Research Fund, the Pacific Institute for the Mathematical Sciences, and the mathematics department of the  University of Illinois at Urbana-Champaign.   We thank them for their support.     We thank Rosona Eldred, Agn\`es Beaudry, Mona Merling, and Sarah Yeakel for their assistance in confirming the homotopy limit properties of Lemma \ref{l:holim}.  We also 
thank Tom Goodwillie for the body of work that inspired this paper, and for the understanding 
of the calculus of functors that he has imparted to us over the years.  Finally, we thank the anonymous referee for a very careful and thoughtful review.   The referee's comments and corrections improved this paper substantially.   In particular, the referee identified a critical error in an earlier version of section 3, and his/her suggestion that we prove $t$ is a cotriple directly and use the category of $t$-coalgebras was crucial in resolving this problem.  
\section{Prerequisites }

In this section we describe the context in which we will be working, and  review
some essential concepts that will be used throughout this paper.   The section is divided
into three parts.   The first describes the categories with which we work and provides a summary
of some properties of model categories, simplicial model categories, and categories of simplicial objects  that we need.   The second  covers
necessary facts about homotopy limits and colimits.   The third  discusses $n$-cubical 
diagrams.   

\subsection{The setting.}
We work with functors from $\C$ to $\D$ where $\C$ and $\D$ are suitable model categories.    
By suitable, we mean that ${\sc C}$ and $\D$ should be  simplicial model categories, that $\C$ has a functorial cofibrant replacement functor and that $\D$ has a functorial fibrant replacement functor. For many results, we will also require that $\D$ (but not $\C$) be pointed, i.e., that it has an object that is both initial and final.  Recall that a model category comes equipped with distinguished classes of morphisms -- weak equivalences, cofibrations, and fibrations -- satisfying the  standard  axioms (as found on pp. 1.1-1.2 of \cite{Quillen}, or in several expository accounts, such as Definition 1.3 of \cite{Goerss}). Requiring  that a category $\mathcal D$ be  a simplicial model category  gives us  the following extra structure:
 \begin{itemize}
 \item for every simplicial set $K$ and object $X$ of $\mathcal D$  there is an object $X\otimes K$ in ${\mathcal D}$; and
\item for every simplicial set $K$ and object $Y$ of $\mathcal D$ there is  an exponential object $Y^K$ in $\mathcal D$ defined by the adjunction formula
\[ {\rm hom}_{\mathcal D}(X\otimes K, Y) \cong {\rm hom}_{\mathcal D}(X, Y^K).\]
\item for each pair of objects $X$ and $Y$, there is a simplicial set of morphisms in $\mathcal D$, ${\rm Hom}_{\mathcal D}(X,Y)$, satisfying an additional axiom (see pp. 1.1, 1.2, and 2.2 of \cite{Quillen}).  
\end{itemize}  
For much of this paper we  focus on subcategories of $\C$ determined by morphisms in $\C$.   In particular, for a morphism $f:A\rightarrow B$ in $\C$, the category ${\sc C}_f$ is the category whose objects are pairs of morphisms in $\C$ of the form $A\to X\to B$  that provide a factorization of $f:A\to B$.   We will usually denote  objects of $\Cf$ simply by the object $X$ through which $f$ factors. 
 A morphism in $\Cf$ is a commuting diagram:
\[ \xymatrix{ & X \ar[d]^{g}\ar[dr] \\
A\ar[ur]\ar[r] &Y\ar[r]& B}\]
which we will denote as $g:X\to Y$ when the context is clear.   
 The category ${ \sc C}_f$ has  initial object $A=A\to B$ (the first map is the identity and the second map is $f$) and terminal object $A\to B=B$ (the first map is $f$ and the second map is
 the identity).   

The category $\Cf$ inherits 
structure from $\C$.  More specifically,
it is a simplicial  model category whenever $\C$ is   (\cite{Quillen}, \S II.2, Proposition 6). 
The map $g $ in ${\sc C}_f$ is a weak equivalence if the underlying map $g:X\to Y$ is a weak equivalence in ${\sc C}$,
  a cofibration in $\Cf$ if the underlying map in ${\sc C}$ is a cofibration, and a fibration if $g$ is a fibration in $\C$. 
  For convenience, we assume from the outset that all objects of $\Cf$ in this paper are cofibrant.  That is, 
 we assume that an object $X$ of $\Cf$ is a factorization $A\rightarrowtail X \to B$ of $f$ where the map $A\rightarrowtail B$ is  a cofibration in $\C$.
  Since we assume all objects are cofibrant, we abuse notation and simply denote the category of cofibrant objects by $\Cf$. 
    Limits and colimits  in $\Cf$ are also inherited from $\C$, i.e.,  they can be computed in the underlying category $\C$.

We will occasionally pass to the category $s\Cf$ of simplicial objects in $\Cf$.   When doing so, 
we extend the model category structure of $\Cf$ to $s\Cf$ using the Reedy model structure.    Recall that to do so,  one uses a Quillen pair.   The following definition dates to \cite{Quillen} and can be found in many modern references, e.g. \cite{Goerss}.

\begin{defn}  Let $\C$ and $\sc D$ be model categories and 
\[ \xymatrix{F: \C \ar@<.5ex>[r] & \sc{D}:G  \ar@<.5ex>[l] } \]
be an adjoint pair of functors, with $F:\C \to \sc{D}$ the left adjoint.  Then $F$ and $G$ are called a Quillen pair (or Quillen functor) if
\begin{itemize}
\item $F$ preserves cofibrations and weak equivalences between cofibrant objects and
\item $G$ preserves fibrations and weak equivalences between fibrant objects.
\end{itemize}
\end{defn}

We would like to produce an adjoint pair of functors between $\Cf$ and $s\Cf$ which becomes a Quillen pair when we put the correct model structure on $s\Cf$.   The left adjoint of the (potential) Quillen pair is the geometric realization functor.  Recall that $\Delta$ is the category whose objects are ordered sets $[n]=\{0, 1, 2, \dots, n\}$ for $n\geq 0$ and morphisms are order-preserving set maps.  For $m\geq 0$, the standard $m$-simplex is 
$\Delta ^m=\operatorname{hom}_{\Delta}(-,[m])$.   If $\C$ (and hence $\Cf$ ) is a simplicial model category and $X.$ is a simplicial object over $\Cf$, then the object $\Delta^m \otimes X_n$ is well-defined for each $m, n\geq 0$.  
The geometric realization of $X.$ in $s\Cf$, denoted $|X.|$, is the coequalizer of 
\[ \xymatrix{\underset{[m]\to [n]}{\coprod} \Delta^m\otimes X_n \ar@<.5ex>[r] \ar@<-.5ex>[r] &\underset{n}{ \coprod} \Delta^n\otimes X_n }\]
where the first coproduct runs over all possible morphisms from $[m]$ to $[n]$ in ${\Delta}$ and the two arrows correspond to evaluation on $\Delta^m$ and $X_n$, respectively.  Its right adjoint is the singular simplicial set functor. The singular simplicial set of  an object $Y$ is the simplicial object ${Y.}^{\Delta}$ defined by $(Y^{\Delta})_n = Y^{\Delta^n}$, with face and degeneracy maps induced by the ones in $\Delta$.

Placing the Reedy model category structure on $s\Cf$  guarantees that the geometric realization and the singular simplicial set functors are a Quillen pair.   The cofibrations and fibrations for this structure can be readily described  via  latching and matching objects.  The $n$th latching object consists of the degenerate simplices in the degree $n$ part of a simplicial object.  If $X_.\in s\Cf$, then 
\[ L_nX. := (\sk_{n-1}X)_n = \underset {\phi:[n]\twoheadrightarrow [m]}{\rm colim} \phi^* X_m\]
where the colimit is taken over all surjections from $[n]$ in $\Delta$.  Note that the degeneracy maps of $X.$ provide a (natural) map $L_nX \to X_n$.  The $n$th matching object is defined similarly, 
\[ M_nX := \underset{\phi:[m]\hookrightarrow [n]}{\rm lim} \phi^* X_m\]
where  the limit is now taken over injections.  There is a natural map $X_n \to M_nX$ that comes
from the face maps of $X.$.   The next theorem describes the Reedy structure and establishes that it gives us the desired model category structure.

\begin{thm} (\cite{Reedy})  There is a model category structure on $s\Cf$ where a morphism $g:X.\to Y.$ is
\begin{itemize}
\item a weak equivalence if $X_n\to Y_n$ is a weak equivalence in $\Cf$ for all $n\geq 0$;
\item a cofibration if the natural morphism
\[ X_n +_{L_nX} L_n Y:=\operatorname{colim}(X_n\leftarrow L_nX\rightarrow L_nY) \to Y_n\]
is a cofibration in $\Cf$ for all $n\geq 0$; and
\item a fibration if the natural morphism
\[ X_n \to Y_n \times_{M_nY} M_nX:=\operatorname{lim}(M_nX\rightarrow M_nY\leftarrow Y_n) \]
is a fibration in $\Cf$ for all $n\geq 0$.
\end{itemize}
With this model category structure, the geometric realization and the singular simplicial set functors form a Quillen pair.  
\end{thm}

\subsection{Homotopy limits and colimits.}  

We use homotopy limits and colimits to describe certain desirable properties of
our functors.  Homotopy limits and colimits can be defined abstractly as total derived functors or, depending on the category, concretely in terms of specific models.   There are several  models for homotopy limits and colimits, typically involving a simplicial construction coming from the nerve of the underlying diagram category $I$.  We describe the 
two particular models that we use in this paper and the properties of these models that we will need.   For more details, we refer readers to   \cite{Shulman}, which provides a good expository account of homotopy limits and colimits in model categories, \cite{G2},  which establishes many of the properties we use for topological spaces and spectra, \cite{H} for more details about the model we use for homotopy limit  or \cite{BK},  the classical reference.    The homotopy colimit model is that of \cite{Shulman} whereas the homotopy limit model is essentially the one described in \cite{BK}, \cite{G2}, and  \cite{H}.

To construct homotopy colimits, we use the generalized bar construction of \cite{Shulman}.
Let ${\sc X}:{I}\to \Cf$ be a functor.  Then 
\[ \operatorname{hocolim}_{I} {\mathcal X} \simeq | B_{\bullet} ({I}, {\mathcal X}) |\]
where the right hand side is the geometric realization of the simplicial object  $B_{\bullet} ({I}, {\mathcal X})$  with 
\[B_n({I}, {\mathcal X}) = \coprod_{i_0 \to \dots \to i_n} {\mathcal X}(i_0).\]
The coproduct is indexed by the $n$-simplices of $N_{\bullet}({I})$, the nerve of $I$, and the face and degeneracy maps are given by those in $N_{\bullet}({I})$.  

The following properties of $B_{\bullet}({I}, X)$ are used in section 4.
\begin{lem}{\label{l:Bar}}  If ${\mathcal X}$ is objectwise cofibrant, then 
\begin{enumerate}
\item $B_\bullet({I}, {\mathcal X})$ is cofibrant under the Reedy model structure on $s.{\Cf}$. 
\item for any full subcategory ${I}'$ of ${I}$, the  induced map
\[ B_\bullet( {I}', {\mathcal X}) \to B_\bullet( {I}, {\mathcal X})\]
is a cofibration in the Reedy model structure.   
\item if  $N_\bullet({I}) = N_\bullet({I'}) \cup N_\bullet ({I}'')$, then
\[\xymatrix{ B_\bullet ( {I}'\cap {I}'', {\mathcal X}) \ar[r] \ar[d] & B_\bullet ( {I}', {\mathcal X}) \ar[d] \\ B_\bullet( {I}'', {\mathcal X}) \ar[r] & B_\bullet ( {I}, {\mathcal X})}\]
is a pushout diagram.
\end{enumerate}
\end{lem}
\begin{proof} The first statement (1) is Lemma 9.2 of \cite{Shulman}.  

To prove (2), we must show that the map  \[ j:B_n( {I}', {\mathcal X})\coprod_{L_nB_\bullet ( {I}', {\mathcal X})} L_nB_\bullet( {I}, {\mathcal X})\to B_n( {I}, {\mathcal X})\] is a cofibration in $\Cf$.   Note that $L_nB_\bullet(I, {\mathcal X})$ is the coproduct
\[ \coprod_{i_0\to \cdots \to i_n} {\mathcal X}(i_0)\]
indexed by chains of maps $i_0\to \cdots \to i_n$ in $I$ for which some $i_k\to i_{k+1}$ is the identity map.  The structure maps $L_nB_\bullet(I, {\mathcal X})\to  B_n( {I}, {\mathcal X})$ and $ B_n( {I}', {\mathcal X})\to  B_n( {I}, {\mathcal X})$ are induced by inclusion maps:  the first is the inclusion of the degenerate elements while the second is induced by the forgetful functor from $I'$ to $I$.
These inclusions mean that the map $j$ is constructed from maps that are cofibrations.  That is, each summand ${\mathcal X(i_0)}$ in \[ B_n( {I}', {\mathcal X})\coprod_{L_nB_\bullet ( {I}', {\mathcal X})} L_nB_\bullet( {I}, {\mathcal X})\] is indexed by a chain of maps $i_0\to \cdots \to i_n$ either in $I'$, which is a subcategory of $I$, or else in $I$ itself (coming from a degenerate chain of maps in $I$), or both.  
Thus, ${\mathcal X}(i_0)$ also represents a summand of $B_n(I, {\mathcal X})$.  To complete the construction of $j$, we take the coproduct of identity maps, one for each summand ${\mathcal X}(i_0)$ of  $B_n( {I}', {\mathcal X})\coprod_{L_nB_\bullet ( {I}', {\mathcal X})} L_nB_\bullet( {I}, {\mathcal X})$ together with the coproduct over the initial object $A$ of maps $A\to {\mathcal X}(i_0)$ for each summand ${\mathcal X}(i_0)$ of $B_n(I, {\mathcal X})$ indexed by a chain of maps $i_0\to \cdots \to i_n$ which is neither degenerate nor contained in the subcategory $I'$.   Since ${\mathcal X}(i_0)$ is cofibrant, each of the maps $A\to {\mathcal X}(i_0)$ is a cofibration and the identity map is always a cofibration.   The coproduct of cofibrations is again a cofibration so it follows that $j$ is a cofibration.


For (3), let $F(i_0)$ be a simplex of $B_\bullet ( {I}, {\mathcal X})$ indexed by $i_0\to \cdots \to i_n$.  Since $i_0\to \cdots \to i_n$ is a simplex of $N_\bullet({I})$, the hypothesis implies that $i_0\to \cdots \to i_n$ is a chain of morphisms in either ${I}'$ or ${I}''$.  Thus the simplex $F(i_0)$ came from one in $B_\bullet ( {I}', {\mathcal X})$ or $B_\bullet ( {I}'', {\mathcal X})$.  Hence the diagram of (3) is a pushout.  
\end{proof}

\bigskip
\noindent
{\bf Convention.} Given a simplicial model category $\M$, the simplicial realization from the category of simplicial objects in $\M$ to $\M$ preserves weak equivalences between cofibrant simplicial objects. 
If $\M$ has functorial cofibrant replacements, then one can define $\hocolim_{\Delta^{op}}$ as a weak-equivalence-preserving functor from simplicial objects in $\M$ to $\M$ by composing the cofibrant replacement with a model for $\hocolim$ such as the one in [10, 18.1.2]. 
The functor 
$\text{hocolim}_{\Delta^{op}}$ is sometimes called the ``fat'' realization, and written $\|X\| = \text{hocolim}_{\Delta^{op}}X$. There is a natural transformation $\|X\|\rightarrow |X|$ which is a weak equivalence when $X$ is a cofibrant simplicial object in $\M$. We will be using the fat realization throughout. 
We will follow the convention (which is a slight abuse of notation) and simply write $|X|$ for the fat realization (as was done in [16]). 
We note that by Lemma 2.3(1), when $\cal X$ is objectwise cofibrant, we can apply the fat realization to the diagram in (3) and obtain a cocartesian diagram.

\bigskip

The model for homotopy colimits described above can be dualized to produce a model for homotopy limits.   Instead, we construct homotopy limits in the following fashion.   
\begin{defn}  Let $I$ be a small category and $F:I\rightarrow \C$ be an $I$-diagram in $\C$.  
For each object $i$ in our indexing category $I$, let ${ I\downarrow i}$ denote the category of elements in $I$ over $i$;  this category has objects $j\to i$ and morphisms given by commuting triangles.  Let  $N_{\bullet}(I\downarrow i)$ be the nerve of the category $I\downarrow i$.  Then 
\begin{equation}\label{e:holim} \operatorname{holim}_IF = \operatorname{hom}^I(N_{\bullet}(I\downarrow -), F(-))\end{equation}
where for a functor $G$ from $I$ to $s.set$ (simplicial sets) and a functor $H:I\to \C$, the construction  $\operatorname{hom}^I(G, H)$  is the equalizer of the two obvious maps
\[ \xymatrix{\underset{i\in I}{\prod}  H(i)^{G(i)} \ar@<.5ex>[r] \ar@<-.5ex>[r] &\underset{i\to j}{\prod} H(j)^{G(i)}.}\]
\end{defn}
We make use of the following properties of homotopy limits.   The first four are essential for  the key results in section 3.   We include their proofs in Appendix A.  The last property follows easily from the definition of homotopy limit above.  

\begin{lem}{\label{l:holim}} Let $I$ and $J$ be small categories.
\begin{enumerate}

\item If $F:I\times J\to {\sc C}$, then \[\operatorname{holim}_{I\times J}F \cong \operatorname{holim}_I(\operatorname{holim}_J F).\]
\item If $\alpha:J\to I$  and $F:I\to {\C}$, then there is a morphism $\operatorname{holim}_I F \to \operatorname{holim}_J F\circ \alpha$.  
\item If $T$ is a constant $I$-diagram with $T(C)=T$ and $T(f)={\text id}_T$ where $T$ is a terminal object of $\C$, then 
\[\holim_I T\cong T.  
\]
\item If ${\mathscr I}$ is the trivial diagram on $i$, then for any ${\mathscr I}$-diagram $X$, 
\[
\holim _{\mathscr I}X\cong X(i). 
\]
\item If $F,G:I\to {\sc C}$  and $\eta:F\to G$ is a natural transformation (also called a map of $I$-diagrams), then $\eta$ induces a map from $\operatorname{holim}_IF \to \operatorname{holim}_IG$.

\end{enumerate}
\end{lem}

As noted on page 379 of \cite{H}, the definition of homotopy limit above is homotopy invariant only when the diagram is objectwise fibrant.    When the $I$-diagram $F$ is objectwise fibrant, $\holim_IF$ is fibrant by Corollary 18.5.2(2) of \cite{H}.    For this reason, we will use functors that take values in fibrant objects.

The first property of Lemma  \ref{l:holim} is often referred to by the slogan ``homotopy limits commute,'' since it also implies that $\operatorname{holim}_I \operatorname{holim_J} = \operatorname{holim_J}\operatorname{holim_I}$.  A special case of this property tells us that 
homotopy fibers and homotopy limits commute, where homotopy fibers are defined as follows.
\begin{defn}Let $\sc{D}$ be a pointed model category with  initial/final object  $\star$, and let $g:X\to Y$ be a morphism in $\sc{D}$.  Then the homotopy fiber of $g$, denoted $\operatorname{hofiber} g$, is the homotopy limit of 
\[ \xymatrix{ & {\star} \ar[d] \\ X \ar[r]^g & Y.}\]
\end{defn}
Note that in this paper, all homotopy fibers are computed in the target category $\sc{D}$.

\subsection{Cubical diagrams.}  Later in this paper, we examine two fundamental concepts, degree $n$ and $n$-excisive, each of  which is used to define a notion of degree $n$ polynomial functor.  Both concepts are determined by the behavior of a functor when applied to certain types of diagrams in $\Cf$ .  

\begin{defn} Let ${\bf n} = \{ 1, 2, \ldots , n\}$ and let $\P({\bf n})$ be the power set of ${\bf n}$ treated as a category whose objects are the subsets of ${\bf n}$ and morphisms are the set inclusions.   An $n$-cubical diagram (or $n$-cube) in a category $\sc D$ is a functor  from  $\P({\bf n})$ to $\sc D$.  
\end{defn}

One can picture an  $n$-cubical diagram as being shaped like a cube of dimension $n$.  For this reason, we say that the object $\chi(S)\in \C$ for any fixed $S\subset {\bf n}$ is a {\it vertex} of the $n$-cube $\chi$.  Similarly, the image of the inclusion $S\subset S\cup \{i\}$ ($1\leq i\leq n$, $i\notin S$) under $\chi$ is called an {\it edge}, and for $i,j\notin S$, the image of 
\[ \xymatrix{S\ar[r]\ar[d]&S\cup \{i\}\ar[d]\\S\cup\{j\}\ar[r]&S\cup\{i,j\}}
\]
under $\chi$ is a $2$-{\it face}.
  To a functor of $n$ variables from ${\Cf}$, we associate two special $n$-cubical diagrams.  

\begin{ex} \label{e:chi} In the category $\Cf$, every object $X$ is equipped with a map $\beta_X$ to the terminal object $B$.  Let $H:\Cf^{\times n} \to {\sc D}$ be a functor of $n$ variables from $\Cf$ to an arbitrary category $\sc D$.  For an $n$-tuple of objects ${\bf X}=(X_1, \ldots , X_n)$ in $\Cf$, the $n$-cube $H^{\bf X}_B$ in $\sc D$ is defined by
\[ H^{\bf X}_B(S) = H(X^1_B(S), \ldots , X^n_B(S))\]
where 
\[ X_B^i(S)= \begin{cases} X_i & \text{if\ \ } i\notin S \\ B & \text{if \ \ } i \in S.\end{cases}\]  The image of the inclusion map $S\subset T$ under $H^{\bf X}_B$ is induced by the maps $\beta_{X_i}$.

We will make use of an $n$-cube of this type obtained by using the functor $\sqcup^n:{\Cfn n}\to \Cf$ with
\[ \sqcup^n(X_1, \ldots , X_n) = X_1\coprod_A \ldots \coprod_A X_n .\]
In particular, the $2$-cube $(\sqcup^2)^{(X,Y)}_B$ is the square diagram
\[ \xymatrix{ X\underset{A}{\coprod} Y \ar[r]\ar[d] & B\underset{A}{\coprod}Y \ar[d] \\ X\underset{A}{\coprod} B \ar[r] & B\underset{A}{\coprod} B.}\]
\end{ex}

We will also make use of an $n$-cubical diagram that exploits the fact that every object
in ${\Cf}$ is equipped with a map from $A$.   Recall that we  assume that the map $A\to X$ is a cofibration. 
\begin{ex} \label{e:Upsilon} For an object $X$ in  $\Cf$, let $\alpha_X:A\to X$ denote the cofibration from the initial object to $X$.  Let $H:\Cf^{\times n} \to {\sc D}$ be a functor of $n$ variables from $\Cf$ to an arbitrary category $\sc D$.  Then for an $n$-tuple of objects ${\bf X}=(X_1, \ldots , X_n)$ in $\Cf$, the $n$-cube  $H^A_{\bf X}: P({\bf n})\rightarrow \sc D$ is defined by
\[ H^A_{\bf X}(S) = H(X_1^A(S), \ldots , X_n^A(S))\]
where \[X^A_i(S)=\begin{cases} A &\text {if\ \ } i\notin S\\ X_i &\text{if\ \ } i\in S.\end{cases}\] The image of the inclusion map $S\subset T$ under $H^{A}_{\bf X}$ is induced by the maps $\alpha_{X_i}$.

Again, the functor $\sqcup^n$ produces useful examples.  The 
$2$-cube $(\sqcup^2)^A_{(X,Y)}$ is the square diagram
\[ \xymatrix{A\underset{A}{\coprod} A \ar[r]
\ar[d]
& A\underset{A}{\coprod} Y \ar[d]
\\ 
X\underset{A}{\coprod} A \ar[r]
& X\underset{A}{\coprod} Y.}\]
Since $A\coprod_A A = A$, $A\coprod_A Y = Y$ and $X\coprod_A A = X$, the diagram is   the diagram which defines the coproduct in $\Cf$.  In particular, it is homotopy cocartesian, as defined below.
\end{ex}

More generally, we are interested in $n$-cubes that are pullbacks or pushouts up to homotopy.
In particular, we use the notions of  homotopy cartesian and cocartesian diagrams introduced 
in \cite{G2}.
To define these terms, we let $\P_0({\bf n})$ be the full subcategory of $\P({\bf n})$ determined by the 
non-empty subsets of ${\bf n}$ and $\P_1({\bf n})$ be the full subcategory of $\P({\bf n})$ determined by 
the subsets other than ${\bf n}$ itself.  
\begin{defn} \label{d:cartesian}Let $\chi $ be an $n$-cubical diagram in a model category ${\sc D}$.  \begin{itemize}

\item There are natural maps from the initial vertex $\chi ({\emptyset})$ to $ {\rm holim}_{S\in \P_0({\bf n})}\chi (S)$ and ${\rm hocolim}_{S\in \P_1({\bf n})} \chi(S)$ to the terminal vertex $ \chi ({\bf n})$ determined by the compositions
\[\chi(\emptyset)=\operatorname{lim}(\chi)\rightarrow \holim_{\P({\bf n})}\chi\rightarrow \holim_{\P_0({\bf n})}\chi,
\]
and
\[
\hocolim_{\P_1({\bf n})}\chi\rightarrow \hocolim_{\P({\bf n})}\chi\rightarrow \operatorname{colim}_{\P({\bf n})}\chi=\chi({\bf n}),
\]
respectively.

\item   We say that $\chi$  is {homotopy cartesian} if the map from the initial vertex $\chi ({\emptyset})$ to $ {\rm holim}_{S\in \P_0({\bf n})}\chi (S)$ is a weak equivalence. 
\item   We say that $\chi$ is {homotopy cocartesian} if the map from\\  ${\rm hocolim}_{S\in \P_1({\bf n})} \chi(S)$ to the terminal vertex $ \chi ({\bf n})$ is a weak equivalence.
\item    We say that $\chi$ is {strongly homotopy cocartesian} if each of its $2$-faces is homotopy cocartesian.  
\end{itemize}
\end{defn} 
Following \cite{G2} we generally omit the term ``homotopy" when speaking of these types of diagrams.

\section{Cross effects}\label{s:xfx}

The cross effects for functors of abelian categories were introduced by Eilenberg and Mac Lane in \cite{EM}.   Given a functor $F$ between two abelian categories and a positive integer $n$,
Eilenberg and Mac Lane defined a functor of $n$ variables, $cr_{n}F$, that measures in some 
sense the extent to which $F$ fails to be additive.    Drawing on their ideas, the second and third
authors of this paper used the cross effects to define the degree of a functor and construct degree $n$ polynomial approximations to functors from a pointed category 
to an abelian category in \cite{RB}.  The construction of the polynomial approximations depended on showing that 
the cross effect functors were parts of adjoint pairs and as such could be used to produce cotriples and 
cotriple resolutions that readily yielded the desired approximations.   

In the present work,
we extend these ideas to functors whose domain category is not pointed and whose  target is not necessarily abelian.
While some of the results of \cite{RB} carry through to this 
new context quite easily, others do not.   In particular, identifying the adjoint pair that yields
the desired cotriple requires a different approach.
We use this section to adapt cross effects and the notion of the degree of a functor to a setting where the domain is  of the form $\Cf$ that we introduced  in Section 2.1. 
 We also identify  adjoint 
pairs and cotriples associated to cross effects that we need.
    
Throughout this section we work with  functors from the category $\Cf$  to  the target category $\D$ where $\C_f$ and $\D$ are both simplicial model categories as described in Section 2.1.  In our constructions, we need to use the fact that the target category (but not the domain category) is pointed.  Thus we assume that $\D$ is pointed and denote the initial/final object by $\star$.  To ensure that our homotopy limit constructions behave nicely with respect to weak equivalences, we further assume that all functors take fibrant values in $\D$.  

The cross effect functors will be  functors of functors.  For this reason, we will often need to consider the ``category" of functors from one category to another.  Strictly speaking, we can not do so since these categories rarely have sets of morphisms (which are defined by natural transformations).  In practice, this can often be resolved.  The functors from $\C$ to $\D$ will form a category if $\C$ is skeletally-small  or if we are careful to fix a suitable universe of sets in which to work (as in \cite{G3}).  For the remainder of this paper, we assume that we are in a situation in which such categories of functors make sense.   We use $\Fun(\Cf, \D)$ to denote the category of functors from $\Cf$ to $\D$ that preserve weak equivalences, and $\Fun(\Cfn n, \D)$ to denote the category of functors of $n$ variables from $\Cf$ to $\D$ that preserve weak equivalences.

\subsection{Iterated Fibers and Cross Effects}

Our first step is to define  cross effects for our context.   The definition of the $n$th cross effect functor involves the iterated fibers of $n$-cubical diagrams associated with the $n$-fold coproduct functor $\sqcup^n$.   To better understand the definition, consider a commuting square of objects in $\D$:
\begin{equation}\label{e:1} \xymatrix{A\ar[r]\ar[d]&B\ar[d]\\
C\ar[r]&D.}
\end{equation}
If we take homotopy fibers vertically, we obtain a map of homotopy fibers:
\[
\hofib\left ( \vcenter{\xymatrix{A\ar[d]\\C}}\right )\rightarrow \hofib\left(\vcenter{\xymatrix{B\ar[d]\\D}}    \right).
\]
We can take  the homotopy fiber of this map to obtain an object, $X$, in $\D$, that we call the {\it iterated fiber} of the diagram:
\[
X=\hfib\left(\hfib \left (\vcenter{\xymatrix{A\ar[d]\\C}}\right)\rightarrow \hfib \left (\vcenter{\xymatrix{B\ar[d]\\D}}\right)    \right).
\] 
Recall that the homotopy fiber of a map such as $A\rightarrow C$ is defined to be the homotopy  limit of the diagram
\[
A\rightarrow C\leftarrow \star.
\]
Using this, we see that $X$ is the homotopy  limit of 
\[
\holim \left (\vcenter{\xymatrix{A\ar[d]\\C\\\star\ar[u]}}\right)\rightarrow \holim \left (\vcenter{\xymatrix{B\ar[d]\\D\\ \star\ar[u]}}\right) \leftarrow \star.  
\]
Using properties (1) and (3) of Lemma \ref{l:holim}, we see that $X$ is the homotopy  limit of 
\begin{equation}\label{e:2}
\xymatrix{A\ar[r]\ar[d]&B\ar[d]&\star\ar[l]\ar[d]\\
C\ar[r]&D&\star\ar[l]\\
\star\ar[u]\ar[r]&\star\ar[u]&\star .\ar[l]\ar[u]}
\end{equation}
This is also isomorphic to 
\[
\hfib\left(\vcenter{ \xymatrix{\hfib(A\rightarrow B)\ar[d]\\ \hfib(C\rightarrow D)} }\right)
\]
and, as a consequence, we see that 
\[\hfib\left(\vcenter{ \xymatrix{\hfib(A\rightarrow B)\ar[d]\\ \hfib(C\rightarrow D)} }\right)\cong \hfib\left(\hfib \left (\vcenter{\xymatrix{A\ar[d]\\C}}\right)\rightarrow \hfib \left (\vcenter{\xymatrix{B\ar[d]\\D}}\right)    \right).
  \]  
  In other words, the order of the directions in which we take fibers does not matter when defining the iterated fiber of $X$.   We generalize this to define the iterated fiber of an $n$-cube.  To do so we define a new diagram associated to an $n$-cubical diagram.
 \begin{defn}\label{d:assoccube}
 Let ${\sc X}$ be an $n$-cubical diagram in $\D$.  The associated $({\P}_0({\bf 2}))^{\times n}$-diagram ${\sc X}^*$ assigns to the $n$-tuple of sets $(S_1, \dots, S_n)$ in $\P_0({\bf 2})$ the object 
 \[{\sc X}^*(S_1,\dots, S_n)=\begin{cases}  \star &\text{if\ } S_i=\{2\} \text{\ for\  at\  least\  one\  $i$,}\\
 {\sc X}(\{i\ |\ S_i=\{1,2\}\})  &\text{otherwise}.   
  \end{cases}
 \]
 ${\sc X}^*$ takes  a morphism of sets in $(\P_0({\bf 2}))^{\times n}$, $\phi: (S_1, \dots, S_n)\rightarrow (T_1, \dots, T_n)$, to the map ${\sc X}^*(\phi)$ defined as follows:
 \[
 {\sc X}^*(\phi)=\begin{cases} {\sc X} (\{i\ |\ S_i=\{1,2\}\}\subseteq \{j\ |\ T_j=\{1,2\}\})&S_i, T_j\neq\{2\}, 1\leq i,j\leq n,\\
 {\sc X}^*(S_1, S_2, \dots, S_n)\rightarrow \star&T_j=\{2\}\ \text {for\ some\ }j,\\
 \star\rightarrow {\sc X}^*(T_1, \dots, T_n)&S_i=\{2\}\text{\ for\ some\ }i,\\
 \end{cases}
 \]
 where the second and third maps are uniquely determined by the fact that $\star$ is the initial/final object in $\D$.  
 \end{defn}
 We note that if ${\sc X}$ is the diagram in (\ref{e:1}), then ${\sc X}^*$ is the diagram in (\ref{e:2}).
\begin{defn}
Let ${\sc X}$ be an $n$-cubical diagram in $\D$.  The iterated (homotopy) fiber of ${\sc X}$, denoted $\ifiber{\sc X}$, is given by 
\[\ifiber{\sc X}=\underset{(\P_0({\bf 2}))^{\times n}}{\holim}{\sc X^*}.
\]
\end{defn}
As demonstrated in the case of a $2$-cube, one can think of the iterated fiber of an $n$-cube as the object constructed by first taking fibers in one direction, $U\rightarrow U\cup \{i\}$, in the $n$-cube, then taking fibers of the resulting fibers in another direction, and continuing until one has exhausted all independent directions in the cube.  
As an immediate consequence we have the following.  
\begin{lem}\label{l:iteratedfiber}  
Let ${\sc X}$ be an $n$-cube and let ${\sc X}_1$ and ${\sc X}_2$ be  $(n-1)$-cubes with ${\sc X}={\sc X}_1\rightarrow {\sc X}_2$.   Then 
\[
\ifiber{\sc X}\cong\hofib\left (\ifiber({\sc X}_1)\rightarrow \ifiber({\sc X}_2)\right).
\]
\end{lem}
We use iterated fibers  to define the $n$th cross effect of a functor.
\begin{defn}\label{d:crn}   Let $F: \Cf \rightarrow \D$ be a functor.   The $n$th cross effect of $F$ is the functor of $n$ variables $cr_nF:\Cfn n\rightarrow \D$ that for an $n$-tuple ${\bf X}=(X_1, X_2,\dots, X_n)$ is the iterated fiber of the $n$-cube $F((\sqcup^n)^{\bf X}_B)$ that results from applying the functor $F$ to the $n$-cubical diagram $(\sqcup^n) ^{\bf X}_B$ (as defined in Example \ref{e:chi}).   When we precompose $cr_n$ with the diagonal functor $\Delta:X\mapsto (X, X, \dots, X)$, the result is a functor from $\Cf$ to $\D$.   We use $\perp_n$ to denote this composition, that is, $\perp _nF(X)=cr_nF(X, X, \dots, X)$.  
  \end{defn}
  When $\Cf$ is pointed and $\D$ is abelian, these definitions agree with those of \cite{RB}.  

  \begin{ex}
  For $n=2$, $cr_2F(X_1, X_2)$ is the iterated fiber of
  \[\xymatrix{F(X_1\underset{A}{\amalg}X_2)\ar[r]\ar[d]&F(B\underset{A}{\amalg}X_2)\ar[d]\\
 F(X_1\underset{A}{\amalg}B)\ar[r]&F(B\underset{A}{\amalg}B),
  }
  \]
  and $\perp_2F(X)$ is the iterated fiber of
   \[\xymatrix{F(X\underset{A}{\amalg}X)\ar[r]\ar[d]&F(B\underset{A}{\amalg}X)\ar[d]\\
 F(X\underset{A}{\amalg}B)\ar[r]&F(B\underset{A}{\amalg}B).
  }
  \]
    \end{ex}
  As discussed at the beginning of this section, the cross effect functors play an essential role in the construction of Taylor towers in \cite{RB}.  When ${\sc C}$ is a pointed category and ${\sc A}$ is an abelian category, it is straightforward
to show that $(\Delta^*, cr_n)$ is an adjoint pair of functors.   (Here $\Delta ^*$ denotes 
precomposition with the diagonal functor.)   As a consequence, $\perp_n=\Delta^*\circ cr_n$ 
is a cotriple -- this was used to construct the $(n-1)st$ term in the Taylor tower of a functor $F$.  

In trying to replicate this process for functors from $\Cf$ to $\D$, one encounters an obstruction
almost immediately.   The functors $\Delta ^*$ and $cr_n$ no longer form a strict adjoint pair.   Goodwillie has shown that they form an adjoint pair up to weak equivalence in a topological setting \cite{G3}, but something more is need to show that $\perp_n$ is a cotriple.   Our solution is to recognize that $\perp_n$ arises naturally from a pair of adjunctions.  The first involves the functor $t$ defined as follows.     

\begin{defn}
  For a functor $H:\Cfn n\rightarrow \D$, the functor $tH:\Cfn n\rightarrow \D$ is defined for an $n$-tuple ${\bf X}=(X_1, X_2, \dots, X_n)$ of objects in $\Cf$ by
  \[tH(X_1, \dots, X_n)=\ifiber(H^{\bf X}_B).
  \]
  \end{defn}
  
  \begin{rem}\label{e:nfld}
  Note that for $H={\sqcup}^n$, the $n$-fold coproduct functor, a functor $F:\Cf\rightarrow \D$ and an $n$-tuple $(X_1, \dots, X_n)$ of objects in $\Cf$, 
  \[cr_nF(X_1, \dots, X_n)=t(F\circ\sqcup^n)(X_1, \dots, X_n).  \]
  \end{rem}

In the next two subsections, we show that $t$ is a cotriple on $\Fun(\Cfn n, \D)$ and identify an associated adjunction.   The other adjoint pair involves $\sqcup^n$ and the diagonal functor.

\subsection{$t$ is a cotriple}

Recall that a cotriple on a category ${\sc A}$ consists of a functor $\perp:{\sc A}\rightarrow {\sc A}$ together with natural transformations $\epsilon: \perp\rightarrow {\rm id}_{\sc A}$ and $\delta: \perp\rightarrow \perp\perp$ such that the following diagrams commute:
  \[
  \xymatrix{\perp \ar[r]^{\delta}\ar[d]_{\delta}&\perp\perp\ar[d]^{\delta_{\perp}}\\
  \perp\perp\ar[r]_{\perp\delta}&\perp\perp\perp}\ \ \ \ \ \ \ \ 
   \xymatrix{&\perp\ar[dl]_{=}\ar[d]_{\delta}\ar[dr]^{=}&\\
  \perp&\perp\perp \ar[l]^{\perp\epsilon}\ar[r]_{\epsilon_{\perp}}&\perp.}
  \] 

Our first goal is to prove the following.
  \begin{thm}\label{t:tiscotriple}
  There are natural transformations $\gamma: t\rightarrow {\rm id}_{\Fun(\Cfn n,\D)}$ and $+: t\rightarrow tt$ such that $(t,+, \gamma)$ is a cotriple on $\Fun(\Cfn n,\D)$.  
  \end{thm}
  To prove this theorem we begin by  defining the natural transformations $\gamma: t\rightarrow {\rm id}_{\Fun(\Cfn n, \D)}$ and $+: t\rightarrow tt$.   Both of these natural transformations will be determined by applying property (2) of Lemma \ref{l:holim} to  maps of the indexing sets used in the homotopy inverse limits that define $tH$ and $ttH$.  We begin with $\gamma$.
  
  \begin{defn}  Let $H:\Cfn n\rightarrow \D$ and ${\bf X}=(X_1, \dots, X_n)$ be an object in $\Cfn n$.
  Consider the inclusion, $g: \{(\{1\},\{1\},\dots,\{1\})\}\rightarrow (\P_0({\bf 2}))^{\times n}$.   By properties (2) and (4) of Lemma  \ref{l:holim}, this induces a natural transformation 
  \[
\xymatrix{\underset {(\P_0({\bf 2}))^{\times n}}{\holim} (H^{\bf X}_B)^*\ar[r]&   \underset {\{(\{1\},\{1\},\dots,\{1\})\}}{\holim}((H^{\bf X}_B)^*\circ g)\cong H(X_1, \dots, X_n) }
  \]
  which is natural in both $H$ and ${\bf X}$.  This gives us the natural transformation $\gamma:t\rightarrow {\rm id}_{\Fun(\Cfn n,\D)}$.
  \end{defn}
  
  The definition of the second natural transformation requires understanding the two-fold iteration of $t$ as the iterated fiber of a single cube.   Note that there is an isomorphism $ \P({\bf n})\times \P({\bf n})\cong \P({\bf 2n})$ realized for example by sending $(S,T)\in \P(\{1,2,\dots, n\})\times \P(\{n+1, \dots, 2n \})$ to $S\cup T.$   Given this, we can treat $2n$-cubes as $(\P({\bf n})\times \P({\bf n}))$-diagrams.
  
  \begin{lem}\label{p:ttcube}
  Let $H:\Cfn n\rightarrow \D$ and ${\bf X}=(X_1, X_2, \dots, X_n)$ be an object in $\Cfn n$.  Then $ttH(X_1, X_2, \dots, X_n)$ is isomorphic to the iterated fiber of the $\P({\bf n})\times \P({\bf n})$-diagram defined by 
  \[(S,T)\in \P({\bf n})\times \P({\bf n})\mapsto H^{\bf X}_B(S\cup T).
  \]
  \end{lem}
  
  \begin{proof}
  By definition, $ttH(X_1, \dots, X_n)$ is the iterated fiber of the $n$-cubical diagram 
  \[
  S\mapsto (tH)^{\bf X}_B(S).
  \]
  But  for each $S\subseteq {\bf n}$, $(tH)^{\bf X}_B(S)$ is itself the iterated fiber of the $n$-cubical diagram
  \[
  T\mapsto (H)^{{\bf X}(S)}_B(T),
  \]
  where ${\bf X}(S)=(X^1(S), \dots, X^n(S))$ is as defined in Example \ref{e:chi}.  
 As observed in the discussion preceding  Definition \ref{d:assoccube},    homotopy fibers commute isomorphically in $\D$.    From this we see that $ttH({\bf X})$ is the iterated fiber of the $\P({\bf n})\times \P({\bf n})$-diagram 
  \[(S,T)\mapsto H^{{\bf X}(S)}_B(T).
  \]   It is easy to check that $H^{{\bf X}(S)}_B(T)=H^{\bf X}_B(S\cup T)$ and the result follows from this.   
  \end{proof}

\begin{rem}
In some instances, it will be more convenient to treat the $\P({\bf n})\times \P({\bf n})$-diagram 
  $(S,T)\in \P({\bf n})\times \P({\bf n})\mapsto H^{\bf X}_B(S\cup T)$
of the previous lemma as the ${\bf 2n}$-cube $\widetilde H^{\bf X}_B$ given by 
\[
U\subseteq {\bf 2n}\mapsto H(M_1(U), M_2(U), \dots, M_n(U))
\]
where 
\[
M_i(U)=\begin{cases} X_i  &\text{\ if \ }\{i,n+i\}\cap U=\emptyset,\\
B&\text{\ if\ }\{i,n+i\}\cap U\neq\emptyset.
\end{cases}
\]
\end{rem}

\begin{defn}\label{d:+}
Let $+: (\P_0({\bf 2}))^{\times 2n}\rightarrow (\P_0({\bf 2}))^{\times n}$ be the map  
\[
(S_1, S_2, \dots, S_n, T_1, \dots, T_n)\mapsto (V_1, V_2, \dots, V_n)
\]
where 
\[
V_i=\begin{cases} S_i\cup T_i&\text {\ if\ }S_i, T_i\neq \{2\}, \\
\{2\}&{\ if\ }S_i=\{2\}\text{\ or\ }T_i=\{2\}.
\end{cases}
\]
\end{defn}

By Lemma \ref{l:holim}.2, the set map $+:(\P_0({\bf 2}))^{\times 2n}\rightarrow (\P_0({\bf 2}))^{\times n}$ induces a map
\begin{equation}\label{e:Delta1}
tH({\bf X})=\underset{(\P_0({\bf 2}))^{\times n}}{\holim} (H^{\bf X}_B)^*\rightarrow\underset {(\P_0({\bf 2}))^{\times 2n}}{\holim}((H^{\bf X}_B)^*\circ +).
\end{equation}
But, it is straightforward to show that $(\widetilde H^{\bf X}_B)^*=(H^{\bf X}_B)^*\circ +.$  Hence, (\ref{e:Delta1}) gives us a natural transformation 
\begin{equation}\label{e:Delta}
+_H: tH\rightarrow ttH.
\end{equation}

With these definitions, we prove  Theorem \ref{t:tiscotriple}.

\begin{proof}
We begin by  showing  that the diagram
\begin{equation}\label{e:+diag}\xymatrix{t \ar[r]^{+}\ar[d]_{+}&tt\ar[d]^{+_{t}}\\
  tt\ar[r]_{t+}&ttt}
  \end{equation}
  commutes.  Let $H:\Cfn n\rightarrow \D$ and ${\bf X}=(X_1, \dots, X_n)$ be an $n$-tuple of objects in $\Cf$. 
Note that, as is the case with $ttH({\bf X})$, $tttH({\bf X})$ can be realized as the iterated fiber of a ${\bf 3n}$-cube.  In particular,  tttH({\bf X}) is the iterated fiber of the ${\bf 3n}$-cube that assigns to the set 
$U\subseteq {\bf 3n}$ the object
\[
H(J_1(U), \dots, J_n(U))
\]
with
\[
J_i(U)=
\begin{cases} 
X_i&\text{\ if\ }\{i, n+i, 2n+i\}\cap U=\emptyset, \\
B&{\text\ otherwise.}
\end{cases}
\]
Hence, $tttH({\bf X})$ can be treated as the holim of a $(\P_0({\bf 2}))^{\times 3n}$-diagram.   From this point of view, we see that  (\ref{e:+diag}) commutes by noting that 
the maps
are induced by the commuting diagram of set maps
\[
\xymatrix{\P_0({\bf 2})^{\times n}&\P_0({\bf 2})^{\times 2n}\ar[l]_+\\
\P_0({\bf 2})^{\times 2n}\ar[u]^+&\P_0({\bf 2})^{\times 3n}.\ar[l]^{{\rm id}\times +}\ar[u]_{+\times {\rm id}}}
\]

To see that 
\[
\xymatrix{&t\ar[dl]_{=}\ar[d]_{+}\ar[dr]^{=}&\\
  t&tt \ar[l]^{t\gamma}\ar[r]_{\gamma_{t}}&t.}
\]
commutes,  we note that $t\gamma$ and $\gamma_t$ are induced at the indexing set level by the maps 
\[
\iota_1: (\P_0({\bf 2}))^{\times n}\rightarrow (\P_0({\bf 2}))^{\times 2n}, \iota_1:(S_1, \dots, S_n)\mapsto (S_1, \dots, S_n, \{1\}, \dots, \{1\}),
\]
and 
\[
\iota_2: (\P_0({\bf 2}))^{\times n}\rightarrow (\P_0({\bf 2}))^{\times 2n}, \iota_2:(S_1, \dots, S_n)\mapsto (\{1\}, \dots, \{1\}, S_1, \dots, S_n),
\]
respectively.   It is straightforward to check that $+\circ\iota_1$ and $+\circ \iota_2$ are the identity map on $(\P_0({\bf 2}))^{\times n}$.   From this it follows that $t\gamma\circ +$ and $\gamma_t\circ+$ are the identity on $t$.  
\end{proof}

\subsection{$\perp_n$ is a cotriple}

To establish that $\perp_n$ is a cotriple, we describe an adjunction determined by the cotriple $t$.    
Categories equipped with cotriples have a related category of coalgebras, related to the original category by a  forgetful-cofree adjunction.

\begin{defn} (\cite{Mac}, Definition VI.2, dualized) If $(\perp, \delta, \epsilon)$ is a cotriple on a category ${\sc B}$, then the category ${\sc B}_\perp$ of $\perp$-coalgebras is the category whose objects are pairs $(B, \beta)$, where $B\in Ob({\sc B})$ and $\beta:B\to \perp B$, which satisfy
\[ \perp \beta\circ \beta= \delta_B\circ \beta\quad \text{and} \quad \epsilon_B\circ \beta=Id_B.\]
A morphism $f:(B, \beta)\to (B', \beta')$ in  ${\sc B}_\perp$ is a morphism $f:B\to B'$ in ${\sc B}$ such that $\perp f \circ \beta = \beta'\circ f$.
\end{defn}

Thus, the category of $t$-coalgebras, ${\Fun}({\Cfn n}, \D)_t$, consists of functors $G:{\Cfn n}\to \D$ that are equipped with a section $\beta:G\to tG$ to the natural transformation $\gamma_G:tG\to G$ which also  makes the diagram
\[ \xymatrix{ G \ar[r]^{\beta} \ar[d]^{\beta} & tG\ar[d]^{t\beta} \\ tG\ar[r]_{+_G} & ttG}\]
commute.  For example, for any functor $G\in{\Fun}({\Cfn n}, \D)$, there is an associated $t$-coalgebra $(tG, +_G)$.     Let $t^+: {\Fun}({\Cfn n}, \D)\to  {\Fun}({\Cfn n}, \D)_t$ be the free coalgebra functor, which is defined on objects by $t^+(G)=(tG, +_G)$.  

\begin{thm}  The functors
\[  \xymatrix{ {\Fun}({\Cfn n}, \D) \ar@/^/[r]^{t^+} & {\Fun}({\Cfn n}, \D)_t\ar@/^/[l]^{U^+}}\]
are an adjoint pair of functors, with the forgetful functor $U^+$ being the left adjoint.
\end{thm}

\begin{proof} The proof follows immediately from Theorem \ref{t:tiscotriple}, since $(t, +, \gamma)$ forms a cotriple.  The adjunction in question is the  forgetful-cofree adjunction which exists for any category of coalgebras over a cotriple.  The proof of this fact is formally dual to the proof of Theorem VI.2.1 found in \cite{Mac} for algebras over triples.
\end{proof}

We now turn our attention to the second adjoint pair of functors that we use to establish that $\perp _n$ is a cotriple.  
\begin{defn}  \label{d:sqcup}Let 
\[\Delta^*: \Fun({\Cfn n}, \D) \to \Fun({\Cf}, \D)\]
be the functor defined for a functor $H:\Cfn n\rightarrow \D$  by $\Delta^*H(X) = H(X, \ldots, X)$.  Let 
\[ \sqcup _n: \Fun({\Cf}, \D) \to \Fun({\Cfn n}, \D) \]
be the functor defined by precomposition with the functor $\sqcup^n$ of Examples \ref{e:chi} and \ref{e:Upsilon}.  That is, for a functor $F$  
\[ \sqcup _n(F)(X_1, \ldots , X_n) = (F\circ \sqcup^n) (X_1, \ldots , X_n) = F(X_1 \amalg_A \ldots \amalg_A X_n).\] 
\end{defn}

\begin{prop} The functors $\Delta^*$ and $\sqcup _n$ are an adjoint pair of functors, with $\Delta^*$ being the left adjoint.

\end{prop}

\begin{proof}  Let $H:\Cfn n\rightarrow \D$ and $F:\Cf \rightarrow \D$.
We must prove that there are isomorphisms
\[\xymatrix{ \hom_{\Fun(\Cf,\D)}(\Delta^*H, F)\ar@/^/[r]^\Phi & \hom_{\Fun(\Cfn n,\D)}(H, \sqcup _n F)\ar@/^/[l]^{\Psi} }. \]
We do so by taking advantage of some coproduct properties.  

In $\Cf$, the coproduct of $X$ and $Y$ is the pushout of 
$$
Y\longleftarrow A\longrightarrow X
$$
which we denote $X\coprod_A Y$.   When $X=Y$, the fact that $X\coprod_AX$ is a pushout 
means that we have a fold map $+: X\coprod _A X\rightarrow X$ that serves as a section to 
the inclusion $X\rightarrow X\coprod_A X$ into either term of the coproduct.   Iterating this gives a fold map $+:\coprod _n X\rightarrow X$ that for each $1\leq k\leq n$ is a section to 
$\iota _k:X\rightarrow \coprod _nX$, inclusion into the $k$th term in the coproduct.

The map $\Phi$ sends a  natural transformation $\sigma: \Delta^*H\to F$ to the natural transformation $\Phi\sigma := \sqcup _n\sigma\circ H(i_1, \ldots , i_n)$.
On the other hand, a natural transformation $\tau: H\to \sqcup _nF$ is sent to $\Psi\tau:=F(+)\circ \tau$.

To see that $\Psi\Phi \sigma$ is equal to $\sigma$, consider the diagram:
\[ \xymatrix{   H(X, \ldots , X) \ar[d]^{H(i_1, \ldots , i_n)} \ar@{=}[dr]\\
H(\underset{i}{\coprod} X, \ldots , \underset{i}{\coprod} X) \ar[d]^{\sigma_{\coprod X}} \ar[r]_{H(+)} &H(X, \ldots , X) \ar[d]^{\sigma_X} \\
F(\underset{i}{\coprod} X) \ar[r]^{F(+)}& F(X) }\]
The bottom rectangle commutes by the naturality of $\sigma$, and the top triangle commutes because $+$ is a section to each $i_k$.  Going along the bottom and left edges of the diagram gives $\Psi\Phi\sigma$, while the right hand edge is just $\sigma$.

On the other hand, the diagram
\[ \xymatrix{ H(X_1, \ldots, X_n)\ar[r]^{H(i_1, \ldots, i_n)}\ar[d]_{\tau_{X_1, \ldots , X_n}} & \quad H(\underset{i}{\coprod} X_i, \ldots , \underset{i}{\coprod} X_i)\ar[d]^{\tau_{\coprod X_i, \ldots \coprod X_i}} \\
F(\underset{i}{\coprod} X_i)\ar@{=}[dr]\ar[r]^{\sqcup _n F(i_1, \ldots , i_n)} & F(\coprod(\underset{i}{\coprod} X_i))\ar[d]^{F(+)}\\
& F(\underset{i}\coprod X_i)}\]
commutes for the same reasons as the previous diagram, and shows that $\Phi\Psi \tau = \tau$.

\end{proof}

Recall that for an adjoint pair of functors, 
  \[ \xymatrix{F: \C \ar@<.5ex>[r] & \sc{D}:G  \ar@<.5ex>[l] }, \] 
  where $G$ is the right adjoint, the composition $F\circ G$ forms a cotriple on ${\sc A}$.   (See, for example, Appendix A.6 of \cite{We}.)
Since $\perp _n$ is the composition of the  left adjoint $\Delta ^*\circ U^+$ with the right adjoint
${t}^+\circ\sqcup _n$, it forms part of a cotriple.   In particular, the counit for the adjunction produced by the pair $(\Delta ^*\circ U^+, {t}^+\circ\sqcup _n)$
yields a natural transformation $ \epsilon:\perp _n\rightarrow {\rm id}$.   And, a natural 
transformation $\delta : \perp _n\rightarrow \perp _n\perp _n$ is defined by $\Delta ^*\circ U^+(\eta 
_{{t}^+\circ\sqcup _n})$ where $\eta$ is a unit for the adjunction.   This gives us the following.

\begin{thm} The functor and natural transformations $$(\perp_n,
\delta : \perp _n\rightarrow \perp _n\perp _n,  \epsilon:\perp _n\rightarrow {\rm id})$$  form a cotriple on the category of functors $\Fun(\Cf,\D)$.
\end{thm}

\subsection{Weakly reduced and degree $n$ functors}
We finish this section by introducing some properties of functors that are related to $t$ and $cr_n$.  

\begin{defn} \label{d:weaklyreduced} A functor $F\in {\Fun}(\Cfn n, \D)$ is weakly $n$-reduced provided that $F(X_1, \ldots , X_n)\simeq \star$ whenever any $X_i= B$.
\end{defn}

 \begin{prop}\label{tHreduced} Let $H:\Cfn n\rightarrow \D$.  The functor $tH$ is a weakly $n$-reduced functor.\\ 
 \end{prop}

 \begin{proof} 
We assume that $X_n=B$.   The argument in other cases is similar.
We describe $H^{\bf X}_B$ as a map of two $(n-1)$-cubes:
$
{\rm top}H^{\bf X}_B\rightarrow {\rm bottom}H^{\bf X}_B.
$
For $S\in {\P}({\bf n-1})$,
$$
{\rm top}H^{\bf X}_B(S)=H^{\bf X}_B(S)
$$
and
$$
{\rm bottom}H^{\bf X}_B(S)=H^{\bf X}_B(S\cup \{n\}).
$$
Consider the $(n-1)$-cube, $\partial H^{\bf X}_B$, obtained by taking the
homotopy fiber of
$
{\rm top}H^{\bf X}_B\rightarrow{\rm bottom}H^{\bf X}_B.
$
More explicitly, for $S\in {P}({\bf n-1})$,
$$
\partial H^{\bf X}_B(S)={\rm hofib}({\rm top}H^{\bf X}_B(S)\rightarrow {\rm bottom}H^{\bf X}_B(S)).
$$
Using Lemma \ref{l:iteratedfiber} one can show that $\ifiber (\partial H^{\bf X}_B)\cong \ifiber(H^{\bf X}_B)$.   
For each $S\in {
\P}({\bf n-1})$,  the map ${\rm top}H^{\bf X}_B(S)\rightarrow {\rm bottom}H^{\bf X}_B(S)$ is the identity since $X_n=B$.   Hence,  $\partial H^{\bf X}_B(S)\simeq \star$ for each $S$,  and so the iterated homotopy fiber of $\partial H^{\bf X}_B$ is equivalent to $\star$.   As a consequence,  the iterated homotopy fiber of
$H^{\bf X}_B$ is as well.

\end{proof}

\begin{cor}\label{c:crnreduced}
For a functor $F:\Cf\rightarrow \D$, $cr_nF$ is weakly $n$-reduced. 
\end{cor}

We use cross effects to define degree $n$ functors.
\begin{defn}   A functor $F:\Cf \rightarrow \D$ is degree $n$ if and only if for all $(n+1)$-tuples ${\bf X}$ 
of objects in $\Cf$, 
$$
cr_{n+1}F({\bf X})\simeq \star.
$$
\end{defn}
Whenever a functor is degree $n$, it is also degree $m$ for any $m>n$.  This is a consequence of the following lemma, which says that higher cross effects can be obtained by iterating  second cross effects.  The lemma implies, in particular, that if $cr_nF\simeq \star$, then $cr_{n+1}F \simeq \star$ as well.

\begin{prop} For a functor $F:\Cf \to \D$, and objects $X_1, \dots, X_n, X_{n+1}$ in $\Cf$, $cr_nF(
X_1, \dots, X_{n-1}, -)$ can be treated as a functor of one variable by holding the first $n-1$ variables fixed.  The second cross effect of this functor is $cr_{n+1}$.   Specifically, 
\[cr_2[cr_{n}F( X_1, \ldots , X_{n-1}, -) ](X_n, X_{n+1}) \simeq cr_{n+1}F(X_1, \ldots , X_{n+1}).\]
\end{prop}

\begin{proof}  The proof makes repeated use of Lemma \ref{l:iteratedfiber} which allows us to rewrite the diagrams whose iterated fibers yield $cr_2(cr_nF)$  to obtain the
$(n+1)$-cubical diagram defining $cr_{n+1}$.   
We begin by noting that $$cr_2[cr_nF(X_1, \ldots , X_{n-1}, -)](X_n, X_{n+1})$$ is defined to be the iterated homotopy fiber of the diagram
\[ \xymatrix{cr_nF(X_1, \ldots, X_{n-1}, X_n\amalg_A X_{n+1}) \ar[r]\ar[d] & cr_nF(X_1, \ldots , X_{n-1}, X_n\amalg_A B)\ar[d]\\
cr_nF(X_1,\ldots , X_{n-1}, B\amalg_A X_{n+1})\ar[r] & cr_nF(X_1, \ldots , X_{n-1}, B\amalg_A B).}\]
Each corner of this square diagram is the iterated homotopy fiber of an $n$-cube, so by Lemma \ref{l:iteratedfiber} the iterated homotopy fiber of the diagram above can be written as the iterated homotopy fiber of the following $2$-cube of $n$-cubes:
\begin{equation}\label{cr2crn}\xymatrix{F((\amalg _n)^{{\bf X}_{\emptyset}}_B )\ar[r] \ar[d] & F((\amalg _n)^{{\bf X}_{\{1\}}}_B ) \ar[d] \\ 
F((\amalg _n)^{{\bf X}_{\{2\}}}_B ) \ar[r] & F((\amalg _n)^{{\bf X}_{\{1,2\}}}_B )}
\end{equation}
 where
\begin{align*}
&{\bf X}_{\emptyset}=(X_1, \dots, X_{n-1}, X_n\amalg _AX_{n+1})\\
&{\bf X}_{\{1\}}=(X_1, \dots, X_{n-1}, X_n\amalg _AB)\\
&{\bf X}_{\{2\}}=(X_1, \dots, X_{n-1}, B\amalg _AX_{n+1})\\
&{\bf X}_{\{1,2\}}=(X_1, \dots, X_{n-1},B\amalg _AB).\\
\end{align*}

 As in the proof of Proposition \ref{tHreduced}, we write each of the $n$-cubes in (\ref{cr2crn}) as a map of $(n-1)$-cubes by replacing  an $n$-cube $\chi$ with the map of $(n-1)$-cubes
 ${\rm top}(\chi)\rightarrow {\rm bottom}(\chi)$
 where for $S\in P({\bf n-1})$, 
\begin{align*}
 &{\rm top}(\chi)(S)=\chi(S)\\
 &{\rm bottom}(\chi)(S)=\chi(S\cup\{n\}).
 \end{align*}

By again applying Lemma \ref{l:iteratedfiber}, we can view the iterated fiber of (\ref{cr2crn})
as the homotopy fiber of the map from the iterated homotopy fiber of the $2$-cube of $(n-1)$-cubes
\begin{equation}\tag{{\rm top}}\xymatrix{{\rm top}F((\amalg _n)^{{\bf X}_{\emptyset}}_B )\ar[r] \ar[d] & {\rm top}F((\amalg _n)^{{\bf X}_{\{1\}}}_B ) \ar[d] \\ 
{\rm top}F((\amalg _n)^{{\bf X}_{\{2\}}}_B ) \ar[r] & {\rm top}F((\amalg _n)^{{\bf X}_{\{1,2\}}}_B )}
\end{equation}
to the iterated homotopy fiber of 
\begin{equation}\tag{{\rm bottom}}\xymatrix{{\rm bottom}F((\amalg _n)^{{\bf X}_{\emptyset}}_B )\ar[r] \ar[d] & {\rm bottom}F((\amalg _n)^{{\bf X}_{\{1\}}}_B ) \ar[d] \\ 
{\rm bottom}F((\amalg _n)^{{\bf X}_{\{2\}}}_B ) \ar[r] & {\rm bottom}F((\amalg _n)^{{\bf X}_{\{1,2\}}}_B ).}
\end{equation}

All four $(n-1)$-cubes in  (bottom) are the same, so the  iterated homotopy fiber of (bottom) is weakly equivalent to $\star$.  Thus we can concentrate on determining the iterated homotopy fiber of (top).
But this diagram can be rewritten as a single $(n+1)$-cube that is precisely the one whose iterated homotopy fiber is $cr_{n+1}F(X_1,  \dots, X_n,  X_{n+1})$.

\end{proof}


\section{Degree $n$ and $n$-excisive functors}


In the next section, we use the cotriple $\perp _{n+1}$ to construct a degree $n$ approximation to a functor $F$.  In \cite{G3}, Goodwillie shows how to construct an $n$-excisive approximation to a functor.    We use this section to compare these two types of functors, showing that $n$-excisive functors are always degree $n$, and that degree $n$ functors behave like $n$-excisive functors on certain types of 
cubical diagrams (the condition we call $n$-excisive relative to $A$).  We conclude by proving that
the two notions are equivalent when $F$ is a functor that commutes with realization.

In this section, we work with functors $F:\Cf\rightarrow \s$, where $\Cf$ is the category of maps factoring $f:A\rightarrow B$ and $\s$ is a suitable model of spectra, such as in \cite{EKMM} or \cite{HSS}. We let $\star$ be the initial/final object in $\s$.  As in previous sections, we assume that $F$ preserves weak homotopy equivalences and takes values in fibrant objects.  We also assume that $\s$ has functorial fibrant and cofibrant replacements.  We first review the definition of $n$-excisive using the notions of cartesian and
strongly cocartesian diagrams from Definition \ref{d:cartesian}.  
 
\begin{defn}\label{d:n-excisive}\cite{G2} A functor $F$ is $n$-excisive if and only if for every strongly cocartesian $(n+1)$-cube of objects in $\Cf$, $\chi $, applying $F$ yields a cartesian $(n+1)$-cube, $F(\chi)$.   \end{defn}

  For ${\bf X}=(X_1, \ldots , X_{n+1})$ in $\Cf^{\times n+1}$, the cube $(\coprod_{n+1})^A_{\bf X}$ of Example \ref{e:Upsilon} is strongly homotopy cocartesian, recalling that each $A\rightarrow X_i$ is assumed to be a cofibration.  In fact, any strongly homotopy cocartesian cube with initial vertex $A$ is weakly equivalent 
  to one of this type by Proposition 2.2 of  \cite{G2}.   We prove that a degree $n$ functor will take strongly cocartesian diagrams like these to cartesian diagrams.   
  
\begin{defn}
The functor  $F:\Cf \rightarrow \s$  is $n$-excisive relative to $A$ if and only if  $F((\coprod_{n+1})^A_{\bf X})$ is cartesian for 
every $(n+1)$-tuple of objects ${\bf X}$ in $\Cf$.  
\end{defn}

\begin{prop} \label{p:degreen} Let $F$ be a functor from $\Cf$ to $\s$.  Let $n\geq 2$ be an integer.  The functor $F$ is degree $n-1$  if and only if $F$ is $(n-1)$-excisive relative to $A$.  \end{prop}

An integral part of the proof of this proposition will be the $n$-cube of $n$-cubes defined below.

\begin{defn}\label{d:supercube}
Let ${\bf X}=(X_1, \dots, X_n)$ be an $n$-tuple of objects in $\Cf$.  Each of these objects is equipped with morphisms $\alpha _{X_i}:A\rightarrow X_i$ and $\beta _{X_i}: X_i
\rightarrow B$ whose composition is $f$.  We use $\widetilde {\bf X}: {\mathcal P}({\bf n})
\times {\mathcal P}({\bf n})\rightarrow \Cf^{\times n}$ to denote the $n$-cube of $n$-cubes  that is defined as follows.   For $(S, T)\in {\mathcal P}({\bf n})
\times {\mathcal P}({\bf n})$, $
\widetilde {\bf X}(S,T)
$ is the $n$-tuple 
whose  $i$th object  is 
\[
(\widetilde {\bf X}(S,T))_i=\begin{cases} A, & \text{ if $i\notin S\cup T$;}\\
X_i, & \text{if $i\in S$, $i\notin T$;}\\
B, & \text{ if $i\in T$.}\end{cases} 
\]
For any $i\notin T$ the map $\widetilde {\bf X}(S,T)\to \widetilde {\bf X}(S,T\cup \{i\})$ is induced by the map $f$ if $i\notin S$, and otherwise is induced by the map $\beta_{X_i}$.  For $i\notin S$ the map $\widetilde {\bf X}(S,T)\to \widetilde {\bf X}(S\cup\{ i\},T)$ is induced by the map $\alpha_{X_i}$ if $i\notin T$ and otherwise is the identity map.   
\end{defn}

The target category for our functor $F$ is assumed to be stable.     Some of the subsequent results in this section hold in a more general context, but for our current applications using $\s$ as the target category suffices.  This 
enables us to make use of the following  observations.

\begin{rem}
   An $n$-cubical diagram in $\s$  is cocartesian if and only if it is cartesian. In particular, this implies that finite hocolimits and finite homotopy inverse limits commute. 
\end{rem}

\begin{rem}
   In $\s$, finite homotopy inverse limits commute with homotopy colimits of a countable filtered diagram. Thus, combined with the remark above one sees that finite homotopy inverse limits commute with the homotopy colimits over $\Delta^{op}$ as these homotopy colimits can be written using the filtration by skeleta as a countable filtered homotopy colimit of finite homotopy colimits. 
\end{rem}

\begin{rem}\label{r:tfiber}
Let $\chi$ be an $n$-cubical diagram.  The total homotopy fiber of $\chi$, denoted $\tfiber(\chi)$ is the homotopy fiber of the map 
\[\chi(\emptyset)\rightarrow \underset{{\P_0({\bf n})}}{\holim}(\chi).
\]
For $n$-cubes $\chi$ in $\s$, the iterated fiber and homotopy fiber of $\chi$ are weakly equivalent.   See section 1 of \cite{G2} for details.  
\end{rem}

The next lemma restates two propositions from \cite{G2}.   
 The lemma makes use of the fact that a map  of two  $n$-cubes, $\chi_1\to \chi_2$, is an $(n+1)$-cube.   

\begin{lem} \label{l:cartesian} (\cite{G2} 1.6, 1.7)  For any map $\chi_1\to \chi_2$ of $n$-cubes of objects in $\s$:
\begin{itemize}
\item  The $n$-cube $\chi_1$ is cartesian if $\chi_2$ is cartesian and the $(n+1)$-cube $\chi_1\to \chi_2$ is cartesian.
\item The $n$-cube $\chi_2$ is cartesian if $\chi_1$ is cartesian and the $(n+1)$-cube $\chi_1\to \chi_2$ is cartesian.
\item The $(n+1)$-cube $\chi_1\to\chi_2$ is cartesian if the $n$-cubes $\chi_1$ and $\chi_2$ are cartesian.
\end{itemize} 
\end{lem}

The proof  of Proposition \ref{p:degreen} relies on analyzing  $\widetilde{\bf X}$ from several 
different perspectives.   We single out two of these perspectives in the next remark.   Recall from 
Definition \ref{d:sqcup}  that 
$\sqcup _n$ denotes precomposition with the $n$-fold coproduct functor.

\begin{rem}\label{r:viewcube}  Let ${\bf X}$ be an $n$-tuple of objects in $\Cf$ and $F:\Cf\rightarrow \s$ be 
a functor.   
\begin{enumerate} 
\item {Fixing $S\in {\mathcal P({\bf n})}$ yields an $n$-cube, $\sqcup^n\widetilde {\bf X}(S,-)$.   When 
$S={\bf n}$, this $n$-cube  is $(\sqcup^n)^{\bf X}_B$, and, as such,  is precisely the type of $n$-cube used to show that a functor is degree $n-1$.
In other words, $F$ is degree $n-1$ if and only if   the 
$n$-cube $\sqcup _nF(\widetilde {\bf X}({\bf n}, -))$ is cartesian.}
\item{Fixing $T\in {\mathcal P({\bf n})}$ yields another $n$-cube, $\sqcup^n\widetilde {\bf X}(-,T)$.   When
$T={\emptyset}$, this $n$-cube is $(\sqcup^n)^A_{\bf X}$, and, as such, is precisely the 
type of $n$-cube used to show that a functor is $(n-1)$-excisive relative to $A$.   In other words,
$F$ is $(n-1)$-excisive relative to $A$ if and only if $\sqcup _nF(\widetilde {\bf X}(-,\emptyset))$
is cartesian.}
\end{enumerate}

For example, taking the first point of view when $n=2$, we have the  $2$-cube of  $2$-cubes $\sqcup _nF(\widetilde {\bf X}(S,T))$ whose outer square is indexed by the $S$ variable:
\setlength{\unitlength}{0.5cm}
\begin{center}
\begin{picture}(16,17)
\put(0,13.75){$F(A)$}
\put(0, 9){$F(B)$}
\put(5, 13.75){$F(B)$}
\put(4, 9){$F(B\coprod_A B)$}
\put(0.75,13.5){\vector(0,-1){3.75}}\put(-0.5, 11.5){\SMALL{$F(f)$}}
\put(1.75,14){\vector(1,0){3}}\put(2.5, 14.5){\SMALL{$F(f)$}}
\put(5.75,13.5){\vector(0,-1){3.75}}
\put(1.75, 9.25){\vector(1,0){2}}

\put(6.5, 12){\vector(1,0){3}}\put(6.5, 13){\SMALL{ $\left( \begin{matrix} F(\alpha_{X_2})\\ F(1) \end{matrix} \right) $ }}
\put(6.5, 11.5){\vector(1,0){3}}

\put(3, 8.5){\vector(0,-1){3}}\put(-1.5, 7){\SMALL{ $\left(\begin{matrix} F(\alpha_{X_1}) & F(1) \end{matrix}\right)$ }}
\put(3.5,8.5){\vector(0, -1){3}}

\put(9.5,13.75){$F(X_2)$}
\put(8.5, 9){$F(B\coprod_A X_2)$}
\put(14.5, 13.75){$F(B)$}
\put(13.5, 9){$F(B\coprod_A B)$}
\put(10.25,13.5){\vector(0,-1){3.75}}
\put(11.5,14){\vector(1,0){2.75}}\put(12, 14.5){\SMALL{$F(\beta_{X_2})$}}
\put(15.25,13.5){\vector(0,-1){3.75}}\put(15.5, 11.5){\SMALL{$F(f)$}}
\put(12.5, 9.25){\vector(1,0){1}}

\put(12.5, 8.5){\vector(0, -1){3}}\put(13.5, 7){\SMALL{$\left(\begin{matrix} F(\alpha_{X_1}) & F(1)\end{matrix}\right)$}}
\put(13, 8.5){\vector(0, -1){3}}

\put(0,4.75){$F({X_1})$}
\put(0, 0){$F(B)$}
\put(4, 4.75){$F(X_1\coprod_A B)$}
\put(4, 0){$F(B\coprod_A B)$}
\put(0.75,4.5){\vector(0,-1){3.75}}\put(-1.25, 2.5){\SMALL{$F(\beta_{X_1})$}}
\put(2,5){\vector(1,0){2}}
\put(5.75,4.5){\vector(0,-1){3.75}}
\put(1.75, 0.25){\vector(1,0){2.25}}\put(2, -.5){\SMALL{$F(f)$}}

\put(6.5, 2.5){\vector(1,0){3}}\put(7, 1.5){\SMALL{$\left(\begin{matrix} F(\alpha_{X_1}) \\ F(1)\end{matrix}\right)$}}
\put(6.5, 3){\vector(1,0){3}}

\put(8.5,4.75){$F(X_1\coprod_A X_2)$}
\put(8.5, 0){$F(B\coprod_A X_2)$}
\put(13.5, 4.75){$F(X_1\coprod_A B)$}
\put(13.5, 0){$F(B\coprod_A B).$}
\put(10.25,4.5){\vector(0,-1){3.75}}
\put(12.5,5){\vector(1,0){1}}
\put(15.25,4.5){\vector(0,-1){3.75}}\put(15.5, 2.5){\SMALL{$F(\beta_{X_1})$}}
\put(12.5, 0.25){\vector(1,0){1}}\put(12, -.5){\SMALL{$F(\beta_{X_2})$}}

\end{picture}
\end{center}
\vskip .5 cm
\noindent
Taking the second point of view yields the $2$-cube of $2$-cubes $\sqcup _2F(\widetilde {\bf X}(S,T))$ whose outer square is indexed by the $T$ variable:
\setlength{\unitlength}{0.5cm}
\begin{center}
\begin{picture}(16,16)
\put(0,13.75){$F(A)$}
\put(0, 9){$F(X_1)$}
\put(5, 13.75){$F(X_2)$}
\put(4, 9){$F(X_1\coprod_AX_2)$}
\put(0.75,13.5){\vector(0,-1){3.75}}\put(-1, 11.5){\SMALL{$F(\alpha_{X_1})$}}
\put(1.5,14){\vector(1,0){3.25}}\put(2.5, 14.5){\SMALL{$F(\alpha_{X_2})$}}
\put(5.75,13.5){\vector(0,-1){3.75}}
\put(1.75, 9.25){\vector(1,0){2}}

\put(6.5, 12){\vector(1,0){3}}\put(6.5, 13){\SMALL{ $\left( \begin{matrix} F(f)\\ F(\beta_{X_2}) \end{matrix} \right) $ }}
\put(6.5, 11.5){\vector(1,0){3}}

\put(3, 8.5){\vector(0,-1){3}}\put(-1, 7){\SMALL{ $\left(\begin{matrix} F(f) & F(\beta_{X_1}) \end{matrix}\right)$ }}
\put(3.5,8.5){\vector(0, -1){3}}

\put(9.5,13.75){$F(B)$}
\put(8.5, 9){$F(X_1\coprod_AB)$}
\put(14.5, 13.75){$F(B)$}
\put(13.5, 9){$F(X_1\coprod_A B)$}
\put(10.25,13.5){\vector(0,-1){3.75}}
\put(11,14){\vector(1,0){3.25}}\put(12, 14.5){\SMALL{$F(1)$}}
\put(15.25,13.5){\vector(0,-1){3.75}}\put(15.5, 11.5){\SMALL{$F(\alpha_{X_1})$}}
\put(12, 9.25){\vector(1,0){1.5}}

\put(12.5, 8.5){\vector(0, -1){3}}\put(13.5, 7){\SMALL{$\left(\begin{matrix} F(f) & F(\beta_{X_1})\end{matrix}\right)$}}
\put(13, 8.5){\vector(0, -1){3}}

\put(0,4.75){$F(B)$}
\put(0, 0){$F(B)$}
\put(4, 4.75){$F(B\coprod_A X_2)$}
\put(4, 0){$F(B\coprod_AX_2)$}
\put(0.75,4.5){\vector(0,-1){3.75}}\put(-.5, 2.5){\SMALL{$F(1)$}}
\put(1.5,5){\vector(1,0){2.5}}
\put(5.75,4.5){\vector(0,-1){3.75}}
\put(1.5, 0.25){\vector(1,0){2.5}}\put(2, -.5){\SMALL{$F(\alpha_{X_2})$}}

\put(6.5, 2.5){\vector(1,0){3}}\put(7, 1.5){\SMALL{$\left(\begin{matrix} F(f) \\ F(\beta_x)\end{matrix}\right)$}}
\put(6.5, 3){\vector(1,0){3}}

\put(8.5,4.75){$F(B\coprod_A B)$}
\put(8.5, 0){$F(B\coprod_A B)$}
\put(13.5, 4.75){$F(B\coprod_A B)$}
\put(13.5, 0){$F(B\coprod_A B)$.}
\put(10.25,4.5){\vector(0,-1){3.75}}
\put(12,5){\vector(1,0){1.5}}
\put(15.25,4.5){\vector(0,-1){3.75}}\put(15.5, 2.5){\SMALL{$F(1)$}}
\put(12, 0.25){\vector(1,0){1.5}}\put(12, -.5){\SMALL{$F(1)$}}

\end{picture}
\end{center}

\end{rem}
\vskip .5 cm
The final step before proving Proposition \ref{p:degreen} is to prove the next lemma.

\begin{lem}\label{l:degreen} Let ${\bf X}$ be an $n$-tuple of objects in $\Cf$ and $F:\Cf\rightarrow \s$ be 
a functor.   If $F$ is $(n-1)$-excisive relative to $A$, then for each $S\subset {\bf n}$, $S\neq {\bf n}$, 
the $n$-cube $\sqcup _n F(\widetilde {\bf X}(S,-))$ is cartesian.
\end{lem} 

\begin{proof}   We prove this by induction on the size of $S$.   When $S=\emptyset$, 
$\sqcup^n(\widetilde {\bf X}(\emptyset, -))$ is the $n$-cube $(\sqcup^n)^A_{\bf B}$
where ${\bf B}=(B, B, \dots, B)$.   Since $F$ is $(n-1)$-excisive relative to $A$, 
we know that $\sqcup _nF(\widetilde {\bf X}(\emptyset, -))$ is cartesian.  

Let $k<n$ and assume that $\sqcup _nF(\widetilde {\bf X}(R, -))$ is cartesian for all $R\subset {\bf n}$ with $|R|<k$.   Let $S\subseteq {\bf n}$ with $|S|=k$.   To establish that $\sqcup _n F(\widetilde {\bf X}(S,-))$ is cartesian, we will first show that the $S$-cube of ${\bf n}$-cubes 
\begin{equation}\label{pr:cubeweneed}
\sqcup _nF(\widetilde {\bf X}(R,T)), \ \ \ \ R\subseteq S, T\subseteq {\bf n}
\end{equation}
obtained by 
restricting $\widetilde {\bf X}$ to $\mathcal{P}(S)\times \mathcal{P}({\bf n})$ is cartesian.   To do 
so we view this $S$-cube of ${\bf n}$-cubes from a different perspective, in particular as an $S$-cube of $S$-cubes of $({\bf n}-S)$-cubes.     

Fix $R'\subseteq  S$
and consider the $S$-cube of $({\bf n}-S)$-cubes, i.e., the ${n}$-cube, given by 
\begin{equation}\label{pr:newcube}
\sqcup _nF(\widetilde {\bf X}(R, R'\cup U))
\end{equation}
where $R$ varies over all subsets of $S$ and $U$ varies over all subsets of ${\bf n}-S$.  We claim
that for each $R'$, this is a cartesian ${n}$-cube.  
When $R'=\emptyset$, this follows because (\ref{pr:newcube}) is $F((\sqcup^n)^A_{\bf Y})$ where ${\bf Y}$ is the
$n$-tuple whose $j$th entry is 
\[
({\bf Y})_j=\begin{cases} X_j&\text{if $j\in S$}\\
B&\text{if $j\notin S$,}
\end{cases}
\] 
and $F$ is $(n-1)$-excisive relative to $A$.    Now 
suppose that $R'\neq \emptyset$ and $i\in R'$.   In this case, we can view (\ref{pr:newcube})
as a map of $(|S|-1)$-cubes of $({\bf n}-S)$-cubes:
\begin{equation}\label{nextnewcube}
\sqcup _nF(\widetilde {\bf X}(R'',R'\cup U))\rightarrow \sqcup _nF(\widetilde {\bf X}(R''\cup \{i\}, R'\cup U))
\end{equation}
where $R''\subseteq S- \{i\}$ and $U\subseteq {\bf n}- S$ vary.
By definition, the map in (\ref{nextnewcube}) is the identity for each choice of $R''$ and $U$.   
Applying Lemma \ref{l:cartesian}, we see that the $S$-cube of $({\bf n}- S)$-cubes in (\ref{pr:newcube})
is cartesian as a result.   

The $S$-cube of ${\bf n}$-cubes obtained by letting $R'$ in (\ref{pr:newcube}) vary over all subsets of $S$ is exactly
(\ref {pr:cubeweneed}).  Since (\ref{pr:newcube}) is cartesian for each choice of $R'$, Lemma {\ref {l:cartesian}} tells us that (\ref{pr:cubeweneed}) is cartesian.        

Finally, by assumption, we know that 
$\sqcup _nF(\widetilde {\bf X}(R,-)$ is a cartesian $n$-cube for each $R\neq S$.   Hence, the 
fact that  (\ref{pr:cubeweneed}) is cartesian coupled with Lemma \ref {l:cartesian} again
yields the fact that $\sqcup _nF(\widetilde {\bf X}(S,-)$ must be a cartesian $n$-cube as well.
\end{proof}

With this, we prove Proposition \ref{p:degreen}.

\begin{proof}
We begin by considering $\sqcup _nF(\widetilde {\bf X}(-,-))$ from the point of view of Remark \ref {r:viewcube}.2.  Let  $T\subset {\bf n}$, $T\neq \emptyset$, and $i\in T$.   We can view $\sqcup _nF(\widetilde{\bf X}(-,T))$ as a map of two $(n-1)$-cubes
\begin{equation}\label{e:splitcube}
\sqcup _nF(\widetilde {\bf X}(V,T))\rightarrow \sqcup _nF(\widetilde {\bf X}(V\cup \{i\},T))
\end{equation}
where $V\subseteq S- \{i\}$.  However, by definition, these two $(n-1)$-cubes are 
identical and the map between them is the identity.   Hence, $\sqcup _nF(\widetilde{\bf X}(-,T))$ is cartesian
for each $T\subseteq {\bf n}$, $T\neq \emptyset$. 

Now assume that $F$ is degree $n-1$.   By Remark \ref{r:viewcube}.2, it suffices to show that $\sqcup _nF(\widetilde{\bf X}(-,\emptyset))$ is cartesian to conclude that $F$ is $(n-1)$-excisive relative to $A$.    Viewing the cube as in Remark \ref{r:viewcube}.1, we see that for any $S\subseteq {\bf n}$, $\widetilde {\bf X}(S,-)$ is the $n$-cube $(\sqcup ^n)^{{\bf Z}(S)}_B$ where ${\bf Z}(S)$ is the $n$-tuple whose $i$th entry is 
\[
{\bf Z}(S)_i=\begin{cases} X_i&\text{if $i\in S$}\\
A&\text{if $i\notin S$}.  \end{cases}
\]
Since $F$ is degree $n-1$, it follows that $\sqcup _nF(\widetilde {\bf X}(S,-))$ is cartesian.
Then by Lemma \ref{l:cartesian}, $\sqcup _nF( \widetilde{\bf X}(-,-))$ is cartesian.   
Since $\sqcup _nF(\widetilde{\bf X}(-,T))$ is cartesian for 
each $T\neq \emptyset,$
Lemma \ref{l:cartesian} guarantees that $\sqcup _nF(\widetilde{\bf X}(-,\emptyset))$ must be as well.   

Assuming that $F$ is $n$-excisive relative to $A$, by Remark \ref{r:viewcube}.1, we need to show that $\sqcup _nF(\widetilde {\bf X}({\bf n},-))$ is cartesian to conclude that $F$ is degree $n-1$.  By Lemma \ref{l:degreen}, we know that $\sqcup _nF(\widetilde {\bf X}(S,-))$ is cartesian for each $S\subset {\bf n}$, $S\neq {\bf n}$.  Hence by Lemma \ref{l:cartesian} it
suffices to show that $\sqcup _nF(\widetilde {\bf X}(-,-))$ is cartesian.  We do so by applying Lemma \ref{l:cartesian} after making sure that 
$\sqcup _nF(\widetilde {\bf X}(-,T))$ is cartesian for each $T\subseteq {\bf n}$.  This has been done
for $T\neq \emptyset$ above.   When $T=\emptyset$, $\widetilde{\bf X}(-,T)$ is $(\sqcup^n)^{A}_{\bf X}$ and so $\sqcup _nF(\widetilde {\bf X}(-,T))$ is cartesian by assumption.   
\end{proof}

One can extend a functor $F:\Cf\rightarrow\s$ to the category of  simplicial objects $s\Cf$ in two ways -- by applying $F$ degreewise to a simplicial object $X.$ to obtain a simplicial object $F(X.)$ in $\s$ or by 
applying $F$ to the geometric realization $|X.|$ (recall our convention that all realizations are ``fat''). 
When these two approaches agree, i.e., when the natural map $|F(X.)|\rightarrow F(|X.|)$ is a weak equivalence 
for each simplicial object $X.$, we say that $F$ {\it commutes with realizations}. In particular, by Remark 4.6, in $\s$, finite homotopy limits commute with realizations.

\begin{prop} \label{p:realizations} If  $F:\Cf \rightarrow \s$ commutes with realizations and preserves weak equivalences, then $F$ is degree $n$ if and only if $F$ is $n$-excisive.
\end{prop}

\begin{proof}  By Proposition \ref{p:degreen}, we know that an $n$-excisive functor is always degree $n$, so we need only show that if $F$ is  degree $n$ then $F$ is $n$-excisive. 
The strategy for the proof is to show that any strongly cocartesian $(n+1)$-cube ${\mathcal X}$ can be replaced by an equivalent $(n+1)$-cube $B({\mathcal X})$ built using the generalized bar construction, as defined 
in Section 2.2.   We are then able to show that $F(B({\mathcal X}))$ is cartesian by applying Lemma \ref{l:cartesian} levelwise to a map of $(n+1)$-cubes $F(B^{-1}({\mathcal X}))\rightarrow F(B({\mathcal X})).$
We illustrate the case $n=1$ first, with the proof of $n>1$ to follow.  

Let 
\begin{equation}\label{e:cosquare}\xymatrix{ X \ar[r] \ar[d] & Y\ar[d]  \\ Z \ar[r] &W}\end{equation}
 be a cocartesian square.    Let $B_\bullet(X,Y)$ and $B_\bullet(X,Z)$ be the generalized bar constructions whose realizations yield ${\rm hocolim}(X\rightarrow Y)$ and ${\rm hocolim}(X\rightarrow Z)$, respectively.  Since $\operatorname{hocolim}\{X\to Y\} \simeq Y$, we know that  $|B_\bullet(X,Y)|\simeq Y$.   Similarly, $|B_{\bullet}(X,Z)|\simeq Z$, and $|B_{\bullet}(X)|\simeq X$ where
 $B_{\bullet}(X)$ is the generalized bar construction coming from the trivial diagram $X$.   
By the first two parts of Lemma \ref{l:Bar}, 
\[ \xymatrix{ B_{\bullet}(X) \ar[r] \ar[d] & B_{\bullet}(X,Y)  \\B_{\bullet}(X, Z)}\] 
 is a cofibrant replacement of the diagram
 \[ \xymatrix{ X \ar[r] \ar[d] & Y  \\ Z, }\] and hence we can take the strict pushout of this diagram in order to compute the homotopy colimit $W$.  But by Lemma \ref{l:Bar}.3, this strict pushout 
 is precisely the bar construction whose realization yields ${\rm hocolim}(Y\leftarrow X\rightarrow Z)$.  We use $B_{\bullet}(Y,X,Z)$ to denote this generalized bar construction.   Since $F$ preserves weak equivalences, we may replace our original diagram (\ref{e:cosquare}) with
 \begin{equation}\label{e:barsquare}\xymatrix{ B_{\bullet}(X) \ar[r] \ar[d] & B_{\bullet}(X,Y)\ar[d]  \\B_{\bullet}(X, Z) \ar[r] &B_{\bullet}(Y,X,Z).}
\end{equation}
  Looking levelwise 
  we see that  in degree $k$, the diagram is
\[ \xymatrix{X \ar[r] \ar[d] &Y\coprod_A X\coprod_A \cdots \coprod_A X \ar[d]\\
Z\coprod_A X\coprod_A \cdots \coprod_A X \ar[r]&Y\coprod_A X\coprod_A \cdots \coprod_A X\coprod_AZ}\]
where the coproducts in the upper right and lower left corners contain $k+1$ copies of $X$ and that in the bottom right corner has $2(k+1)-1$ copies of $X$.  

After applying the degree $1$ functor $F$ to (\ref{e:barsquare}),  we would like to show that 
\[F(B_\bullet(X,Y,Z)) = B_\bullet(\tilde F(X), \tilde F(Y), \tilde F(Z)).\]
(where $\tilde F(-)$ represents the cofibrant replacement of $F(-)$).  This is easily done in the degree $1$ case because Proposition \ref{p:degreen} guarantees that whenever $F$ is degree $1$,  the diagram 
\[\xymatrix{F(A)\ar[r]\ar[d]&F(X_1)\ar[d]\\
F(X_2)\ar[r]&F(X_1\underset{A}\amalg X_2)}
\]
obtained by applying $F$ to a cocartesian square must be cartesian.   Since $F$ takes values in $\s$, the square is also cocartesian and so
\[F(X_1\coprod _A X_2)\simeq F(X_1)\coprod _{F(A)}F(X_2).\]
However, this approach cannot be generalized to functors of degree $n>1$, so instead we will describe  the proof for degree $1$ with an eye towards the general case.

At each simplicial level, we want to expand (\ref{e:barsquare}) into the cube 
\begin{equation}\label{e:barcube} \xymatrix{ A \ar[rr] \ar[dd] \ar[dr] && B_k^{-1}(X,Y) \ar[dd] \ar[dr] \\
& X \ar[rr] \ar[dd] && B_k(X,Y)\ar[dd] \\
B^{-1}_k(X,Z) \ar[rr] \ar[dr] && B^{-1}_k(Y,X,Z) \ar[dr]\\
& B_k(X,Z) \ar[rr] && B_k(Y,X,Z)}\end{equation}
where $B_k(X,Y)$, $B_k(X,Z)$ and $B_k(Y,X,Z)$ are the $k$th levels of the simplicial sets described above, and  $B^{-1}_k(X,Y)$ (respectively, $B_k^{-1}(X,Z)$ and $B_k^{-1}(Y,X,Z)$) is the object of $\Cf$ obtained from $B_k(X,Y)$ (respectively, $B_k(X,Z)$ and $B_k(X,Y,Z)$) by removing the copy of $X$ corresponding to the constant sequence $X= \cdots = X$.  The map $B^{-1}_k(X,Y) \to B_k(X,Y)$ is then given by the natural inclusion (respectively, $B_k^{-1}(X,Z)\to B_k(X,Z)$ and $B_k^{-1}(Y,X,Z)\to B_k(Y, X, Z)$).  We make no claim that $B^{-1}_\bullet(X,Y)$ is a simplicial object.  Each square face of this cube is easily seen to be a strict pushout, so the cube is strongly cocartesian.  As we determined in Section 3.4, the fact that  $F$ is degree 1 implies that it  is also degree 2.   Then by Proposition \ref{p:degreen}, we know that applying $F$ to (\ref{e:barcube}) yields a cartesian diagram.     The back face of (\ref {e:barcube})
\[\xymatrix{A \ar[r] \ar[d] & B^{-1}_k(X,Y)\ar[d] \\ B^{-1}_k(X,Z) \ar[r] & B^{-1}_k(Y,X,Z)}\]
is a pushout diagram with initial vertex $A$, and since $F$ is degree 1, Proposition \ref{p:degreen}  tells us that this square will be a cartesian square after $F$ is applied.  Recognizing that $F$ takes values in $\s$, and applying  Lemma \ref{l:cartesian}, we can  conclude that 
\[ \xymatrix{ F(B_k(X)) \ar[r] \ar[d] & F(B_k(Y)) \ar[d] \\ F(B_k(X,Z)) \ar[r] & F(B_k(Y,X,Z) )}\]
is also a pushout square.

We have now calculated that the diagram
\[ \xymatrix{ F(B_{\bullet}(X)) \ar[r] \ar[d] & F(Y) \simeq F(B_\bullet(X,Y)) \ar[d] \\
F(B_\bullet(X,Z))\simeq F(Z) \ar[r] & F(B_\bullet(Y,X,Z))}\]
is a pushout diagram levelwise, and hence is a pushout diagram of simplicial sets.  
As a result, 
\[|F(B_{\bullet}(Y,X,Z))|\simeq |\hocolim(F(B_{\bullet} (X,Y))\leftarrow F(B_{\bullet}(X))\rightarrow F(B_{\bullet}(X,Z))|.\]  
However, since
$F$ commutes with realizations and hocolim commutes with realizations, the right hand side of this
equivalence is equivalent to the homotopy pushout of
$F(Y)\leftarrow F(X)\rightarrow F(Z).$

To complete the proof, we note that by Remark 4.5  it suffices to show that applying $F$ to (\ref{e:cosquare}) yields a cocartesian square. Since $F$ commutes with realization  we have
\[ F(\operatorname{hocolim} \{Z \leftarrow X \rightarrow Y\}) \simeq F(|B_\bullet(Y,X,Z)|)\]\[ \simeq |F(B_\bullet(Y,X,Z))| \simeq \operatorname{hocolim} \{F(Z) \leftarrow  F(X) \rightarrow  F(Y)\},\]
which concludes the proof in the case $n=1$.

To prove the result for $n>1$, we let ${\mathcal X}:{\P}({\bf n}) \to \Cf$ be a strongly cocartesian $n$-cube.  For $S\in \P({\bf n})$, let ${\mathcal D}_S$ be the restriction of $\P({\bf n})$ to the collection of sets $\{ \emptyset, \{s\}\ |\ s\in S\}$.   Let $B_\bullet({\mathcal X}, S)$ be the generalized bar construction computing $\operatorname{hocolim}_{{\mathcal D}_S} $.   In simplicial level $k$, 
\[ B_k ({\mathcal X}, S) = \coprod_{\alpha_0\to \cdots \to \alpha_k \in {\mathcal D}_S} {\mathcal X}(\alpha_0) \]
where all coproducts are taken over ${\mathcal X}(\emptyset)$. 
We claim that the $n$-cube ${\mathcal X}$ is weakly equivalent to the $n$-cube $B_{\bullet}({\mathcal X},-)$.     One can verify this by using the fact that ${\mathcal X}$ is strongly cocartesian and applying Lemma \ref{l:Bar}.3.  In particular, if $S$, $S'$  are subsets of $\bf n$, then 
\[ \xymatrix { B_\bullet({\mathcal X}, S\cap S')\ar[r] \ar[d] & B_\bullet({\mathcal X}, S')\ar[d] \\ B_\bullet({\mathcal X}, S) \ar[r] & B_\bullet({\mathcal X}, S\cup S')}\]
is a pushout diagram.

Now, for each $k$ we consider the $(n+1)$-cube $B^{-1}_k({\mathcal X}, - ) \to B_k({\mathcal X}, -)$ where $B^{-1}_k({\mathcal X}, S)$ is obtained from  $B_k({\mathcal X}, S)$ by removing the copy of ${\mathcal X}(\emptyset)$ corresponding to the constant sequence $\emptyset\to \cdots \to\emptyset$  for each $\emptyset\neq S\subset {\bf n}$, and $B^{-1}_k({\mathcal X}, \emptyset)= A$.   There are natural inclusion maps $B^{-1}_k({\mathcal X}, S) \to B_k({\mathcal X}, S)$ given by the inclusion $A\to X$ indexed by the sequence $\emptyset \to \cdots \to \emptyset$.  The $n$-cube $B^{-1}_k({\mathcal X}, S)$ is again strongly cocartesian.  For any $S \to S\cup \{i\}$ in $\P({\bf n})$, the diagram
\[ \xymatrix{ B^{-1}_k({\mathcal X}, S) \ar[r] \ar[d] & B^{-1}_k({\mathcal X}, S\cup \{i\}) \ar[d] \\
B_k({\mathcal X}, S) \ar[r]  & B_k({\mathcal X}, S\cup \{i\})} \]
is easily seen to be a pushout diagram, since it can be rewritten as
\[ \xymatrix{ B^{-1}_k({\mathcal X}, S) \ar[r] \ar[d] & B^{-1}_k({\mathcal X}, S\cup \{i\}) \ar[d] \\
B^{-1}_k({\mathcal X}, S)\coprod_A X \ar[r]  & B^{-1}_k({\mathcal X}, S\cup \{i\})\coprod_A X.} \]

Thus the $(n+1)$-cube $B^{-1}_k({\mathcal X}, S)\to B_k({\mathcal X}, S)$ and the $n$-cube $B^{-1}_k({\mathcal X}, S)$ are strongly cocartesian cubes with initial vertex $A$, and any degree $n-1$ functor $F$ will take these to cartesian cubes by Proposition \ref{p:degreen}.  The remainder of the proof now proceeds exactly as in the case $n=1$.

\end{proof}

\section{A Taylor tower from cotriples}

Having established that $(\perp _n, \delta, \epsilon)$ is a cotriple for $n\geq 1$, we show in this
section how to define the cotriple Taylor tower for a functor from $\Cf$ to $\s$ where $\Cf$ and $\s$ are as described in the beginnings of Sections 3 and 4.    The results in this
section are generalizations to $\Cf$ and $\s$ of results in Section 2 of \cite{RB}.  In the next 
section of this paper, we compare this cotriple Taylor tower to Goodwillie's Taylor tower of $n$-excisive
approximations.  

To define the terms in our Taylor tower, we use the augmented simplicial objects associated
to the cotriples $(\perp _n, \delta, \epsilon)$.

\begin{defn}  Let $F$ be a functor from $\Cf$ to $\s$ and $n\geq 1$.   We use $\perp _n^{*+1}F$
to denote the simplicial object constructed from $F$ using the cotriple $(\perp _n, \delta, \epsilon).$   More specifically, $\perp _n^{*+1}F$ is the simplicial object that in simplicial degree 
$k$ is 
$
\perp _n^{k+1}F
$
with face and degeneracy maps defined by 
\begin{align*}
s_i=\perp ^{i}\delta \perp ^{k-i}\  :\   \perp _n^{k}F\rightarrow \perp _n^{k+1}F\\
d_i=\perp ^{i}\epsilon\perp ^{k-i}\  : \ \perp _n^{k}F\rightarrow \perp _n^{k-1}.
\end{align*}
\end{defn}

\begin{rem}
We note that $\perp _n^{*+1}$ is augmented over ${\rm id}_{\Fun(\Cf, \s)}$ by $\epsilon$.
Hence, $\epsilon$ yields a natural simplicial map from $\perp _n^{*+1}$ to the 
simplicial object ${\rm id}^{*+1}$ associated to the identity cotriple $({\rm id}_{\Fun(\Cf, \s)}, {\rm id},
{\rm id})$ built out of the identity functor.  
\end{rem}

\begin{defn}
Let $F$ be a functor from $\Cf$ to $\s$ and $n\geq 0$.    The $n$th term in the cotriple Taylor tower of 
$F$ is the functor 
\[
\Gamma_nF:={\operatorname{hocofiber}}(\xymatrix{|\perp_{n+1}^{*+1}F| \ar[r]^{\ \ \ \ \widehat\epsilon}&F})
\]
(recall our convention that realizations are ``fat'') where the map \[\xymatrix{|\perp_{n+1}^{*+1}F| \ar[r]^{\ \ \ \ \widehat\epsilon}&F}\] is the composition of the map induced by $\epsilon$ with the weak equivalence \[\xymatrix{|{\rm id}^{*+1}F| \ar[r]^{\ \ \ \ \ \simeq}& F}\] and $\operatorname{hocofiber}$ denotes the homotopy cofiber given by 
$$
\operatorname{hocofiber}(A\rightarrow B):=\hocolim(*\longleftarrow A\longrightarrow B).
$$
We use $\gamma _nF$ to denote the natural transformation $F\rightarrow \Gamma_nF$ in the resulting cofibration 
sequence $|\perp_{n+1}^{*+1}F|\rightarrow F\rightarrow \Gamma_nF$.  
\end{defn}

\bigskip\noindent
{\bf Convention.} Even when $F$ takes fibrant values, $\Gamma_nF$ may not. Since $\s$ has functorial fibrant replacements, we can replace $\Gamma_nF$ by a fibrant-valued, weakly equivalent functor $\hat\Gamma_nF$, in which case we actually obtain a natural diagram
\[\xymatrix{F\ar[r]& \Gamma_nF& \hat\Gamma_nF\ar[l]_{\simeq}}\]
We will abuse notation and continue to write this as $F\rightarrow \Gamma_nF$ and assume without loss of generality that our $\Gamma_nF$ takes fibrant values. 
\bigskip

Using properties of adjoint pairs, one can show that $\Gamma _nF$ is a degree $n$ 
approximation to $F$.

\begin{prop}\label{degPn}
For a functor $F:\Cf \rightarrow \s$, the functor $\Gamma_nF$ is degree $n$.  
\end{prop}

The analogous statement for functors to abelian categories, in \cite{RB} (Lemma 2.11), was proved by establishing that $\perp_{n+1}\perp_{n+1}^{*+1}F\rightarrow 
\perp_{n+1}F$ has a simplicial homotopy inverse, and hence, $\perp _{n+1} \Gamma_nF$ is contractible.
In the abelian setting, this was enough to conclude that for any collection of objects, $X_1, \dots, X_{n+1}$, in the domain category,  $cr _{n+1} \Gamma_nF(X_1, \dots, X_{n+1})$ is contractible, since it
is a direct summand of $\perp _{n+1}\Gamma_nF(\coprod X_i, \dots, \coprod X_i)$.   In the present paper, 
this last step is not an option.   However, one can modify the homotopies used in the abelian
case to prove directly that $$cr _{n+1}\perp _{n+1}^{*+1}F\rightarrow cr_{n+1}F$$ is a homotopy equivalence.   We adapt this approach for use in this paper.  

The proof makes use of the following general facts about adjoint pairs of functors.   

\begin{lem}Let $(L,R): {\mathcal A}\rightarrow {\mathcal B}$ be a pair of adjoint functors where $L:{\mathcal A}\rightarrow {\mathcal B}$ is the left adjoint.  Let $\perp=LR: {\mathcal B}\rightarrow {\mathcal B}$ be the associated cotriple and $B$ be an object in ${\mathcal B}$.   Then there
are natural simplicial maps  $f: R({\rm id} ^{*+1}B)\rightarrow R(\perp ^{*+1}B)$ and $\upsilon: R(\perp ^{*+1}B)\rightarrow R({\rm id} ^{*+1}B)$ such that $f\circ \upsilon$ is naturally homotopic to the identity on $R({\perp} ^{*+1}B)$ and $\upsilon\circ f$ is naturally homotopic to the identity on $R({\rm id}^{*+1}B).$
\end{lem}  

\begin{proof}   Let $\eta : {\rm id}_{\mathcal A}\rightarrow RL$ be the 
unit of the adjunction and $\epsilon: LR \rightarrow {\rm id}_{\mathcal B}$ be the counit of the adjunction.    One can use  $\epsilon$ and $\eta$  to construct the simplicial maps $f: R({\rm id} ^{*+1}B)\rightarrow R(\perp ^{*+1}B)$ and $\upsilon: R(\perp ^{*+1}B)\rightarrow R({\rm id} ^{*+1}B)$,  and simplicial homotopies between $f\circ \upsilon$ and $\upsilon \circ f$ and the appropriate identity maps.  
  In particular, in degree $k$, $\upsilon _k=R(\epsilon\circ \epsilon _{\perp}\circ \dots \circ \epsilon _{\perp^k})$ and $f_k=\eta _{R(LR)^k}\circ f_{k-1}$ with $f_0=\eta _R$.  One can construct the homotopies in a similar fashion to that of  Exercise 8.3.7 of \cite{We}, using $\eta _R$ in place of the extra degeneracy $\sigma _0$. 
  \end{proof}
With this, we prove Proposition \ref{degPn}.  
\begin{proof} 
Applying the above to the adjoint pair $(\Delta ^*\circ U^+, t^+\circ \sqcup _{n+1})$ of the previous 
section, we see that  $(t^+\circ \sqcup_{n+1}) \perp _{n+1}^{*+1}F$ is weakly equivalent 
to $t^+\circ \sqcup_{n+1}F$ via the augmentation $t^+\circ \sqcup _{n+1}(\epsilon)$.   Setting $cr_{n+1}^+F=t^+\circ \sqcup_{n+1}F$, this becomes $cr_{n+1}^+(\perp _{n+1}^{*+1}F)\simeq
cr_{n+1}^+({\rm id}^{*+1}F)$.   Applying the forgetful functor from $\Fun(\Cfn {n+1},\s)_t$ to $\Fun(\Cfn {n+1},\s)$ (this functor ``forgets'' the coalgebra structure)  gives us an equivalence between $cr_{n+1}F$ and $cr_{n+1}\perp _{n+1}F$.  The fact that $cr_{n+1}$, as a finite homotopy limit, commutes
with finite and filtered homotopy colimits implies that
\begin{align*}
cr_{n+1}\Gamma_nF&=cr_{n+1}(\operatorname{hocofiber}(|\perp _{n+1}^{*+1}F|\rightarrow {|\rm id}^{*+1}F|))\\
&\simeq \operatorname{hocofiber}(|cr_{n+1}(\perp _{n+1}^{*+1}F)|\rightarrow {|cr_{n+1}(\rm id}^{*+1}F)|))\\
&\simeq \star.
\end{align*}
Hence, $\Gamma_nF$ is degree $n$.
\end{proof}
We consider the functor $\Gamma_nF$ to be an ``approximation" to $F$ in the following sense.

\begin{prop}\begin{enumerate}
\item {If $F$ is degree $n$, then the natural transformation $\gamma_nF: F\rightarrow \Gamma_nF$ is a weak equivalence.}
\item {The pair $(\Gamma_nF, \gamma_nF)$ is universal, up to weak equivalence, among degree $n$ functors with natural transformations from $F$.}  \end{enumerate}
\end{prop}
The proofs of these results are similar to those of Proposition 1.18 in \cite{G3} and Lemma 2.11 in \cite {RB}  and are omitted.

We end this section by defining natural transformations $\Gamma_nF\rightarrow \Gamma_{n-1}F$ that
allow us to assemble the $\Gamma_nF$s into a Taylor tower for $F$.  

\begin{defn}
Let $X$ be an object in $\Cf$.  We define the $(n+1)$-cube of $n$-tuples, ${\bf X'}: {\mathcal P}({\bf n})\rightarrow \Cf ^{\times n}$, as follows.  For $i\neq 1$ and $S\subseteq {\bf n+1}$, the  $i$th entry in ${\bf X'}(S)$ is 
\[
{\bf X'}(S)_i=\begin{cases}X&\text{if $\{2, i+1\}\cap S=\emptyset$}\\  B &\text{otherwise,}

\end{cases}
\]
and, for $i=1$, we have
\[ {\bf X'}(S)_1=\begin{cases} X\coprod _A X&\text{if $\{1,2\}\cap S=\emptyset$}\\
B\coprod _AB&\text{otherwise.}
\end{cases}
\]
The morphisms are all induced by the morphism $X\rightarrow B$.  
\end{defn}  

Recalling the definition of $(\sqcup ^{n+1})^{(X, X, \dots, X)}_B$ from Example \ref{e:chi}, it is easy to verify that there is a natural map of $(n+1)$-cubes $\tau _n: (\sqcup^{n+1})^{(X, X, \dots, X)}_B\rightarrow
\sqcup ^{n}{\bf X'}$.   Applying a functor $F$ and taking the total fiber of the resulting $(n+1)$-cubes, yields a natural map
$$
\rho _n: \perp _{n+1}F(X)\rightarrow \operatorname{tfiber}F(\sqcup^{n}{\bf X'}).
$$
But, by construction, letting ${\bf X''}$ denote the restriction of ${\bf X'}$ to $S\subseteq {\bf n+1}$
with $2\notin S$, we see that
\begin{align*}\operatorname{tfiber}F(\sqcup ^{n}{\bf X'})&\simeq \operatorname{hofiber}(\operatorname{tfiber}F({\sqcup}^{n}{\bf X''}(-))\rightarrow \operatorname{tfiber}F({\sqcup}^{n}{\bf X''}(-\cup \{2\}))\\
&\simeq \operatorname{hofiber}(\operatorname{tfiber}F({\sqcup}^{n}{\bf X''}(-))\rightarrow \star)\\
&\simeq \operatorname{tfiber}F({\sqcup}^{n}{\bf X''}(-)),
\end{align*}
and we have a natural map
\[
\iota _n: \operatorname{tfiber}F(\sqcup ^{n}{\bf X'})\buildrel \simeq\over \longrightarrow \operatorname{tfiber}F(\sqcup ^{n}{\bf X''}).
\]
Applying the fold map $+:Y\coprod _A Y\rightarrow Y$ (with $Y$ equal to $X$ or $B$) to the first pair of terms in the coproducts produces a natural map of $n$-cubes that yields
\[
\sigma _n: \operatorname{tfiber}F(\sqcup ^{n}{\bf X''})\rightarrow \perp_{n}F(X).
\]
We define $\nu _n:\perp _{n+1}F(X) \rightarrow \perp_{n}F(X)$ to be the natural 
composition 
$$
\nu _n=    \sigma _n      \circ   \iota _n   \circ \rho _n.
$$
This gives us a map of simplicial objects $\perp _{n+1}^*F\rightarrow \perp _n^*F$ that can be used to construct the natural transformation $q_nF: \Gamma_nF\rightarrow \Gamma_{n-1}F$.   With this map we obtain the desired Taylor tower.   
\begin{thm} \label{t:tower}There is a natural tower of functors:
\[ \xymatrix{ && F \ar[dl]_{\gamma_{n+1}}\ar[d]^{\gamma_n}\ar[dr]^{\gamma_{n-1}} \\
\cdots \ar[r] & \Gamma_{n+1}F \ar[r]^{q_{n+1}} & \Gamma_nF \ar[r]^{q_n\quad} & \Gamma_{n-1}F \ar[r] ^ \cdots \ar[r] & \Gamma_1F \ar[r] & \Gamma_0F.}\]
\end{thm}


\section{A comparison to Goodwillie's $n$-excisive approximation}

We use this section to compare the degree $n$ approximation, $\Gamma _nF$, of a functor $F:\Cf\rightarrow\s$ to Goodwillie's $n$-excisive approximation, $P_nF$.   After reviewing the definition of $P_nF$, we first show that when $F$ commutes with realizations, 
$P_nF$ and $\Gamma_nF$ are weakly equivalent as functors from $\Cf$ to $\s$.    When $F$ 
does not commute with realizations, we also obtain agreement of the functors $\Gamma_nF$
and $P_nF$, but only when evaluated at the initial object of $\Cf$.   We conclude the section by 
using this fact to show that for an object $X$, the degree $n$ and $n$-excisive approximations
to $F$ at $X$ agree, but only after restricting the constructions (in a sense to be made clear later)
to $\C _{\beta:X\rightarrow B}$.   Throughout this section $\C$ and $\D$ are used to denote
simplicial model categories, and $\s$ is a suitable model of spectra as in section 4.  As in previous sections, for a morphism $f:A\rightarrow B$ in $\C$, we use $\Cf$ to denote the category of objects of $\C$ that factor $f$, and we use $\star$ to denote the initial/final object in $\s$.     While the proofs of the  two main theorems of 
the section require that our functors take values in $\s$, some of the lemmas used to prove them are stated and proved more generally for functors from $\C$ to $\D$.

\subsection{Goodwillie's $n$-excisive approximations}  

To define the $n$-excisive approximation to a functor $F$, we make use of the following
construction for objects $X$ in $\Cf$.

\begin{defn}  For a finite set $U$ of cardinality $u$,  let ${\mathcal D}_U$ be the poset obtained from  $\P(U)$ by restricting to the 
empty set and one element subsets of $U$.   That is, the set of objects of ${\mathcal D}_U$  is $\{\emptyset, \{t\}\ |\ t\in U\}.$   For an object $X$ in $\Cf$ with $\beta:X\rightarrow B$, we define a functor $X_U:{\mathcal D}_U\rightarrow \Cf$ by 
\[X_U(S)=\begin{cases}X&\ \text{if}\  S=\emptyset,\\ B&\ \text{if}\  S\neq \emptyset, \end{cases}
\]
and  $X_U(\emptyset \rightarrow \{t\})=\beta$.  
We define $B\otimes _X-$ to be the functor from finite sets to 
$\Cf$ that for the finite set $U$ is given by 
\[
B\otimes_X U:=\underset{S\in {\mathcal D}_U}{\rm hocolim}  X_U(S).
\]
\end{defn}

\begin{rem}{\label {r:tensor}} Suppose that $U$ has $u$ elements.  
\begin{enumerate}   
\item If we assume that  $\beta: X\rightarrow B$ is a cofibration, then $B\otimes_X U$  is the coproduct over $X$ of $u$ copies of $B$.    In particular, if $U$ is the empty set, then
$B\otimes _X U\simeq X$, and if $U$ has a single element, then $B\otimes _XU\simeq B$.

\item Let ${\mathcal U}_Y$ be the category of  unbased spaces over a fixed space $Y$ used by  Goodwillie. The initial and final objects are $\emptyset$ and $Y$, respectively.  When $X$ is an object  in ${\mathcal U}_Y$, the space $Y\otimes _{X} U$ is equivalent to the fibrewise join of $X$ with $U$ over $Y$, denoted $X*_Y U$, that Goodwillie uses to define $P_nF$.   
Recall that $X*_YU$ is defined as
\[{\rm hocolim}(X\leftarrow X\times U\rightarrow Y\times U).
\]
One can see that $Y\otimes_{X}U$ and $X*_YU$ are equivalent either by directly 
comparing the homotopy colimits used to define them or by noting that both constructions yield 
$u$ copies of the mapping cone of $X\rightarrow Y$ identified together along a single 
copy of $X$.
\end{enumerate}   
\end{rem}
To define $P_nF(X)$ for a functor of ${\mathcal U}_Y$ and object $X$ in ${\mathcal U}_Y$, Goodwillie defines an intermediate functor $T_nF$, by
\[
T_nF(X):=\underset {U\in \P_0({\bf n+1})}{\rm holim} F(X*_Y U).
\]
There is a  natural map $F(X)\simeq F(X*_Y \emptyset)\rightarrow T_nF(X)$,
and iterating yields a sequence
\[F(X)\rightarrow T_nF(X)\rightarrow T_n^2F(X)\rightarrow T_n^3F(X)\rightarrow \dots .\]
The functor $P_nF$ is defined as the homotopy colimit of this sequence: 
\[
P_nF(X):={\rm hocolim}_k T_n^kF(X).
\]  See section 1 of \cite{G3} for more details.  

In light of Remark \ref{r:tensor}.2, we extend Goodwillie's definition of $P_nF$ to functors from $\Cf$ to $\D$ as follows.
\begin{defn}\label{d:Pn}
Let $F:\Cf \rightarrow \D$, and let $X$ be an object in $\Cf$ (with canonical map $\beta: X\rightarrow B$).   Let $T_nF(X)$ be defined by 
\[T_nF(X):=\underset{U\in \P_0({\bf n+1})}{\rm holim} F(B\otimes _X U).
\]
Then $P_nF(X)$ is given by 
\[P_nF(X):={\rm hocolim}_k(T_n^kF(X)).
\]
\end{defn}
\noindent N. Kuhn  has also defined $P_nF$ for functors of pointed simplicial or topological model categories (\cite {Kuhn}).   
As an immediate consequence of Definition \ref{d:Pn}, we have the next lemma. 
\begin{lem}{\label{l:pn}}For $F$ a functor from $\C_f$ to $\s$, the construction $P_n$ satisfies the following properties:\begin{enumerate}
\item Let $\{F_i\}, {i\in {I}},$ be a finite diagram of functors from $\Cf$ to $\s$.  Then 
\[\underset{i\in {I}} {\rm holim}P_nF_i\simeq P_n(\underset {i\in {I}}{\rm holim}F_i)\]
as functors from $\Cf$ to $\s$.

\item Let $G_{\cdot}$ be a simplicial object in the category of functors from $\Cf$ to $\s$.   Then
\[P_n|G_{\cdot}|\simeq |P_nG_{\cdot}|
\] 
where $P_n$ is applied levelwise to $G_{\cdot}$.  
\item  Cofibration sequences of functors from $\Cf$ to $\s$ are preserved by $P_n$.  
\item For $F:\Cf \rightarrow {\mathcal D}$, there is natural transformation $F\rightarrow P_nF$.  When $F$ is $n$-excisive, this is a weak equivalence.  
\end{enumerate}  \end{lem}
\begin{proof}
The first part of the lemma follows from the definition of $P_n$ and the facts that 
homotopy limits commute  and filtered countable homotopy colimits commute with finite homotopy limits in $\s$. 
To see that the second part is true, note that the fact that homotopy colimits commute tells us that for an object $X$ in $\Cf$,
\[|P_nG_{\cdot}(X)|:= |{\rm hocolim}_kT^k_nG_{\cdot}(X)|\simeq {\rm hocolim}_k|T^k_n G_{\cdot}(X)|.\]
So, it remains to show that 
\[|T^k_nG_{\cdot}(X)|\simeq T^k_n|G_{\cdot}|(X),
\]
where, by definition,  $T^k_n|G_{\cdot}|(X)$ is the finite homotopy limit of a diagram of spectra.   In general, 
homotopy limits do not commute with (fat) realizations, but in this case we have a finite homotopy limit in $\s$. The third part is proved using similar arguments.  

The natural transformation of the fourth part is the transformation from $F$, the initial object 
of the sequence defining $P_nF$, into the homotopy colimit of that sequence.   To understand what happens when $F$ is $n$-excisive, note that for an object $X$, the $(n+1)$-cube
\[
U\in \P({\bf n+1}) \mapsto B\otimes _XU
\]
is a strongly cocartesian diagram.   Applying $F$ yields a cartesian diagram, and so we have 
an equivalence
\[F(B\otimes _X\emptyset)\buildrel \simeq\over \longrightarrow \underset{U\in \P_0({\bf n+1})} {\rm holim}F(B\otimes _XU),
\]  
but this is simply $F(X)\buildrel\simeq\over\longrightarrow T_nF(X)$.  The result follows.   \end{proof}

\subsection{Functors that commute with realizations}

Given a functor $F:\Cf \rightarrow \s$ we seek to show that $P_nF$ and $\Gamma _nF$ agree as functors of $\Cf$ when $F$ commutes with realizations.   This is achieved via the next theorem.

\begin{thm}  {\label {t:commute}}  Let $F:\Cf \rightarrow \s$ be a functor that commutes with realizations.   Then there is a (co)fibration sequence of functors
\[
|\perp _{n+1}^{*+1}F|\rightarrow F\rightarrow P_nF.
\]
\end{thm}
\begin{proof}
Consider the cofibration sequence used to define $\Gamma_nF$:
\[|\perp _{n+1}^{*+1}F|\rightarrow F\rightarrow \Gamma _nF.
\]
Applying $P_n$ to the cofibration sequence and using the natural transformation of Lemma \ref{l:pn}.4  gives us the commutative diagram below:
\begin{equation}\label{e:fibration}\xymatrix{|\perp _{n+1}^{*+1}F| \ar[d] \ar[r] &F \ar[d] \ar[r] &\Gamma _nF \ar[d]\\
P_n|\perp _{n+1}^{*+1}F| \ar[r] &P_nF \ar[r] &P_n\Gamma _nF.
}
\end{equation}
Both rows of this diagram are cofibration sequences, the first by definition, and the second because $P_n$ preserves cofibrations.    Consider the top row of the diagram.   By definition,
$\perp _{n+1}F(X)={\rm tfiber}(F(\sqcup ^n)_B^{(X,\dots, X)}))$.   Since $\perp_{n+1}F$ is formed by a finite homotopy inverse limit, and since $F$ commutes with realizations and finite homotopy inverse limits in $\s$ commute with realizations, $\perp_{n+1}F$ commutes with realizations as well. As a result, we see that each functor in the top row commutes with realizations.   In particular, $\Gamma _nF$ commutes with realizations, and so,
by Propositions \ref{p:realizations} and \ref{degPn}, $\Gamma _nF$ is  $n$-excisive.   By Lemma \ref{l:pn}.4, the rightmost map in (\ref{e:fibration}) is an equivalence.    This tells us that the 
square
\[\xymatrix{|\perp _{n+1}^{*+1}F| \ar[d] \ar[r] &F \ar[d] \\
P_n|\perp _{n+1}^{*+1}F| \ar[r] &P_nF 
}
\]
is cocartesian.
To finish the proof, we show  $P_n|\perp _{n+1}^{*+1}F|\simeq \star$, the initial/final object in $\s$.    To do so, note that  Lemma \ref{l:pn}.2 gives us 
\[P_n|\perp _{n+1}^{*+1}F|\simeq |P_n\perp _{n+1}^{*+1}F|,
\]
and so it is enough to show that $P_n\perp _{n+1}^{*+1}F\simeq \star$
by proving that  $P_n\perp _{n+1}F\simeq \star.$   We do so in Corollary \ref{l:perpreduced} below.   
\end{proof}

To prove Corollary \ref{l:perpreduced}, we use the following lemma, which is Lemma 3.1 of  \cite{G3}. 
In stating the lemma we  use the notion of weakly $n$-reduced functors and the functor $\Delta^*$ from 
Definitions \ref{d:weaklyreduced} and \ref{d:sqcup}.
\begin{lem}
Let $G:\Cfn n\rightarrow {\mathcal D}$ be a weakly $n$-reduced functor.   Then $P_{n-1}\Delta^*G\simeq \star.$
\end{lem}
\begin{proof}
We prove this by first showing that the natural transformation from $\Delta^*G$ to $T_{n-1}\Delta^*G$ factors through $\star$.   To do so,  for an object $X\in \Cf$, we consider the $n$-cube $\widetilde G(X)$:
\[
U\in \P({\bf n})\mapsto \widetilde G(X)(U)=G(X_1(U), X_2(U), \dots, X_n(U))
\]
where 
\[X_i(U)=\begin{cases}B\otimes_X\emptyset\simeq X &\ \text{if}\ i\notin U\\
B\otimes _X\{i\}\simeq B&\ \text{if}\ i\in U.\end{cases}
\]
We also consider the $n$-cube $\widetilde TG(X)$ used to define $T_{n-1}\Delta^*G(X)$:
\[
U\in \P({\bf n})\mapsto \widetilde TG(X)(U)=\Delta^*G(B\otimes_XU).
\]
The inclusions $\{i\}\rightarrow U$ and $\emptyset\rightarrow U$ induce a map of 
$n$-cubes $\widetilde G(X)\rightarrow \widetilde TG(X)$.  Moreover, the map
$\Delta^*G(X)\rightarrow T_{n-1}\Delta^*G(X)$ factors as 
\begin{align*}
\Delta^*G(X)\simeq \widetilde G(X)(\emptyset)&\rightarrow \underset {U\in \P_0({\bf n})}{\rm holim}
\widetilde G(X)(U)\\
&\rightarrow \underset{U\in \P_0({\bf n})} {\rm holim}\widetilde TG(X)(U)\simeq T_{n-1}\Delta^*G(X).
\end{align*}
Since $G$ is weakly $n$-reduced, $\widetilde G(X)(U)\simeq \star$ for $U\neq \emptyset$,
and so we see that $\Delta^*G(X)\rightarrow T_{n-1}\Delta^*G(X)$ factors through $\star$.  In a similar fashion, we can show that $T_{n-1}^k\Delta^G(X)\rightarrow T_{n-1}^{k+1}\Delta^G(X)$
factors through $\star$ for all $k$ and obtain the result.  
\end{proof}
Noting that $\perp_{n+1}F=\Delta^* cr_{n+1}F$ where $cr_{n+1}F$ is a weakly $(n+1)$-reduced functor (by Corollary \ref{c:crnreduced}), we obtain the desired corollary.  
\begin{cor}{\label{l:perpreduced}}  For a functor $F:\Cf\rightarrow {\mathcal D}$, $P_n\perp _{n+1}F\simeq \star$.  
\end{cor}   

The following corollary is an immediate
consequence of Theorem 6.5.

\begin{cor}
Let $F: \Cf\rightarrow \s$ be a functor that commutes with realizations.   Then $P_nF$ and 
$\Gamma_nF$ are weakly equivalent as functors from $\Cf $ to $\s$.  
\end{cor}
\subsection{Functors that do not commute with realizations} 

If $F$ does not commute with realizations, the functors $P_nF$ and $\Gamma _nF$ no longer
agree on all objects in $\Cf$, but they still agree at the initial object $A$.   We establish this 
fact in the next theorem.

\begin{thm} \label{t:agreeatA}  Let $F:\Cf \rightarrow \s$ where $\Cf$ is the category of objects factoring the morphism $f:A\rightarrow B$.   Then $\Gamma _nF(A)\simeq P_nF(A)$.   
\end{thm}

We prove this theorem in a manner similar to that of Theorem \ref{t:commute} and its corollary.   In particular, we use the commutative diagram (\ref{e:fibration}) evaluated at $A$, 
\begin{equation}\label{e:fibrationA}\xymatrix{|\perp _{n+1}^{*+1}F(A)| \ar[d] \ar[r] &F(A) \ar[d] \ar[r] &\Gamma _nF(A) \ar[d]\\
P_n|\perp _{n+1}^{*+1}F(A)| \ar[r] &P_nF(A) \ar[r] &P_n\Gamma _nF(A).
}
\end{equation}
In this case, we will use the following lemmas to prove that the right vertical 
map is an equivalence and that $P_n|\perp _{n+1}^{*+1}F(A)|\simeq \star$.   The second  lemma is a consequence of the key observation in the first lemma.   

\begin{lem}\label{l:Tn}
Let $F:\Cf \rightarrow \s$.   Then
\[\perp _{n+1}F(A)\rightarrow F(A)\rightarrow T_nF(A)
\]
is a fibration sequence in $\s$.
\end{lem}
A generalization of this result for $T_n^kF(A)$, $k\geq 1$,  is the main result of  Appendix B.  
\begin{proof} Recall our assumption that $A\buildrel f\over\rightarrow B$ is a cofibration.  
Because we are evaluating at the initial object in $\Cf$, the coproducts used to define $T_nF(A)$ are taken over $A$, as are the coproducts used in the $(n+1)$-cubical diagram whose total fiber defines $\perp _{n+1}F$.   In particular, $\perp_{n+1}F(A)$ is the total fiber of the $(n+1)$-cubical diagram that assigns the object $F(A_1(S) \coprod\dots \coprod A_{n+1}(S))$ to the set $S$ where
\[A_i(S)=\begin{cases}A &\ \text{if}\ i\notin S\\
B&\ \text{if}\ i\in S.\end{cases}
\]
Since the coproducts are taken over $A$,  letting $s$ denote the cardinality of $S$, we have \[A_1(S)\coprod \dots \coprod A_n(S)\simeq \overbrace{B\coprod_A\dots\coprod_AB}^{s\ \text{times}}
\simeq  B\otimes _A S.\]
Then, by Remark \ref{r:tfiber}, 
\begin{align*}
\perp _{n+1}F(A)&\simeq{\rm tfiber}\left (S\in \P({\bf n+1})\mapsto F( A_1(S)\amalg \dots \amalg A_{n+1}(S))  \right )\\
&\simeq {\rm tfiber}\left (S\in \P({\bf n+1})\mapsto F(B\otimes _A S)\right )\\
&= {\rm hofiber}\left (F(B\otimes _A \emptyset)\rightarrow {\rm holim}\left ( S\in \P_0({\bf n+1}) \mapsto F(B\otimes _A S)  \right)\right)\\
&\simeq  {\rm hofiber}\left (F(A)\rightarrow T_nF(A)   \right),
\end{align*}
as desired.  

\end{proof}
\begin{lem}\label{l:Fdegreen}
If $F:\Cf \rightarrow \s$ is a degree $n$ functor, then $F(A)\rightarrow P_nF(A)$ is a weak equivalence.  
\end{lem}
%
\begin{proof}  
Since $F$ is degree $n$, $cr_{n+1}F\simeq \star$.   Thus, $\perp _{n+1}F(A)\simeq \star$, and by Lemma \ref{l:Tn}, 
$F(A)\simeq T_nF(A)$.  Moreover, since $T_n$ and $cr _{n+1}$ are both homotopy inverse limit constructions,
and homotopy inverse limits commute, 
\[cr _{n+1}T_nF\simeq T_ncr _{n+1}F\simeq \star.
\]
Hence, $T_nF$ is also a degree $n$ functor and by Lemma \ref{l:Tn}
\[T_nF(A)\simeq T_n^2F(A).
\]
Continuing in this fashion, we see  that 
$$F(A)\buildrel\simeq\over\rightarrow {\rm hocolim}_k(T_n^kF(A))= P_nF(A).$$
\end{proof}

With this, we prove Theorem \ref{t:agreeatA}.

\begin{proof}   We prove this theorem by proving that we have a fibration sequence of spectra
\[ |\perp_{n+1}^{*+1}F(A)|\rightarrow F(A)\rightarrow P_nF(A).
\]
Consider diagram (\ref{e:fibrationA}).  Following the same strategy as in the proof of Theorem \ref{t:commute} we see that it suffices to show that the rightmost vertical arrow is an equivalence and that the object in the bottom left corner is equivalent to $\star$.     The first fact is a consequence of Lemma \ref{l:Fdegreen}.  To prove the second fact, we note that  as was the case for Theorem \ref{t:commute},  $P_n|\perp _{n+1}^{*+1}F|\simeq |P_n\perp _{n+1}^{*+1}F|$
since $F$ takes values in $\s$, and so it is enough to show that $P_n\perp _{n+1}F(A)\simeq \star$.   This 
was done in Corollary \ref{l:perpreduced}.   The result  follows.
\end{proof}

The proof of Theorem \ref{t:agreeatA} relies on the critical observation in the proof of Lemma \ref{l:Tn} that in order to obtain agreement between $\perp_{n+1}F$ and the fiber of $F\rightarrow T_nF$
we must evaluate at the same object over which the coproducts for $\perp_{n+1}F$ are taken. 
This suggests that for general $F$, the cotriple construction can be used in place of Goodwillie's construction only when evaluating at the initial object of the domain category.
In fact, when evaluated at other objects, $\Gamma_nF$ and $P_nF$ can differ greatly. 

\begin{ex} Consider the functor $H_1:{\mathcal U}_*\rightarrow \s$ that takes a space $X$ to the Eilenberg-Mac Lane spectrum associated to its first (singular) homology group.   By Theorem \ref{t:agreeatA}, we know that $\Gamma_1H_1(\emptyset)\simeq P_1H_1(\emptyset)$.   For an arbitrary space $X$, one can show that $\perp_2H_1(X)\simeq \star$, and hence, $\Gamma_1H_1(X)\simeq H_1(X)$.   However, if $X$ is a connected space, then 
\[T_1H_1(X)\simeq {\rm holim}\left (H_1(*\otimes _{X} \{1\})\rightarrow H_1(*\otimes_X\{1,2\})\leftarrow H_1(*\otimes _X\{2\})\right  ) \simeq \star\]
since $*\otimes_{X}\{1\}$ and $*\otimes_{X}\{2\}$ are equivalent to the cone on $X$ and 
$*\otimes_X\{1,2\}$ is equivalent to the (unreduced) suspension of $X$.  As a result, one sees that $P_1H_1(X)\simeq \star$ when $X$ is connected, so $\Gamma_1H_1$ and $P_1H_1$ are not equivalent as functors.      One can obtain similar results for higher homology groups. \end{ex}
 
 This example shows that we cannot guarantee in general that  $\Gamma_nF$ and $P_nF$ agree as functors.  However,  we can show that the particular value $P_nF(X)$  is equivalent to the $n$th term in a cotriple Taylor tower for $F$, albeit not the same tower as used in Theorem \ref{t:agreeatA}.   To do so, we change our focus from $\Cf$, which has a fixed initial and fixed terminal object, to the category $\C_{/B}$ of objects over $B$.  To state the result
 precisely we use the following notation.   
 
 An object of the category $\C_{/B}$ of objects over $B$ is a morphism $\beta:X\to B$ in $\C$. 
Given a functor $F:\C_{/B} \rightarrow \s$, we can restrict it to the category $\C _{\beta}$ determined by objects that factor $\beta:X\to B$ in $\C$.   We use $F^{\beta}$ to denote
this restriction.   We can define the cross effects of $F^{\beta}$ using $X$ as our initial object,  and $B$ as our final object.   We denote these cross effects and the associated cotriples by $cr_n^{\beta}F$ and $\perp_{n}^{\beta}F$, respectively.   Note that even if $\C_{/B}$ had an initial object $A$ (so that it is secretly a category $\Cf$ for some fixed $f:A\to B$), $cr_n^{\beta}F$ is not 
the restriction of $cr_nF$ to $\C_{\beta}$.   For example, for an object $Z$ in $\C_{\beta}$, 
$cr_2F(Z,Z)$ is the total fiber of the 
diagram 
\[
\xymatrix{F(Z\coprod _A Z)\ar[r] \ar[d] & F(Z\coprod _A B)\ar[d]\\
F(B\coprod _A Z)\ar[r]&F(B\coprod_A B)}
\]
whereas $cr_2^{\beta}F(Z,Z)$ is the total fiber of 
\[
\xymatrix{F(Z\coprod _XZ)\ar[r] \ar[d]& F(Z\coprod _X B)\ar[d] \\
F(B\coprod _X Z)\ar[r]  &F(B\coprod _X B).}
\]
We use $\Gamma _{n}^{\beta}F$ to denote the homotopy cofiber of 
\[
|(\perp _{n+1}^{\beta})^{*+1}F|\rightarrow F.
\]

As indicated above, using the map $\beta$ 
in place of $f$ to define our cross effects changes the construction $\Gamma _nF$ to
$\Gamma_n^{\beta}F$.   
However,
this is not the case for $P_nF$ as its definition requires a specific initial object (the empty set, or the point in the base pointed case) and prohibits restriction to the category $\C_{\beta}$.
In light of this, Theorem \ref{t:agreeatA} can be restated to obtain equivalences between $\Gamma_n^\beta F(X)$ and $P_nF(X)$.  Moreover, this equivalence is functorial in $X$.
\begin{thm}{\label{t:lastthm}}
Let $F:\C_{/B}\rightarrow \s$  and let $\beta:X\to B$ 
be an object in $\C_{/B}$.
\begin{enumerate}
\item There is a functor $\Gamma_n^{(-)}F: \C_{/B} \to \s$ taking $\beta:X\to B$ to $\Gamma_n^{\beta}F(X)$.
\item    There is a natural weak equivalence
\[
\Gamma _n^{\beta}F(X)\simeq P_nF(X).
\]
\end{enumerate}
\end{thm}
\begin{proof}  Let $\beta:X\to B$ and $\beta':Y\to B$ in $\C_{/B}$.  The functoriality of $\Gamma_n^{(-)}F$ comes from the fact that a map $\sigma:X\to Y$ in $\C_{/B}$ induces enough maps between coproducts over $X$ and $Y$, respectively, to induce maps between $\perp_n^{\beta}F(X)$ and $\perp_n^{\beta'}F(Y)$.  The rest follows from Theorem \ref{t:agreeatA}.
\end{proof}

\section{Convergence of the cotriple tower}


This section identifies criteria that guarantee convergence of our tower.   Our first  goal is to show that our tower converges for analytic functors, just as Goodwillie's tower does.   In the second part of this section we also identify  conditions  on cross effects that guarantee convergence.  In this section we work with functors $F$ whose target category $\s$ is a category of spectra. Since analyticity requires a notion  of connectivity in both the domain and target categories,  we restrict ourselves to the setting of Goodwillie's calculus of homotopy functors for the results concerning analyticity.   In particular, we let ${\mathcal T}$ denote the category of  topological spaces and ${\mathcal T}_{g}$ denote the category determined by the morphism $g:C\rightarrow D$.  For the results in the second part of this section, we will work with functors 
$F:\Cf\rightarrow \s$ where $\C$ is a simplicial model category and  $\Cf$ is determined by a fixed morphism  $f:A\rightarrow B$ in $\C$.    Throughout this section we use $P_{\infty}F$ to denote the homotopy inverse limit of the Goodwillie tower for $F$ and $\Gamma_{\infty}F$ to denote the homotopy inverse limit of the cotriple tower.   We say that the Goodwillie tower for $F$ {\it converges at X} if the natural map $F(X)\rightarrow P_{\infty}F(X)$ is a weak equivalence.   Convergence of the cotriple tower is defined analogously.

\subsection{Analyticity and Convergence}   

As in \cite{G2}, by a $k$-connected map of spaces we mean a map  whose homotopy fibers are all $(k-1)$-connected.    
Recall the following definitions from \cite{G2}.

\begin{defn}  (\cite{G2}, 1.3)  The $n$-cubical diagram ${\mathcal X}$ in ${\mathcal T}$ or $\s$ is $k$-cartesian provided that the map from ${\mathcal X}(\emptyset)$ to $\holim _{S\in {\mathcal P}_0({\bf n})}{\mathcal X}(S)$ is $k$-connected.   
\end{defn}

\begin{defn} (\cite{G2}, 4.1)  A functor $F:{\mathcal T}\rightarrow \s$ is {\it stably $n$-excisive}, if the following is true for some numbers $c$ and $\kappa$:

\noindent $E_n(c,\kappa)$: If ${\mathcal X}: {\mathcal P}({\bf n+1})\rightarrow \Cf$ is any strongly cocartesian $(n+1)$-cube such that for all $s\in {\bf n+1}$ the map ${\mathcal X}(\emptyset)\rightarrow {\mathcal X}(\{s\})$ is $k_s$-connected and $k_s\geq \kappa$, then the diagram $F({\mathcal X})$ is $(-c+\sum k_s)$-cartesian.
\end{defn}

\begin{defn} (\cite{G2}, 4.2)
A functor $F:{\mathcal T}\rightarrow \s$ is $\rho$-analytic if there is some number $q$ such that $F$ satisfies $E_n(n\rho-q,\rho +1)$ for all $n\geq 1$.
\end{defn}

In \cite{G3}, Goodwillie showed that his tower for $F$ converges at $X$ when $F$ is $\rho$-analytic and $X\rightarrow D$ is at least $(\rho +1)$-connected.     We establish a similar result for our tower below, using the next lemma.   

\begin{lem}
If $F:{\mathcal T}_g\rightarrow \s$ is $\rho$-analytic and $\beta: X\rightarrow D$ is at least $(\rho +1)$-connected, then $\perp _{n+1}^kF(X)$ is at least $(q+(n+1)^k-1)$-connected.   Here $q$ is the constant such that $F$ satisfies $E_n(n\rho -q, \rho+1)$.   
\end{lem}

\begin{proof}  Since $\perp _{n+1}F(X)$ is the total fiber of a strongly cocartesian $(n+1)$-cube in ${\mathcal T}_g$, $\perp ^k_{n+1}F(X)$ can be described as the total fiber of a strongly cocartesian $(n+1)^k$-cube ${\mathcal Y}$.   For each $S\in {\mathcal P}({\bf (n+1)^k})$, ${\mathcal Y}(S)$ is 
a coproduct of copies of $X$ and $D$ over $C$ and each morphism is a coproduct of copies of $\beta$ and identity morphisms.   In particular, this ensures that ${\mathcal Y}(\emptyset)\rightarrow {\mathcal Y}(\{s\})$ is at least $(\rho +1)$-connected for each 
$s\in {\bf (n+1)^k}$.   
The analyticity condition then guarantees that the total fiber of $F({\mathcal Y})$ is at least $(q+(n+1)^k-1)$-connected.   The result follows.  
\end{proof}

\begin{prop}
If $F:{\mathcal T}_g\rightarrow \s$ is $\rho$-analytic, satisfying $E_n(n\rho -q, \rho+1)$ for each $n$, and $X\rightarrow B$ is at least $(\rho+1)$-connected, then $F(X)\rightarrow \Gamma _nF(X)$ is at least $(q+n+1)$-connected.    As a consequence, 
$F(X)\simeq \Gamma_{\infty}F(X)$.  
\end{prop}

\begin{proof}
By the lemma, we know that in each simplicial degree $\perp_{n+1}^{*+1}F(X)$ is at least $(q+(n+1)^1)-1)$-connected.   Since homotopy colimits preserve connectivity, it follows that the realization of $\perp_{n+1}^{*+1}F(X)$ is at least $(q+(n+1)^1-1)$-connected.   The result follows using the fact that in spectra, 
\[|\perp _{n+1}^{*+1}F(X)|\rightarrow F(X)\rightarrow \Gamma_nF(X)
\]
is also a fibration sequence.  
\end{proof}

\subsection{Cross effects and convergence.}

Mimicking the convergence results for the abelian case found in section 4 of {\cite {RB}}, we can place conditions on $\perp _{n+1}F$ that guarantee convergence of the cotriple tower.

\begin{prop}{\label{p:crconv}}  Let $F:\Cf\rightarrow \s$ be a functor that takes values in connective spectra.   Let $X$ be an object in $\Cf$.  Suppose that there is a $c\geq 0$ such that for $1\leq t\leq c+1$,    
 $\perp _{n+1}^tF(X)$ is at least $(c+1-t)$-connected.  Then 
$\gamma_nF:F(X)\rightarrow \Gamma_nF(X)$ is  $(c+1)$-connected.   
\end{prop}
\begin{proof}
The connectivity condition is sufficient to ensure that $|\perp _{n+1}^{*+1}F(X)|$ is at least $c$-connected.  To see this one can use the spectral sequence associated to the simplicial spectrum $\perp_{n+1}^{*+1}F(X)$ that has 
$$
E^1_{p,q}=\pi _p (\perp_{n+1}^{q+1}F(X))
$$  and converges to $\pi_{p+q}(|\perp_{n+1}^{*+1}F(X)|)$  (see, e.g., \cite{Dug} for details).    The connectivity condition guarantees that $E_{p,q}^1\cong 0$ for $p+q\leq c$.   Hence, $|\perp _{n+1}^{*+1}F(X)|$ must be at least $c$-connected.  The result follows again by using the fibration sequence 
\[|\perp _{n+1}^{*+1}F(X)|\rightarrow F(X)\rightarrow \Gamma_nF(X).
\]  
\end{proof}
The condition that $\perp_{n+1}^tF(X)$ is at least $(c+1-t)$-connected for $1\leq t\leq c$ is like that of  stable 
$n$-excision in Goodwillie's calculus, in that the condition guarantees a certain connectivity of the map $F(X)\rightarrow \Gamma_nF(X)$.   We use this condition to guarantee convergence of our tower as follows.

\begin{defn}
For a functor $F:\Cf\rightarrow \s$, object $X$ in $\Cf$, and $n\geq 0$, we set 
$$
{F_{\mathrm {con}}(X,n)={\mathrm {max}}\{c\in {\mathbb Z}
\ |\ {\mathrm {conn}}(\perp ^t_{n+1}F(X))\geq c+1-t\ {\mathrm {for}}\ 1\leq t\leq c+1\}.
}$$
\end{defn}
With this, the next proposition is an immediate consequence of Proposition \ref{p:crconv}.

\begin{prop}  Let $F:\Cf\rightarrow Spec$ be a functor that takes values in connective spectra.  
If $\lim _{n\rightarrow \infty}F_{\mathrm {con}}(X,n)=N$, then $F(X)\rightarrow \Gamma_{\infty}F(X)$ is $(N+1)$-connected.   When $N=\infty$, the cotriple tower converges for $X$.   
\end{prop}
\newpage

\appendix

\section{Proofs of homotopy limit properties}
In this appendix, we address the four properties of homotopy inverse limits in Lemma \ref{l:holim} that are essential to our proof in section 3 that $t$ is a cotriple.   Although these properties are widely accepted as being true, we had difficulty finding proofs of them in the literature.   As they are critical to our work in section 3, we include their proofs in this appendix.  
For $\mathscr{C}$ a category, let ${\rm hom}_{\mathscr{C}}(A,B)$ be the \textit{set} of morphisms between $A,B \in \text{ob}(\mathscr{C})$.  Recall that a simplicial model category also has a simplicial set of maps. For objects $A$ and $B$, we use ${\text {Hom}}_{\mathscr{C}}(A,B)$ to denote this and note that  $({\text {Hom}}_\mathscr{C} (A,B))_n := {\rm hom}_{\mathscr{C}} (A \ox \Delta[n], B)$.

Let ${\sc M}$ be a simplicial model category and let ${\sc C}$ be a small category. 
Following Hirschhorn [H], for $X$ a ${\cal C}$ diagram in ${\cal M}$ and $K$ a ${\cal C}$-diagram in simplicial sets, define 
\[\text{hom}^{\cal C}(K,X) = 
\text{equalizer} \left(
\prod_{C\in{\cal C}} (X_C)^{K(C)}\mapright . \prod_{f:C\rightarrow C' \in {\cal C}} (X_{C'})^{K(C)}\right)
\]
where $K(C)$ is an object  in  $SS$, the category of simplicial sets.  The above makes use of the fact that for $M \in {\rm ob}({\cal M})$ and $K \in {\rm ob}(SS)$, there is an object $M^K$ in ${\cal M}$.

When we let $K = \N({\cal C}\downarrow *)$ (the classifying space or nerve of the over category ${\cal C}\downarrow *$, see [H, 14.1.1]) , we define 
\[\text{holim}_{\cal C}X = \text{hom}^{\cal C}(\N({\cal C}\downarrow *),X(*)).\]
As noted on page 379 of [H], this model for $\text{holim}_{\cal C}X$ is homotopy invariant only when $X$ is objectwise fibrant. When $X$ is objectwise fibrant, 
by Corollary 18.5.2(2) of [H], $\text{holim}_{\cal C}X$ is a fibrant object of ${\cal M}$. 

To establish the desired homotopy inverse limit properties, we make use of the following important adjunction.   This adjunction is stated and proved for morphism sets in [H], but it is straightforward to generalize the proof to simplicial mapping spaces to obtain the version stated below.   

\bigskip\noindent
{\bf Proposition 18.3.10(2) of [H]:} 
Let $X$ be a ${\cal C}$-diagram in ${\cal M}$, $K$ a ${\cal C}$-diagram of simplicial sets and let $W$ be an object of ${\cal M}$, then there is a natural isomorphism 
\[
{\text {Hom}}_{\cal M}\left(W, \text{hom}^{\cal C}(K,X)\right) \cong
{\text {Hom}}_{SS^{\cal C}} \left( K,{\text {Hom}}_{\cal M}(W,X)\right).
\]
Here $SS$ is the category of simplicial sets and $SS^{\cal C}$ is the category of  ${\cal C}$-diagrams in simplicial sets.  

\bigskip

With this adjunction and the Yoneda Lemma, we can establish the desired properties for this model of holim in ${\cal M}$ by reducing them to properties of diagrams of simplicial sets.  In order to do this, we first recall some facts about diagrams of simplicial sets. 

By Definition 18.2.3(1) and Proposition 18.2.5 of [H], for $A$ and $B$ elements of $SS^{\cal C}$, ${\rm Hom}_{SS^{\C}}(A,B)$ consists of simplicial sets of the form
\[
[n]\mapsto \text{Nat}_{SS^{\cal C}}(A\otimes \Delta[n],B)\] 
That is, the $n$-simplices are simplicial maps from $A(C)\times \Delta[n]$ to $B(C)$ that are natural in ${\cal C}$. 
Since the one point set $\ast$ is terminal, we see that the constant ${\cal C}$ diagram to $\ast$ is terminal in $SS^{\cal C}$. 


Given a functor $\alpha$ from ${\cal D}$ to ${\cal C}$ of small categories and objects $A,B$ in $SS^{\cal C}$, we have a natural map
\[\text{Hom}_{SS^{\cal C}}(A,B)\mapright |_{\alpha}. \text{Hom}_{SS^{\cal D}}(A\circ\alpha,B\circ\alpha)\]
by restricting $f$ to $\alpha({\cal C})$,or $f|_{\alpha}(D) = f(\alpha(D))$.

The last property we want is a generalization of the Exponential Law found in [GJ], Proposition II.5.1, to simplicial mapping spaces for diagrams of simplicial sets. 
Given $X\in {\rm ob}(SS^{\cal D})$ and $Y\in {\rm ob}(SS^{{\cal C}\times {\cal D}})$ there is an evaluation map in $SS^{{\cal C}\times {\cal D}}$
\[\text{Hom}_{SS^{\cal D}}(X(\star), Y(-,\star))\times X \mapright ev. Y\]
given by sending $(f,x_n)\in (\text{Hom}_{SS^{\cal D}}(X(\star), Y(-,\star))\times X)_n$ to $f(x_n, {\text id}_n)$. Using this we obtain the following. 

\bigskip\noindent
{\bf Proposition} (Exponential Law). For objects $K$ in $SS^{\cal C}$, $X$ in $SS^{\cal D}$ and $Y$ in $SS^{{\cal C}\times {\cal D}}$, there is a natural isomorphism 
\[\text{Hom}_{SS^{\cal C}}(K(\ast),\text{Hom}_{SS^{\cal D}}(X(\star),Y(\ast,\star)))\mapright ev_*. \text{Hom}_{SS^{{\cal C}\times {\cal D}}}(K(\ast)\times X(\star),Y(\ast,\star)).\]
{\it Proof.} The function $ev_*$ is defined 
by sending $g: K\times \Delta[n] \rightarrow \text{Hom}_{SS^{\cal D}}(X,Y)$ to the composite
\[K\times X\times \Delta[n] \cong K\times\Delta[n]\times X \mapright g\times 1.\text{Hom}_{SS^{\cal D}}(X,Y)\times X \mapright ev. Y.\]
This is an isomorphism whose inverse is the map defined by sending 
$g:K\times X\times \Delta[n]\rightarrow Y$ to the map $g_*:K\times \Delta [n]\rightarrow \text{Hom}_{SS^{\cal D}}(X,Y)$ where
$g_*$ sends 
\[x\time\tau\in K_m \times \Delta[n]_m \cong (\Delta[m]\mapright i_{(x\times\tau)}. K\times\Delta[n])\]
to the composite
\[X\times \Delta[m]\mapright 1\times i_{(x\times\tau)}. X\times K\times \Delta[n]\cong K\times X\times \Delta[n]\mapright g. Y.\]
Here $\iota_{x\times \tau}$ is as defined on p. 6 of [GJ]

\bigskip
We now establish the four properties of our model for $\text{holim}_{\cal C}X$ for ${\cal M}$ that we need to prove $t$ is a cotriple, that is, we prove  the first four properties of Lemma \ref{l:holim}.

\bigskip\noindent
{\bf Property 1.} Given ${\cal C}$ and ${\cal D}$ small categories and $X$ a ${\cal C}\times {\cal D}$ diagram in ${\cal M}$, there are natural isomorphisms
\[\text{holim}_{\cal C} \text{holim}_{\cal D} X \cong \text{holim}_{\cal C\times \cal D}X
\cong \text{holim}_{\cal D} \text{holim}_{\cal C} X.\]

\bigskip\noindent
{\it Proof.} 
Recall that $\N({\cal C}\times {\cal D}\downarrow *) \cong
\N({\cal C}\downarrow *)\times \N({\cal D}\downarrow *)$ as simplicial sets.  Making repeated use of the two propositions, we obtain the following sequences of isomorphisms of simplicial sets:  

\[\begin{aligned}
\text{Hom}_{\cal M}(W,\text{holim}_{{\cal C}\times{\cal D}}X) 
&=\text{Hom}_{\cal M}(W,\text{hom}^{{\cal C}\times {\cal D}}(\N({\cal C}\times{\cal D}\downarrow *),X)\\
&\cong \text{Hom}_{SS^{{\cal C}\times{\cal D}}}(\N({\cal C}\times{\cal D}\downarrow *),\text{Hom}_{\cal M}(W,X))\\
&\cong \text{Hom}_{SS^{{\cal C}\times{\cal D}}}((\N({\cal C}\downarrow *)\times \N({\cal D}\downarrow *),\text{Hom}_{\cal M}(W,X))\\
&\cong \text{Hom}_{SS^{\cal C}}(\N({\cal C}\downarrow *), \text{Hom}_{SS^{\cal D}}(\N({\cal D}\downarrow *),\text{Hom}_{\cal M}(W,X)))\\
&\cong \text{Hom}_{SS^{\cal C}}(\N({\cal C}\downarrow *), \text{Hom}_{\cal M}(W,\text{hom}^{\cal D}(\N({\cal D}\downarrow *),X))\\
&\cong \text{Hom}_{\cal M}(W,\text{hom}^{\cal C}(\N({\cal C}\downarrow *),\text{hom}^{\cal D}(\N({\cal D}\downarrow *),X)))\\
&\cong \text{Hom}_{\cal M}(W,\text{holim}_{\cal C}\text{holim}_{\cal D} X).\\
\end{aligned}
\]

\if false 
\noindent {\color{blue}  Let $W \in \M$. %
The following isomorphisms are all as sets. 
\[
\begin{array}{rcll}
\M (W, \holim_{\C \x \D} X) &:=& \M (W, \hom^{\C \x \D} (B(\C \x \D \downarrow *), X)  & \text{Our model of holim}\\
&\cong & SS^{\C \x \D}( B(\C \x \D \downarrow *), \M (W, X) )& \text{ By 18.3.10(2) of [H]}\\
&\cong & SS^{\C \x \D}( B(\C\downarrow *) \x B(\D \downarrow *), \M (W, X) ) & \text{ By }(\star)\\
& \cong & SS^{\C}\left( B(\C\downarrow *),  SS^{\D}( B(\D \downarrow *), \M (W, X)) \right)  & \text{ by Exp. Law}\\ 
& \cong &SS^{\C}\left( B(\C\downarrow *),  \M (W, \hom^{\D}( B(\D \downarrow *), X)) \right) &  \text{ By 18.3.10(2) of [H]}\\
& \cong &\M \left( W, \hom^{\C} (B(\C\downarrow *),  \hom^{\D}( B(\D \downarrow *), X)) \right) &  \text{ By 18.3.10(2) of [H]}\\
& \cong &\M(W, \holim_{\C} \holim_{\D} X) &  \text{Our model of holim}\\
\end{array}
\]
}\fi
By the  Yoneda lemma, we obtain the natural isomorphism.  

\bigskip\noindent
{\bf Property 2.} If $\alpha:{\cal C}\rightarrow {\cal D}$ is a functor of small categories, and $X$ is a ${\cal D}$-diagram in ${\cal M}$, we obtain a natural map 
\[\text{holim}_{\cal D}X\rightarrow \text{holim}_{\cal C} (X\circ\alpha).\]

\bigskip
\noindent
{\it Proof}:
Observe that there is a natural transformation of objects in $SS^{\cal C}$
\[\N({\cal C}\downarrow *)\mapright\alpha. \N({\cal D}\downarrow *)\circ \alpha.\]
Then
\[\begin{aligned}
\text{Hom}_{\cal M}(W,\text{holim}_{\cal D} X) & \cong \text{Hom}_{SS^{\cal D}}(\N({\cal D}\downarrow *),\text{Hom}_{\cal M}(W,X))\\
&\rightarrow \text{Hom}_{SS^{\cal C}}(\N({\cal D}\downarrow *)\circ\alpha,\text{Hom}_{\cal M}(W,X)\circ\alpha)\\
&\mapright\alpha^*.\text{Hom}_{SS^{\cal C}}(\N({\cal C}\downarrow *),\text{Hom}_{\cal M}(W,X)\circ\alpha)\\
&\cong \text{Hom}_{SS^{\cal C}}(\N({\cal C}\downarrow *),\text{Hom}_{\cal M}(W,X\circ\alpha))\\
&\cong \text{Hom}_{\cal M}(W,\text{hom}^{\cal C}(\N({\cal C}\downarrow *),X\circ\alpha))\\
&\cong \text{Hom}_{\cal M}(W,\text{holim}_{\cal C}X\circ\alpha)\\
\end{aligned}
\]
and so by Yoneda, we have our natural transformation. 

\bigskip\noindent
{\bf Property 3.} If $T$ is a {\it constant} ${\cal C}$-diagram in ${\cal M}$ with $T(C) = T$ and $T(f) = id_T$ where $T$ is a terminal object of ${\cal M}$, then
\[\text{holim}_{\cal C} T \cong T.\]

\bigskip\noindent
{\it Proof}: 
\[\begin{aligned}
\text{Hom}_{\cal M}(W,\text{holim}_{\cal C}T) &\cong \text{Hom}_{SS^{\cal C}}(\N({\cal C}\downarrow *),\text{Hom}_{\cal M}(W,T))\\
&\cong *\\
\end{aligned}
\]
(this is because $\text{Hom}_{\cal M}(W,T) = *$, the constant one point ${\cal C}$-diagram of simplicial sets, since $T$ is terminal in ${\cal M}$). Since $W$ was arbitrary, $\text{holim}_{\cal C}T$ is terminal and terminal objects are all isomorphic. 

\bigskip\noindent
{\bf Property 4.} If ${\cal I}$ is the trivial category on $i$, then for $X$ any ${\cal I}$-diagram in ${\cal M}$, 
\[\text{holim}_{\cal I} X \cong X(i).\]

\bigskip\noindent
{\it Proof.} This is Proposition 18.3.7(2) of [H].

\section{ \textit{ by Rosona Eldred}}
%

Lemma \ref{l:Tn} may be viewed as the base case of a more general phenomenon. The purpose of this appendix is to explain how $\perp_{n+1}^{*+1}F$ and the tower of functors $\mathrm{T}_n^kF$ defining $\mathrm{P}_nF$ are related by a sequence of fibration sequences. We prove the following proposition.

\begin{prop}\label{prop:alldim} For a functor $F:\Cf \rightarrow \s$, there is a homotopy fiber sequence
$|\mathrm{sk}_k( \perp^{\ast +1}_{n+1} F)(A)| \rightarrow F(A) \rightarrow \mathrm{T}_n^{k+1} F(A).$  
\end{prop}

This gives an alternative approach to understanding Theorem \ref{t:agreeatA}, since as $k$ goes to infinity, Proposition \ref{prop:alldim} suggests that there is a homotopy fiber sequence $|\perp_{n+1}^{*+1} F(A)| \to F(A) \to P_nF(A)$, i.e., the result of Theorem \ref{t:agreeatA}.

The proof of Proposition \ref{prop:alldim} is by induction, taking Lemma \ref{l:Tn} as the base case with $k=0$.

\subsection{Lemmas necessary for proof}

Given a square of spectra,
\[
\xymatrix{
A\ar[r]\ar[d] & B \ar[d] \\
C\ar[r] & D\\
}
\]

we have an associated diagram which includes the fibers. Then, let
\[
\begin{array}{rclp{.5 cm}rcl}
h_1 &:=& \mathrm{hofib} (C \rightarrow D)& & 
       h_4 &:=& \mathrm{hofib} (A \rightarrow C)\\
h_2 &:=& \mathrm{hofib} (A \rightarrow B) & &
      h_5 &:=& \mathrm{hofib} (B \rightarrow D)\\ 
h_3 &:=& \mathrm{hofib} (h_4 \rightarrow h_5)  & &  
      h_6 &:=& \mathrm{hofib} (A \rightarrow D)\\
&=& \mathrm{hofib} (h_2 \rightarrow h_1)\\
\end{array}
\]

\begin{lem}\label{lem:yoga} With the definitions given above, the following is a cartesian square:
\[
\xymatrix{h_3 \ar[r]\ar[d] & h_4 \ar[d] \\ h_2 \ar[r] & h_6}
\]
\end{lem}

\begin{proof}[Proof of Lemma \ref{lem:yoga}]
We can construct an associated diagram of fibers. It is the leftmost in the homotopy fiber sequence shown in Figure \ref{fig:hofib-9pt-square}. 

\begin{figure}[h!]
\[
\scalebox{.85}{$
\begin{array}{c}
\xymatrix{
h_3 \ar[r]\ar[d] & h_4 \ar[r]\ar[d] & h_5 \ar@{=}[d]\\
h_2 \ar[r]\ar[d] & h_6 \ar[r] \ar[d]& h_5 \ar[d]\\
h_1 \ar@{=}[r] & h_1 \ar[r] & \pt\\
}
\end{array}
\ra
\begin{array}{c}
\xymatrix{
h_3 \ar[r]\ar[d] & h_4 \ar[r]\ar[d] & h_5 \ar[d]\\
h_2 \ar[r]\ar[d] & A \ar[r] \ar[d]& B \ar[d]\\
h_1 \ar[r] & C \ar[r] & D\\
}
\end{array}
\ra
\begin{array}{c}
\xymatrix{ 
\pt \ar[r]\ar[d] & \pt \ar[r]\ar[d] & \pt \ar[d]\\
\pt \ar[r]\ar[d] & D \ar@{=}[r] \ar@{=}[d]& D \ar@{=}[d]\\
\pt \ar[r] & D \ar@{=}[r] & D\\
}
\end{array}
$}
\]
\caption{Homotopy fiber sequence of diagrams}
\label{fig:hofib-9pt-square}
\end{figure}
\medskip

\noindent 
Since $h_3 = \hofib (h_4 \ra h_5)$, then $h_3 \ra h_4 \ra h_5$ is a homotopy fiber sequence. This says that $\Omega h_5 = \hofib (h_3 \ra h_4)$. Figure \ref{fig:h_2h_6h_5} illustrates that $h_2 \ra h_6 \ra h_5$ is also a homotopy fiber sequence, as homotopy-fiber-taking commutes. 

\begin{figure}[h]
\[
\xymatrix{
h_2 \ar[r]\ar[d] & h_6 \ar[r] \ar[d]& h_5 \ar[d]\\
h_2 \ar[r]\ar[d] & A \ar[r] \ar[d]& B \ar[d]\\
\pt \ar[r]& D \ar@{=}[r]& D\\
}
\]
\caption{$h_2 \ra h_6 \ra h_5$ is a homotopy fiber sequence}
\label{fig:h_2h_6h_5}
\end{figure}

Thus,  $\Omega h_5 \ra h_2 \ra h_6$ is also a homotopy fiber sequence. We can extend the square to the left with its homotopy fibers, as in Figure \ref{fig:eqfibers-cartesian}. As they are equivalent and we're in spectra (which cures the basepoint troubles), we can conclude that the square is a homotopy pullback.

\begin{figure}[h]
\[
\xymatrix{
\Omega h_5 \ar[r]\ar@{=}[d] &h_3 \ar[r]\ar[d] & h_4 \ar[d] \\ 
\Omega h_5 \ar[r] & h_2 \ar[r] & h_6}
\]
\caption{Square and its fibers}
\label{fig:eqfibers-cartesian}
\end{figure}

\end{proof}

\begin{lem}\label{lem:sk_k-hocolimcube} 
$|\mathrm{sk}_k \perp_{n+1}^{\ast +1} F| \simeq \mathrm{hocolim}_{\mathscr{P}_0^{op}([k])}( \perp_{n+1}^{\ast +1} F).$
\end{lem}

\begin{proof}[Proof of Lemma \ref{lem:sk_k-hocolimcube}]  Let $\# S$ be the cardinality of the set $S$.
From \cite{Dev}, we know the natural functor that maps $\mathscr{P}_0 ([n])$ to $\Delta_{\leq n}$ via $S \mapsto [\# S-1]$ is homotopy left cofinal. Its dual which maps $\mathscr{P}_0([n])^{op}$ to $\Delta^{op}_{\leq n}$  is then homotopy right cofinal. That is, we have that $\hocolim_{\Delta^{op}_{\leq n}} X_\pt \simeq \hocolim_{\mathscr{P}_0([n])^{op}} X_\pt$. 

Given that $|\sk_n X_\pt| \simeq \hocolim_{\Delta^{op}_{\leq n}} X_\pt$, with $X_\pt = \perp_{n+1}^{\ast +1} F$, we have the statement of our lemma.

\end{proof}

\begin{lem}\label{lem:skelperp} For $F:\Cf\rightarrow \s$, 
$$
|\mathrm{sk}_k \perp^{\ast +1}_{n+1} F| \simeq \mathrm{hocolim}(|\mathrm{sk}_{k-1} \perp^{\ast +1}_{n+1}F|\leftarrow |\mathrm{sk}_{k-1} \perp^{\ast +1}_{n+1} (\perp_{n+1} F)| \rightarrow  \perp_{n+1} F).
$$
\end{lem}
\smallskip

\if false
\begin{lem}\label{lem:skelperp}
The homotopy pushout of the following diagram is equivalent to $||\mathrm{sk}_k \perp^\ast_{n+1} F||$:
\[
\xymatrix{
||\mathrm{sk}_{k-1} \perp^\ast_{n+1} (\perp_{n+1} F)|| \ar[r]\ar[d] & \perp_{n+1} F\\
||\mathrm{sk}_{k-1} \perp^\ast_{n+1}F)||
}
\]
\end{lem}
\fi

Note that Lemma \ref{lem:skelperp} is a simple application of the covering lemma for hocolim cubes (given on p. 299 of \cite{G2} as the dual situation to  Proposition 0.2 in the same paper) to the $(k+1)$-cube that at $U \in \mathscr{P}([k])$ is $\perp^{|[k]-U|}_{n+1}F$, whose co-punctured (with the final element of the cube removed) hocolim is equivalent to $|\mathrm{sk}_k \perp^{\ast +1}_{n+1} F|$ by Lemma \ref{lem:sk_k-hocolimcube}.

\subsection{General Case of Induction} 
Assume that $|\mathrm{sk}_{k-1}( \perp^{\ast +1}_{n+1}F(A))| \rightarrow F(A) \rightarrow \mathrm{T}_n^{k} F(A)$ is a homotopy fiber sequence. Then construct 
the diagram of homotopy fiber sequences in Figure \ref{fig:perp-hof-k}.

\begin{figure}[h!]
\[
\xymatrix{
|\mathrm{sk}_{k-1} \perp^{\ast +1}_{n+1} (\perp_{n+1} F(A))| \ar[r]\ar[d]&\perp_{n+1} F(A)\ar[r]\ar[d] & \perp_{n+1} \mathrm{T}_n^{k} F(A)\ar[d]\\
|\mathrm{sk}_{k-1} \perp^{\ast +1}_{n+1} F(A) | \ar[d]\ar[r]&F(A) \ar[r]\ar[d]        & \mathrm{T}_n^{k} F(A)\ar[d]\\
|\mathrm{sk}_{k-1} \perp^{\ast +1}_{n+1} \mathrm{T}_n F(A) | \ar[r]&\mathrm{T}_nF(A)\ar[r] & \mathrm{T}_n^{k+1} F(A)\\
}
\]
\caption{$\perp_{n+1}$ and iterated fiber diagram}
\label{fig:perp-hof-k}
\end{figure}

Since homotopy limit constructions commute, $\perp_{n+1} \mathrm{T}_n^{k} F(A) \simeq \mathrm{T}_n^{k} \perp_{n+1} F(A)$, and the top line is also a homotopy fiber sequence. 

Application of Lemma \ref{lem:yoga} to the diagram of Figure  \ref{fig:perp-hof-k}  lets us conclude that Figure \ref{fig:cartsq-k} is a cartesian square. Since we're in spectra, we also know that it is cocartesian.

\begin{figure}[h]
\[
\xymatrix{
|\mathrm{sk}_{k-1} \perp^{\ast +1}_{n+1} (\perp_{n+1} F)(A)| \ar[r]\ar[d] & \perp_{n+1} F(A)\ar[d]\\
|\mathrm{sk}_{k-1} \perp^{\ast +1}_{n+1}F(A)| \ar[r]& \mathrm{hofib}(F(A) \rightarrow \mathrm{T}_n^{k+1} F(A))\\
}
\]
\caption{Desired cartesian square for general $k$}
\label{fig:cartsq-k}
\end{figure}

We now apply Lemma \ref{lem:skelperp} to conclude that $\mathrm{hofib}(F(A) \rightarrow \mathrm{T}_n^{k+1} F(A)) \simeq |\mathrm{sk}_k \perp^{\ast +1}_{n+1} F(A)|$. \qed


\end{document}

@article{Dev,
  title={{The topology of spaces of knots: cosimplicial models}},
  author={Sinha, D.P.},
  journal={American Journal of Mathematics},
  volume={131},
  number={4},
  pages={945--980},
  year={2009},
  publisher={The Johns Hopkins University Press}
}

For example, $cr_2F(X,Y)$ is the total fiber of the  $2$-cube
\[\xymatrix{F(X\amalg _A Y)\ar[r]\ar[d]&F(X\amalg _A B)\ar[d]\\
F(B\amalg _AY)\ar[r]&F(B\amalg _AB).\\}\]

\subsection{Cotriples from Cross Effects}

The cross effect functors play an essential role in the construction of Taylor towers in \cite{RB}.  When ${\sc C}$ is a pointed category and ${\sc A}$ is an abelian category, it is straightforward
to show that $(\Delta^*, cr_n)$ is an adjoint pair of functors between the categories 
$\WR({\sc C}, {\sc A})$ and $\WR({\sc C}^{\times n}, {\sc A})$.   (Here $\Delta ^*$ denotes 
precomposition with the diagonal functor.)   As a consequence, $\perp_n=\Delta^*\circ cr_n$ 
is a cotriple which was used to construct the $(n-1)st$ term in the Taylor tower of a functor $F$.  

In trying to replicate this process for functors from $\Cf$ to $\D$, one encounters an obstruction
almost immediately.   The functors $\Delta ^*$ and $cr_n$ no longer form a strict adjoint pair.   However,  Goodwillie shows that they form an adjoint pair up to weak equivalence in a topological setting \cite{G3}.  Inspired by his result, we examine to what extent the expected adjunction of \cite{RB} is an adjoint pair up to weak equivalence in this setting.
When trying to identify a strict adjunction with $\Delta ^*$ and $cr_n$, one runs into problems with $tH$.   When $\sc C$ is pointed  and $\sc A$ is abelian, $tH$
splits off naturally from $H$ as a direct summand.   This splitting becomes a natural isomorphism
when $H$ is weakly $n$-reduced, and the existence of this splitting and natural isomorphism enabled us to show that $(\Delta ^*, cr_n)$ form an adjoint pair in the abelian case \cite{RB}.   To understand better where this isomorphism arises in this process, one needs to recognize that the
adjoint pair $(\Delta ^*, cr_n)$ is  the composition of two adjoint pairs,  one of which
involves $t$.   In particular, the left adjoint of $t$ is $U$, the forgetful functor from weakly $n$-reduced
functors to functors of $n$ variables.   

In the 
present setting, we can only assume the existence of a weak equivalence from $tH$ to $H$
when $H$ is weakly $n$-reduced.   This means that we can 
only prove that $U$ and $t$  form an adjoint pair after formally inverting this weak equivalence.  To obtain a strict adjunction, we replace the category $\WR(\Cfn n, \D)$ with the full subcategory $\WR^+(\Cfn n, \D)$ whose objects are functors of the form $tH$.  We construct a natural transformation $\Delta_H$ which is a natural section to both of the natural transformations $\gamma_{tH}:ttH\rightarrow tH$ and $t\gamma_H: ttH\rightarrow tH$.  The natural transformation $\gamma_{tH}$ is the natural transformation of Definition 3.2, applied to the functor $tH$.  To define the natural transformation $t\gamma_H$, we note that the natural transformation $\gamma_H$ produces a map of cubes $\gamma_H: tH(T)\to H(T)$ (one map for each $T\in P({\bf n})$). This in turn produces a map of total fibers of cubes; this map is $t\gamma_H$.

Our first step is to construct the section, $\Delta$.  We will use this section to prove that the pair $(U^+, t^+)$ obtained by restricting $(U,t)$ to $\WR^+(\Cfn n,\D)$ is an adjoint pair.   We will then discuss the second adjoint pair alluded to earlier, and show 
that it can be coupled with $(U^+,t^+)$ to yield $\perp _n$ as a cotriple.

\begin{lem} \label{l:section} For any functor $H\in {\Fun}({\Cfn n}, \D)$, 
\begin{enumerate}
\item the natural transformation $\gamma_{tH}: t(tH) \to tH$ has a section $\Delta_H$;
\item  the section $\Delta$ is functorial in $H$;  and
\item  the natural transformation $\Delta_H$ is also a section to $t\gamma_H$.
\end{enumerate}
\end{lem}

\begin{proof} For (1), let ${\bf X}=(X_1, X_2, \dots, X_n)$ be an $n$-tuple of objects in $\Cf$.  Much of this proof involves using the properties of homotopy limits and fibers summarized in Section 2 to show
that $ttH({\bf X})$ is isomorphic to 
\[\operatorname{hofib}\left ( \underset  {(S,T)\in P({\bf n})\times P({\bf n})} {\holim} H^{\bf X}_B(S\cup T)
\rightarrow \underset {(S,T)\in P({\bf n})\times P({\bf n})-\emptyset\times\emptyset }{\holim}H^{\bf X}_B(S\cup T)\right ).\] 
Once we have done this, we will use this reformulation of $ttH({\bf X})$ to construct the desired 
section.  

To start, we note that  the fact that the $n$-cube $H^{\bf X}_B$ has an initial object, 
$H_B^{\bf X}(\emptyset )$, implies  that we can replace $H^{\bf X}_B(\emptyset)$ with 
$\holim _{P({\bf n})}H^{\bf X}_B$ in the definition of the total fiber of $H^{\bf X}_B$ to obtain
$$
tH({\bf X}):=\tfiber H^{\bf X}_B\cong \operatorname{hofib}\left (\underset {P({\bf n})}{\holim} H^{\bf X}_B \rightarrow \underset {P_0({\bf n})}{\holim} H^{\bf X}_B\right ),
$$
and, similarly,
\begin{equation}\label{ttH} ttH({\bf X}):=\tfiber (tH)^{\bf X}_B\cong \operatorname{hofib}\left (\underset {P({\bf n})}{\holim} (tH)^{\bf X}_B \rightarrow \underset {P_0({\bf n})}{\holim} (tH)^{\bf X}_B\right ).
\end{equation}

For each $S\in P({\bf n})$, $tH^{\bf X}_B(S)$ is the total fiber of the $n$-cube $H^{{\bf X}(S)}_B$ 
where ${\bf X}(S)=(A_1, A_2, \dots, A_n)$ with 
\[ A_i = \begin{cases} X_i & \text{if\ \ } i\notin S \\ B & \text{if \ \ } i \in S.\end{cases}\]
Moreover, it is straightforward to see that for $T\in P({\bf n})$, 
$$
H^{{\bf X}(S)}_B(T)=H^{\bf X}_B(S\cup T).
$$
Using these observations and applying Lemmas \ref{l:holim} and \ref{iteratedfiber} to (\ref{ttH}), we see that 
\begin{align*}
ttH({\bf X})&=\operatorname{hofib}\left (\underset{S\in P({\bf n})}{\holim} tH^{\bf X}_B(S)\rightarrow
\underset{S\in P_0({\bf n})}{\holim} tH^{\bf X}_B(S)\right )\\
&\cong \operatorname{hofib} \left (\vcenter{\xymatrix{\underset{S\in P({\bf n})}{\holim}\left ( \operatorname{hofib}\left ( \underset{T\in P({\bf n})}{\holim}H^{\bf X}_B(S\cup T)\rightarrow \underset{T\in P_0({\bf n})}{\holim}H^{\bf X}_B(S\cup T)\right )\right)\ar[d]\\
\underset{S\in P_0({\bf n})}{\holim}\left (\operatorname{hofib}\left (\underset {T\in P({\bf n})}{\holim}H^{\bf X}_B(S\cup T)\rightarrow \underset{T\in P_0({\bf n})}{\holim}H^{\bf X}_B(S\cup T)\right )\right )}}\right )\\
&\cong {\tfiber}\left (\vcenter{\xymatrix {\underset{(S, T)\in P({\bf n})\times P({\bf n})}{\holim}H_B^{\bf X}(S\cup T)\ar[r]\ar[d]&\underset{(S,T)\in P({\bf n})\times P_0({\bf n})}{\holim} H^{\bf X}_B(S\cup T)\ar[d]\\
\underset{(S,T)\in P_0({\bf n})\times P({\bf n})}{\holim}H^{\bf X}_B(S\cup T)\ar[r]&\underset{(S,T)\in P_0({\bf n})\times P_0({\bf n})}{\holim}H^{\bf X}_B(S\cup T)}}
\right ).
\end{align*}
Hence, 
$ttH({\bf X})$ is isomorphic to the homotopy fiber of the map from 
$$
\underset {(S,T)\in P({\bf n})\times P({\bf n})}{\holim}H^{\bf X}_B (S\cup T)
$$
to 
\begin{equation}\label{whatever}{\holim}\left ( \begin{matrix} & & \underset {(S,T)\in P({\bf n})\times P_0({\bf n})}{\holim}H^{\bf X}_B(S\cup T)\\
&&\downarrow\\
\underset{(S,T)\in P_0({\bf n})\times P({\bf n})}{\holim}H^{\bf X}_B(S\cup T)&\rightarrow &
\underset {(S,T)\in P_0({\bf n})\times P_0({\bf n})}{\holim}H^{\bf X}_B(S\cup T)\\
\end{matrix}\right ).
\end{equation}
But, by Lemma \ref{l:holim}.5, (4) is isomorphic to 
$$
\underset {(S,T)\in P({\bf n})\times P({\bf n})-\emptyset\times \emptyset}{\holim} H^{\bf X}_B(S\cup T).
$$
Hence, it follows that $ttH({\bf X})$ is isomorphic to 
\[\operatorname{hofib}\left( \underset  {(S,T)\in P({\bf n})\times P({\bf n})} {\holim} H^{\bf X}_B(S\cup T)
\rightarrow \underset {(S,T)\in P({\bf n})\times P({\bf n})-\emptyset\times\emptyset }{\holim}H^{\bf X}_B(S\cup T)\right ).\] 

Thus, to finish the proof, it suffices to find dashed maps that make the following diagram commute:
\[\xymatrix{
\underset{P({\bf n})}{\holim} H^{\bf X}_B \ar[r] \ar@{-->}[d]  &  
\underset{P_0({\bf n})}{\holim} H^{\bf X}_B  \ar@{-->}[d] \\
\underset{P({\bf n})\times P({\bf n})}{\holim} H^{\bf X}_B \ar[r]  &  
\underset{P({\bf n})\times P({\bf n}) - \emptyset\times \emptyset}{\holim} H^{\bf X}_B.
}\]
Since $P({\bf n})\times P({\bf n})\cong P({\bf n}\amalg {\bf n})$, the fold map $+:{\bf n}\amalg {\bf n} \to {\bf n}$ induces $P({\bf n})\times P({\bf n})\to P({\bf n})$ taking $(S,T) \mapsto S\cup T$.    By 
Lemma \ref{l:holim}, this yields a morphism $\Delta _H({\bf X}): tH({\bf X})\rightarrow ttH({\bf X})$
that is natural with respect to $H$ and ${\bf X}$, which establishes statement (2) of the Lemma.   Moreover, it is easy to see that $\gamma _{tH}$
is induced by the inclusion
$$
P({\bf n})\buildrel {\iota}\over{\longrightarrow}P({\bf n})\times P({\bf n}), S\mapsto (S, \emptyset).
$$
Since the composite
$$
P({\bf n})\buildrel {\iota}\over{\longrightarrow} P({\bf n})\times P({\bf n})\buildrel +\over \longrightarrow P({\bf n})
$$
is the identity and the constructions are natural, it follows that the induced composition
$$
tH\buildrel {\Delta _H}\over \longrightarrow ttH\buildrel {\gamma_{tH}}\over \longrightarrow tH
$$
is the identity map, and so $\Delta_H$ is a section to $\gamma_{tH}$, proving (1).

For the proof of statement (3), we note that $t\gamma_H: t(tH)\rightarrow t(H)$ is induced by the map $\mu: P({\bf n})\rightarrow P({\bf n})\times P({\bf n})$ that sends $S$ to $(\emptyset, S)$.   Since $ +\circ\mu$ is also the identity, the proof of (3) can be obtained by replacing $\iota$ by $\mu$ in the proof of (1).    


\end{proof}

With this lemma, we can establish the desired strict adjunction.   To do so, we replace $\WR(\Cfn n,\D)$ with the full subcategory of $\Fun(\Cfn n, \D)$ generated by   functors of the form $tG$  (where $G:\Cfn n\rightarrow \D$).   We use $\WR^+(\Cfn n, \D)$ to denote this subcategory, and $\Hom_{\WR^+}$ to indicate morphism sets in this subcategory.  Hence for $G, G':\Cfn n\rightarrow \D$, 
\[\Hom_{\WR ^+}(tG,tG')=\HomFun(tG, tG').  
\]    Since every weakly reduced functor $H$ is weakly equivalent to $tH$, every weakly reduced functor is represented by an object in this subcategory.     

This allows us to identify the adjoint pair to replace $(t, U)$.

\begin{defn}
Let  $t^+: {\Fun}({\Cfn n}, \D)\to \WR^+(\Cfn n, \D)$ be the functor given by $t^+(H) = tH$.\end{defn}

\begin{thm} The functors 
 \[\xymatrix{ {\Fun}({\Cfn n}, \D) \ar@/^/[r]^{t^+} &  {\WR^+}({\Cfn n}, \D)\ar@/^/[l]^{U^+}}\]
 are an adjoint pair of functors, with $U^+$ being the left adjoint.  (Here, $U^+$ denotes the 
 forgetful functor.)
\end{thm}

\begin{proof}
 We describe  functions 
\[\xymatrix {{\HomFun (U^+(tG), H)}\ar@/^/[r]& {\Hom_{\WR^+} (tG, t^+H)}\ar@/^/[l]},\]
which yield the desired isomorphism.  

Let $\sigma$ be a natural transformation from $tG=U^+(tG)$ to $H$.   We associate $\sigma$ to the natural transformation $t\sigma\circ \Delta_G$ in $\Hom_{\WR^+} (tG, t^+H)$.

On the other hand, suppose that $\tau:tG \to tH$ is a natural transformation of  functors.  Associate $\tau$ to the natural transformation $\gamma_H\circ\tau$ in $\HomFun(U^+(tG), H)$.

These are inverse associations. The natural transformation $\sigma:tG\to H$ is associated first to  $ t{\sigma}\circ\Delta_G$ which is in turn sent to  
  $\gamma_H\circ t{\sigma}\circ\Delta_G$
in $\HomFun(U^+(tG), H)$.  However, from the commuting diagram
\[\xymatrix{ tG\ar[r]^{\Delta_G}&ttG\ar[r] ^{t\sigma}\ar[d]_{\gamma_{tG}} & tH\ar[d]^{\gamma_H}\\ &tG\ar[r]_{\sigma}& H,}\] 
we see that $\gamma_H\circ t{\sigma}\circ\Delta_G= \sigma\circ\gamma_{tG}\circ\Delta_G$.  Since $\Delta_G$ is a section to $\gamma_{tG}$ by Lemma \ref{l:section} (1), this is just $\sigma$.

The natural transformation $\tau: tG\to tH$ is associated first to $\gamma_H\circ \tau$ which is sent in turn to $t(\gamma_H\tau)\circ \Delta_G$ in $\Hom_{\WR^+}(tG, tH)$.  Consider the commutative diagram
\[\xymatrix{ tG\ar[r]^{\tau}\ar[d]_{\Delta_G}&tH\ar[d]^{\Delta_H}& \\
ttG\ar[r]_{t\tau}&ttH\ar[r]_{t\gamma_H}&tH.}\]
From this diagram and Lemma \ref{l:section} (3), we have
\begin{align*}
t(\gamma_H\tau)\circ \Delta_G&=t\gamma_H\circ t\tau\circ \Delta_G\\
&=t\gamma_H\circ\Delta_H\circ \tau\\
&={\rm id}\circ \tau\\
&=\tau.
\end{align*}
\end{proof}

\begin{rem}  This proof can be adapted to show that there is a kind of adjunction up to homotopy between the categories $\WR(\Cfn n, \D)$ (as opposed to $WR^+(\Cfn n,\D)$) and $\Fun(\Cfn n,\D)$.  The left adjoint in that case is again the forgetful functor, and the right adjoint is the functor sending $H$ to $tH$.  The proof is essentially the same, but one must use a general weakly reduced functor $G$ in place of $tG$ and formally invert the weak equivalence $\gamma_G: tG\to G$ instead of using $\Delta_G$.  It is the fact that one must include this formal inverse which makes us call this an adjunction ``up to homotopy".
\end{rem}

\begin{rem}\label{perp} \begin{enumerate} \item  The $n$th cross effect of a functor $F:\Cf \to \D $ is given by the composite 
\[cr_nF=({t}\circ \sqcup _n) (F).\]
\item  We denote the diagonal of $cr_nF$ by $\perp_nF$ and note that   $\perp_nF = (\Delta^*\circ U^+)\circ({t^+}\circ\sqcup _n)(F)$.
\end{enumerate}
\end{rem}
Since $\perp _n$ is the composition of the  left adjoint $\Delta ^*\circ U^+$ with the right adjoint
${t}^+\circ\sqcup _n$, it forms part of a cotriple (for more on cotriples and adjoint pairs, see section 8.6 of \cite{We}).   In particular, the counit for the adjunction produced by the pair $(\Delta ^*\circ U^+, {t}^+\circ\sqcup _n)$
yields a natural transformation $ \epsilon:\perp _n\rightarrow {\rm id}$.   And, a natural 
transformation $\delta : \perp _n\rightarrow \perp _n\perp _n$ is defined by $\Delta ^*\circ U^+(\eta 
_{{t}^+\circ\sqcup _n})$ where $\eta$ is a unit for the adjunction.   This gives us the following.

\begin{thm} The functor and natural transformations $$(\perp_n,
\delta : \perp _n\rightarrow \perp _n\perp _n,  \epsilon:\perp _n\rightarrow {\rm id})$$  form a cotriple on the category of functors $\Fun(\Cf,\D)$.
\end{thm}

\subsection{Cross effects and degree $n$ functors}
We return for a moment to the setting of \cite{RB}.   
Let ${\mathcal C}$ be a pointed category with initial/final object $\ast$,  ${\mathcal A}$ be an abelian category,  and $F:{\mathcal C}\rightarrow {\mathcal A}$. 
In this context, the $n$th cross effect of $F$ evaluated at the $n$-tuple of objects ${\bf X}$ is 
naturally isomorphic to 
$$\tfiber F((\amalg _n)^{\bf X}_{\ast})$$ where $(\amalg _n)^{\bf X}_{\ast}$ is as defined in Example \ref{e:chi}.
The pointed category ${\mathcal C}$ can be viewed as a category of the form $\Cf$ by using the 
identity morphism
${\rm id}:\ast \rightarrow \ast$ in place of $f$.  In building the $n$-cube $(\amalg _n)^{\bf X}_{\ast}$, we use the fact that $\ast$ is initial to form the coproducts and the fact that $\ast $ is final to form
the morphisms.   Recognizing this, we generalize the notion of cross effect to functors of 
an arbitrary $\Cf$  as follows.  
\begin{defn}
Let $F:\Cf \rightarrow \D$ and ${\bf X}$ be an $n$-tuple of objects in $\Cf$.  Then $cr_nF:\Cfn n\rightarrow \D$ is the functor given by 
$$cr_nF({\bf X}):= \tfiber(F((\amalg_n)_B^{\bf X})).$$
\end{defn}